\documentclass{amsart}

\usepackage{mathtools}              
\usepackage{amssymb, amsthm, amsfonts} 
\usepackage[mathscr]{eucal}         
\usepackage{relsize}                
\usepackage{accents}                
\usepackage{bm}
\usepackage{verbatim}               
\usepackage{comment}                
\usepackage{enumitem}               
\usepackage[normalem]{ulem}         
\usepackage{bbm}                    
\usepackage{dsfont}                 
\usepackage{upgreek}                
\usepackage{kotex}                  

\usepackage{tikz}                   
\usepackage{subcaption}             
\graphicspath{{./images/}}          

\usepackage[colorlinks=true, citecolor=blue]{hyperref} 
\usepackage[backgroundcolor=black!5,linecolor=black]{todonotes} 

\usepackage{setspace}
\theoremstyle{plain}
\newtheorem{thm}{Theorem}[section]

\newtheorem{prop}[thm]{Proposition}
\newtheorem{cor}[thm]{Corollary}
\newtheorem{lem}[thm]{Lemma}

\theoremstyle{definition}
\newtheorem{defn}[thm]{Definition}

\theoremstyle{remark}
\newtheorem{rem}[thm]{Remark}

\numberwithin{equation}{section}

\newcommand{\supp}{\operatorname{supp}}

\newcommand{\ar}{\mathsf{ar}}
\DeclareRobustCommand{\aracc}[1]{{\accentset{\raisebox{0.17ex}{\rule{0.40em}{0.6pt}}}{#1}}}
\DeclareRobustCommand{\rdacc}[1]{\accentset{\bm{\sim}}{#1}}

\renewcommand{\Re}{\operatorname{Re}}
\renewcommand{\Im}{\operatorname{Im}}

\newcommand{\conv}{%
    \mathop{\vcenter{\hbox{\scalebox{1.5}{$\ast$}}}}\displaylimits
}
\newcommand{\bconv}{{\mathlarger{\mathlarger {\mathlarger {\conv}}}}}

\DeclarePairedDelimiter\floor{\lfloor}{\rfloor}
\DeclarePairedDelimiter\ceil{\lceil}{\rceil}

\begin{document}

\author[ Sanghyuk  Lee]{Sanghyuk Lee} \address[Sanghyuk Lee]{Department of Mathematical Sciences and RIM, Seoul National University, Seoul 08826, Republic of  Korea} \email{shklee@snu.ac.kr}
\author[Sungchul Lee]{Sungchul Lee} \address[Sungchul Lee]{Department of Mathematical Sciences, Seoul National University, Seoul 08826, Republic of  Korea} \email{lsngchl127@snu.ac.kr}

\makeatletter
\@namedef{subjclassname@2020}{\textup{2020} Mathematics Subject Classification}
\makeatother
\title[Convolution bound]{Sharpness of  convolution bounds for measures}

\begin{abstract}In this paper, we determine the optimal universal \(L^p\)-\(L^q\) type sets
for convolution operators \(f\mapsto \mu*f\) associated with fractal measures $\mu\in \mathcal P_{\alpha,\beta}(\mathbb R^d)$, which denotes the class of compactly
supported Borel probability measures satisfying the \(\alpha\)-Frostman condition
\[
    \mu(B(x,\rho)) \lesssim \rho^\alpha,
    \qquad x\in\mathbb R^d,\quad 0<\rho<1,
\]
and the \(\beta/2\)-Fourier decay condition
\[
    |\widehat{\mu}(\xi)| \lesssim |\xi|^{-\beta/2},
     \qquad \xi\in\mathbb R^d. 
\]
More precisely, we characterize the largest \(L^p\)-\(L^q\) region that is forced solely by the Frostman and Fourier decay assumptions throughout the
full admissible range of \((\alpha,\beta)\), with distinct optimal regions in the geometric and nongeometric regimes. We prove optimality in the worst-case
sense over \(\mathcal P_{\alpha,\beta}(\mathbb R^d)\) by constructing, for each admissible pair \((\alpha,\beta)\), a single extremal measure whose support has
the smallest Hausdorff dimension allowed by the hypotheses. Moreover, variants of the same constructions also yield a single-measure sharpness theorem
for the \(L^2\) Fourier restriction theorem of Mockenhaupt--Mitsis--Bak--Seeger: in every dimension and in both the geometric and nongeometric regimes, we
construct a measure in \(\mathcal P_{\alpha,\beta}(\mathbb R^d)\) for which the Mockenhaupt--Mitsis--Bak--Seeger threshold exponent is sharp. 
\end{abstract}

\maketitle

\section{Introduction}

Let $\mu$ be a finite Borel measure on $\mathbb{R}^d$. Then  convolution with \(\mu\) is bounded on \(L^p\) for \(1\le p\le\infty\). 
If $\mu$ is absolutely continuous with respect to Lebesgue measure and its density belongs to $L^r(\mathbb{R}^d)$ for some $r > 1$, then Young's convolution inequality implies that $\| \mu\ast f\|_{q} \le C \|f\|_{p}$ for some $q=q(p)>p$ when $1<p<\infty$.
In this case, the measure $\mu$ is said to be \emph{$L^p$-improving}.
Even when $\mu$ is singular, it may still enjoy an $L^p$-improving property.

The following classical result is due to Strichartz~\cite{Strichartz} and Littman~\cite{Littman} (see also~\cite[Theorem 2]{Oberlin}).
It provides a complete characterization of  the $L^p$-improving property of the surface measure on the unit sphere $\mathbb S^{d-1}$.

\begin{thm}
\label{stri-litt}
Let $\sigma$ denote the surface measure on $\mathbb S^{d-1}$ in $\mathbb R^d$.
Then there is a constant $C$ such that
\[
\|\sigma\ast f\|_{q} \le C \|f\|_{p}
\]
holds for all $f\in L^p(\mathbb R^d)$ if and only if $(1/p,1/q)$ lies in the closed triangle with vertices $(0,0),$ $(1,1),$ and $(\frac{d}{d+1}, \frac{1}{d+1}).$
\end{thm}

The necessity of this theorem is easy to verify.  In fact, it follows from translation invariance and  testing the inequality against characteristic functions of small balls with radius shrinking to zero (see, for example, \cite{Oberlin}).
The proof of the sufficiency, on the other hand, relies on the decay of the Fourier transform of $\sigma$, together with the fact that $\sigma$ is supported on a $(d-1)$-dimensional manifold.

\subsubsection*{Convolution with general measures}
The proof of the sufficiency naturally leads to a generalization obtained by considering compactly supported Borel probability measures $\mu$ on $\mathbb{R}^d$ satisfying the following two quantitative conditions:
\begin{itemize}
\item \emph{$\alpha$-Frostman condition}:
\begin{equation}
\label{frostman}
\mu\bigl(B(x,\rho)\bigr) \lesssim \rho^\alpha,
\quad \forall (x,\rho)\in \mathbb R^d\times (0,1),
\end{equation}
where $B(x,\rho)$ denotes the open ball centered at $x$ with radius $\rho$.
We say that $\mu$ is $\alpha$-Frostman if \eqref{frostman} holds.

\item \emph{$\beta/2$-Fourier decay}:
\begin{equation}
\label{fourier}
|\widehat{\mu}(\xi)| \lesssim |\xi|^{-\beta/2},
\quad \forall \xi\in\mathbb{R}^d.
\end{equation}
In this case, we say that $\mu$ has $\beta/2$-Fourier decay.
\end{itemize}

In the above and what follows, for two nonnegative quantities $X$ and $Y$, we write $X \lesssim Y$ if $X \le CY$ for some constant $C>0$.
The dependence of the implicit constant $C$ on parameters is indicated by subscripts.
For instance, $X \lesssim_{\varepsilon} Y$ means that the implicit constant $C$ may depend on $\varepsilon$.

We exclude the cases $\alpha=0$ and $\alpha=d$, which are less interesting.
In the former case, as illustrated by the Dirac delta measure, no $L^p$-improving property can generally be expected.
In the latter case, since $\mu$ is absolutely continuous with respect to the Lebesgue measure, the Radon--Nikodym theorem implies that $d\mu = h\,dx$ for some $h \in L^1(\mathbb{R}^d)$.
Moreover, the Lebesgue differentiation theorem together with \eqref{frostman} for $\alpha=d$ shows that $h \in L^\infty(\mathbb{R}^d)$.
Thus, the convolution operator $f \mapsto \mu * f$ is bounded from $L^p$ to $L^q$ if and only if $1 \le p \le q \le \infty$.

Although there exist measures with no polynomial Fourier decay that still enjoy \(L^p\)-improving properties \cite{Oberlin82, Christ85}, we exclude the endpoint case \(\beta=0\) for convenience and focus on the roles of the two conditions \eqref{frostman} and \eqref{fourier} in the estimate \eqref{conv}.
If \(\beta>d\), then Plancherel's theorem implies that any measure \(\mu\) satisfying \eqref{fourier} is not singular but has \(L^2\) density.
Furthermore, since \eqref{fourier} implies \eqref{frostman} with \(\alpha=\beta/2\) (see, for example, \cite[p.~95]{Mitsis}), 
 the Frostman assumption with exponent \(\alpha<\beta/2\) becomes redundant. To avoid  such redundancy, we restrict our attention to parameters $\alpha$ and $\beta$ satisfying
\begin{equation}
\label{ab}
\alpha,\beta\in(0,d), \quad \beta/2 \le \alpha.
\end{equation}

\begin{defn}
Let \(\mathcal P(\mathbb R^d)\) denote the set of compactly supported Borel
probability measures on \(\mathbb R^d\). For \(\alpha,\beta\in(0,d)\), let
\(\mathcal P_{\alpha,\beta}(\mathbb R^d)\) denote the set of measures
\(\mu\in\mathcal P(\mathbb R^d)\) satisfying \eqref{frostman} and \eqref{fourier}.
In the main convolution and restriction results we shall focus on the
nonredundant range \eqref{ab}. Additionally, we set
\[
p_{\alpha,\beta} = \frac{2(d-\alpha)+\beta}{(d-\alpha)+\beta}, \quad r_\beta= \frac{2d}{d+\beta}.
\]
Let
$\Delta_{\alpha,\beta}$ denote the closed triangle with vertices
$(0,0)$, $(1,1)$, and 
\[
C_{\alpha,\beta}
:=
\bigl(1/p_{\alpha,\beta},1/p_{\alpha,\beta}'\bigr),
\]
where \(p_{\alpha,\beta}'\) is the Hölder conjugate of
\(p_{\alpha,\beta}\) (see Figure \ref{fig1}). Let \(\mathrm{Trap}_\beta\) denote the closed
trapezoid with vertices $(0,0)$, $(1,1)$, $D_\beta:=(1/r_\beta, 1/2),$ and $D_\beta':=(1/2,1/r_\beta')$ (see, for example,  Figure \ref{fig:geo}).
\end{defn}

By an argument based on the Littlewood--Paley theory, the Sobolev embedding,  and interpolation, we obtain the following theorem.

\begin{thm}
\label{suff}
Let $\alpha,\beta$ satisfy \eqref{ab}. 
Suppose that $\mu\in \mathcal P_{\alpha, \beta} (\mathbb R^d)$.
Then,
\begin{equation}
\label{conv}
\|\mu\ast f\|_q\le C\|f\|_p
\end{equation}
holds whenever
$
(1/p,1/q)
\in
\operatorname{conv}
\bigl(\Delta_{\alpha,\beta}\cup\mathrm{Trap}_\beta\bigr).
$
\end{thm}

If \(\beta\le\alpha\), then
\(\mathrm{Trap}_\beta\subset\Delta_{\alpha,\beta}\), so the region
above equals \(\Delta_{\alpha,\beta}\) (see Figure \ref{fig:geo}). If \(\alpha<\beta\), we write
\[
\mathrm{Pent}_{\alpha,\beta}
:=
\operatorname{conv}
\bigl(\Delta_{\alpha,\beta}\cup\mathrm{Trap}_\beta\bigr).
\]
At the endpoint \(\alpha=\beta/2\), one has
\(\mathrm{Pent}_{\alpha,\beta}=\mathrm{Trap}_\beta\) (see Figure \ref{fig:ngeo}).

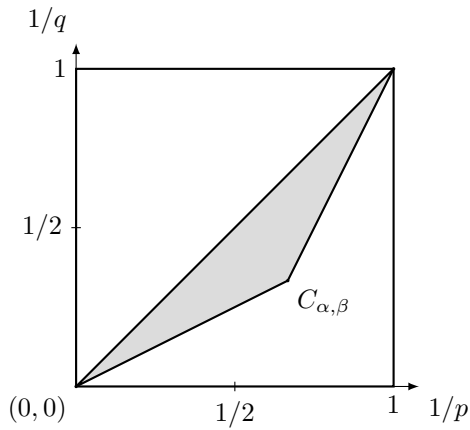
\begin{figure}[t]
    \begin{tikzpicture}[scale=.7, x=6cm,y=6cm]
        \def\pzero{1.5} 
        \pgfmathsetmacro{\xC}{1/\pzero}
        \pgfmathsetmacro{\yC}{1 - 1/\pzero} 

        \pgfmathsetmacro{\xHalf}{0.5}
        \pgfmathsetmacro{\yHalfRaw}{(\yC/\xC)*\xHalf} 
        \pgfmathsetmacro{\yHalf}{min(max(\yHalfRaw,0),1)} 
        \pgfmathsetmacro{\xHalfP}{1 - \yHalf}
        \pgfmathsetmacro{\yHalfP}{1 - \xHalf}

        \coordinate (H) at (\xHalf,\yHalf);
        \coordinate (I) at (\xHalfP,\yHalfP);

        \coordinate (A) at (0,0);
        \coordinate (B) at (1,1);
        \coordinate (C) at (\xC,\yC);

        \draw[-latex] (0,0) -- (1.08,0) node[below right] {$1/p$};
        \draw[-latex] (0,0) -- (0,1.08) node[above left] {$1/q$};
        \draw[thick] (0,0) rectangle (1,1);

        \filldraw[
            fill=black!18,
            fill opacity=0.75,
            draw=black,
            thick,
            line join=round
        ] (A) -- (C) -- (B) -- cycle;

        \fill (A) circle (0.9pt) node[below left] {$(0,0)$};
        \fill (B) circle (0.9pt) node[above right] {};
        \fill (C) circle (0.9pt) node[below right] {$C_{\alpha,\beta}$};

        \draw (1,0) node[below] {$1$};
        \draw (0,1) node[left] {$1$};
        \draw (0.5,0.015) -- (0.5,-0.015) node[below] {$1/2$};
        \draw (0.015,0.5) -- (-0.015,0.5) node[left] {$1/2$};
    \end{tikzpicture}
    \caption{%
    The triangle $\Delta_{\alpha,\beta}$.}
    \label{fig1}
\end{figure}

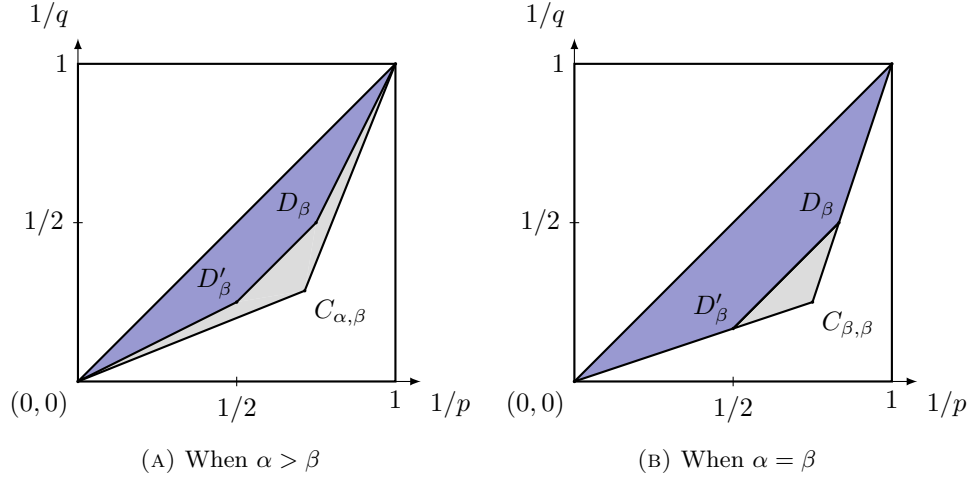
\begin{figure}[t]
    \hspace{-3pt}
    \begin{subfigure}[t]{0.48\textwidth}
        \begin{tikzpicture}[scale=.7, x=6cm,y=6cm]
            \def\dval{3}
            \def\alphaval{2}
            \def\betaval{1.5}
            \pgfmathsetmacro{\xC}{(\dval - \alphaval + \betaval)/(2*\dval - 2*\alphaval + \betaval)}
            \pgfmathsetmacro{\yC}{(\dval - \alphaval)/(2*\dval - 2*\alphaval + \betaval)}
            \pgfmathsetmacro{\yD}{(\dval - \betaval)/(2*\dval)}
            \pgfmathsetmacro{\xE}{(\dval + \betaval)/(2*\dval)}

            \coordinate (A) at (0,0);
            \coordinate (B) at (1,1);
            \coordinate (C) at (\xC, \yC);
            \coordinate (D) at (0.5,\yD);
            \coordinate (E) at (\xE, 0.5);

            \draw[-latex] (0,0) -- (1.08,0) node[below right] {$1/p$};
            \draw[-latex] (0,0) -- (0,1.08) node[above left] {$1/q$};
            \draw[thick] (0,0) rectangle (1,1);

            \fill[
                fill=black!18,
                fill opacity=0.75,
            ] (A) --  (D) -- (C) -- cycle;

              \fill[
                fill=black!18,
                fill opacity=0.75,
            ]  (E) -- (D)-- (C) -- cycle;

              \fill[
                fill=black!18,
                fill opacity=0.75,
            ]  (E)-- (C) -- (B) -- cycle;

            \draw[black,  thick] (A)--(C)--(B);

            \filldraw[
                fill=blue!55!black,
                fill opacity=0.40,
                draw=black,
                thick,
                line join=round
            ] (A) -- (D) -- (E) -- (B) -- cycle;

            \fill (A) circle (0.9pt) node[below left] {$(0,0)$};
            \fill (B) circle (0.9pt) node[above right] {};
            \fill (C) circle (0.9pt) node[below right] {$C_{\alpha,\beta}$};
            \fill (D) circle (0.9pt) node[above left= -3pt] {$D_\beta'$};
            \fill (E) circle (0.9pt) node[above left=-3pt] {$D_\beta$};

            \draw (1,0) -- (1,0) node[below] {$1$};
            \draw (0,1) -- (0,1) node[left] {$1$};
            \draw (0.5,0.015) -- (0.5,-0.015) node[below] {$1/2$};
            \draw (0.015,0.5) -- (-0.015,0.5) node[left] {$1/2$};
        \end{tikzpicture}
        \caption{When $\alpha>\beta$}
    \end{subfigure}
    \hfill
    \begin{subfigure}[t]{0.48\textwidth}
        \begin{tikzpicture}[scale=.7, x=6cm,y=6cm]
            \def\dval{3}
            \def\alphaval{2}
            \def\betaval{2}
            \pgfmathsetmacro{\xC}{(\dval - \alphaval + \betaval)/(2*\dval - 2*\alphaval + \betaval)}
            \pgfmathsetmacro{\yC}{(\dval - \alphaval)/(2*\dval - 2*\alphaval + \betaval)}
            \pgfmathsetmacro{\yD}{(\dval - \betaval)/(2*\dval)}
            \pgfmathsetmacro{\xE}{(\dval + \betaval)/(2*\dval)}

            \coordinate (A) at (0,0);
            \coordinate (B) at (1,1);
            \coordinate (C) at (\xC, \yC);
            \coordinate (D) at (0.5,\yD);
            \coordinate (E) at (\xE, 0.5);

            \draw[-latex] (0,0) -- (1.08,0) node[below right] {$1/p$};
            \draw[-latex] (0,0) -- (0,1.08) node[above left] {$1/q$};
            \draw[thick] (0,0) rectangle (1,1);

            \filldraw[
                fill=black!18,
                fill opacity=0.75,
                draw=black,
                thick,
                line join=round
            ] (D) -- (C) -- (E) -- cycle;

            \filldraw[
                fill=blue!55!black,
                fill opacity=0.40,
                draw=black,
                thick,
                line join=round
            ] (A) -- (D) -- (E) -- (B) -- cycle;

            \fill (A) circle (0.9pt) node[below left] {$(0,0)$};
            \fill (B) circle (0.9pt) node[above right] {};
            \fill (C) circle (0.9pt) node[below right] {$C_{\beta,\beta}$};
            \fill (D) circle (0.9pt) node[above left=-3pt] {$D_\beta'$};
            \fill (E) circle (0.9pt) node[above left=-3pt] {$D_\beta$};

            \draw (1,0) -- (1,0) node[below] {$1$};
            \draw (0,1) -- (0,1) node[left] {$1$};
            \draw (0.5,0.015) -- (0.5,-0.015) node[below] {$1/2$};
            \draw (0.015,0.5) -- (-0.015,0.5) node[left] {$1/2$};
        \end{tikzpicture}
        \caption{When $\alpha=\beta$}
    \end{subfigure}
    \caption{The geometric case. In each panel, the outer region is \(\Delta_{\alpha,\beta}\), and the darker region is \(\mathrm{Trap}_\beta\).}
    \label{fig:geo}
\end{figure}

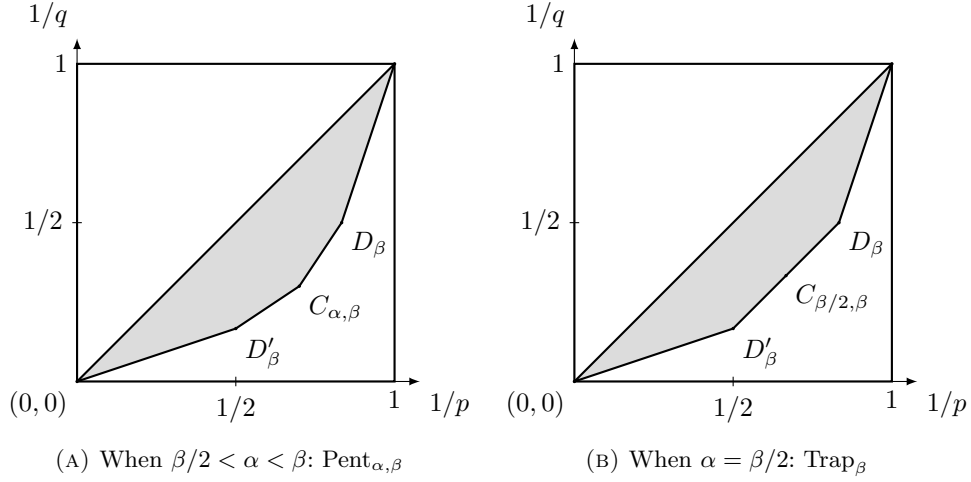
\begin{figure}[t]
    \begin{subfigure}[t]{0.48\textwidth}
    \begin{tikzpicture}[scale=.7, x=6cm,y=6cm]
        \def\dval{3}
        \def\alphaval{1.5}
        \def\betaval{2}
        \pgfmathsetmacro{\xC}{(\dval - \alphaval + \betaval)/(2*\dval - 2*\alphaval + \betaval)}
        \pgfmathsetmacro{\yC}{(\dval - \alphaval)/(2*\dval - 2*\alphaval + \betaval)}
        \pgfmathsetmacro{\yD}{(\dval - \betaval)/(2*\dval)}
        \pgfmathsetmacro{\xE}{(\dval + \betaval)/(2*\dval)}

        \coordinate (A) at (0,0);
        \coordinate (B) at (1,1);
        \coordinate (C) at (\xC, \yC);
        \coordinate (D) at (0.5,\yD);
        \coordinate (E) at (\xE, 0.5);

        \draw[-latex] (0,0) -- (1.08,0) node[below right] {$1/p$};
        \draw[-latex] (0,0) -- (0,1.08) node[above left] {$1/q$};
        \draw[thick] (0,0) rectangle (1,1);

        \fill[black!18,opacity=0.75] (A) -- (D) -- (C) -- (E) -- (B) -- cycle;
        \draw[thick, line join=round] (A) -- (D) -- (C) -- (E) -- (B) -- cycle;

        \fill (A) circle (0.9pt) node[below left] {$(0,0)$};
        \fill (B) circle (0.9pt) node[above right] {};
        \fill (C) circle (0.9pt) node[below right] {$C_{\alpha,\beta}$};
        \fill (D) circle (0.9pt) node[below right] {$D_\beta'$};
        \fill (E) circle (0.9pt) node[below right] {$D_\beta$};

        \draw (1,0) -- (1,0) node[below] {$1$};
        \draw (0,1) -- (0,1) node[left] {$1$};
        \draw (0.5,0.015) -- (0.5,-0.015) node[below] {$1/2$};
        \draw (0.015,0.5) -- (-0.015,0.5) node[left] {$1/2$};
        \end{tikzpicture}
            \caption{When $\beta/2<\alpha<\beta$: $\mathrm{Pent}_{\alpha, \beta}$}
        \end{subfigure}\hfill
        \begin{subfigure}[t]{0.48\textwidth}
        \begin{tikzpicture}[scale=.7, x=6cm,y=6cm]
        \def\dval{3}
        \def\alphaval{1}
        \def\betaval{2}
        \pgfmathsetmacro{\xC}{(\dval - \alphaval + \betaval)/(2*\dval - 2*\alphaval + \betaval)}
        \pgfmathsetmacro{\yC}{(\dval - \alphaval)/(2*\dval - 2*\alphaval + \betaval)}
        \pgfmathsetmacro{\yD}{(\dval - \betaval)/(2*\dval)}
        \pgfmathsetmacro{\xE}{(\dval + \betaval)/(2*\dval)}

        \coordinate (A) at (0,0);
        \coordinate (B) at (1,1);
        \coordinate (C) at (\xC, \yC);
        \coordinate (D) at (0.5,\yD);
        \coordinate (E) at (\xE, 0.5);

        \draw[-latex] (0,0) -- (1.08,0) node[below right] {$1/p$};
        \draw[-latex] (0,0) -- (0,1.08) node[above left] {$1/q$};
        \draw[thick] (0,0) rectangle (1,1);

        \fill[black!18,opacity=0.75] (A) -- (D) -- (C) -- (E) -- (B) -- cycle;
        \draw[thick, line join=round] (A) -- (D) -- (C) -- (E) -- (B) -- cycle;

        \fill (A) circle (0.9pt) node[below left] {$(0,0)$};
        \fill (B) circle (0.9pt) node[above right] {};
        \fill (C) circle (0.9pt) node[below right] {$C_{\beta/2,\beta}$};
        \fill (D) circle (0.9pt) node[below right] {$D_\beta'$};
        \fill (E) circle (0.9pt) node[below right] {$D_\beta$};

        \draw (1,0) -- (1,0) node[below] {$1$};
        \draw (0,1) -- (0,1) node[left] {$1$};
        \draw (0.5,0.015) -- (0.5,-0.015) node[below] {$1/2$};
        \draw (0.015,0.5) -- (-0.015,0.5) node[left] {$1/2$};
    \end{tikzpicture}
    \caption{When $\alpha=\beta/2$: $\mathrm{Trap}_{\beta}$}
    \end{subfigure}
    \caption{The nongeometric case.}
    \label{fig:ngeo}
\end{figure}

Theorem~\ref{suff} records a universal \(L^p\)--\(L^q\) boundedness region
under the two quantitative assumptions \eqref{frostman} and \eqref{fourier}.
Its proof uses standard Littlewood--Paley and Sobolev estimates together with
interpolation. In the geometric range \(\beta\le\alpha\), the triangular
estimate away from the endpoint \(C_{\alpha,\beta}\) is contained in
\cite[Lemma~4.1]{HW19};
the discussion following that lemma notes that the endpoint
estimate is also known in the Euclidean setting, although it had apparently
not appeared in print
(see Section~\ref{sufficiency}). To the best of our knowledge, the
resulting formulation in the nongeometric range \(\alpha<\beta\) has not
appeared explicitly in the literature. Theorem~\ref{suff} can be viewed as the
fractal-measure extension of the Strichartz--Littman theorem for spherical
convolution. Indeed, when \(\alpha=\beta=d-1\) and $d\ge 2$, the region in
Theorem~\ref{suff} recovers the Strichartz--Littman triangle.

However, unlike in the case of the spherical convolution operator
\(f \mapsto f * \sigma\), the sharp \((p,q)\) range for \eqref{conv} is less
transparent when \(\mu\) is assumed only to satisfy \eqref{frostman} and
\eqref{fourier}. In the case \(\alpha<\beta\), referred to below as the
nongeometric case, the available estimates are also related to the
Sobolev-regularity estimates described below, which follow from Fourier decay.
This naturally leads to the question of sharpness of the universal range.

Closely related questions arise in the study of restriction estimates, where the sharp range of exponents has been extensively investigated, particularly in connection with extensions of the Stein--Tomas theorem to fractal measures \cite{Mockenhaupt,Mitsis,BakSeeger}.
Sharp necessary conditions and, in many cases, matching sufficient conditions have been established for a variety of measures (see, for example, \cite{HL13, Hambrook, Chen, CS}).
A more detailed discussion will be given in Section~\ref{sec:restriction}.

There is, of course, a substantial literature on \(L^p\)-improving convolution
operators. Classical results treat surface measures on curved submanifolds, and
there are also important examples of singular measures with \(L^p\)-improving
properties. These works, however, are tied either to specific geometric models or
to particular singular measures. They do not address \emph{the following universal
sharpness problem: given only the exponents \(\alpha\) and \(\beta\) in the
Frostman and Fourier decay assumptions, what is the largest \(L^p\)-\(L^q\) type
set that is forced for every measure in \(\mathcal P_{\alpha,\beta}(\mathbb R^d)\),
and can its boundary be realized by a single extremal measure?}

While the restriction theory under Frostman and
Fourier decay hypotheses has a well-developed sharpness literature, the
corresponding sharpness theory for convolution operators had not been developed
in comparable generality.  There has
been no systematic study of sharp \(L^p\)-\(L^q\) convolution type sets for
measures satisfying only assumptions such as \eqref{frostman} and \eqref{fourier}, despite the fact
that convolution is one of the most basic operations associated with a measure.

\subsection{\texorpdfstring{The optimal universal type set for convolution}{The optimal universal type set for convolution}}

For a compactly supported probability measure \(\mu\), let
\(\mathfrak T(\mu)\subset[0,1]^2\) consist of those \((1/p,1/q)\) for which
\(f\mapsto\mu*f\) is bounded from \(L^p(\mathbb R^d)\) to
\(L^q(\mathbb R^d)\), and define
\[
\mathfrak U_{\alpha,\beta}
=
\bigcap_{\mu\in\mathcal P_{\alpha,\beta}(\mathbb R^d)}
\mathfrak T(\mu).
\]
Here the operator norm may depend on \(\mu,p,q\); \emph{``universal''} means
boundedness for every measure in the class.

The main purpose of this paper is to determine the optimal universal type set
for convolution with measures in \(\mathcal P_{\alpha,\beta}(\mathbb R^d)\).
More precisely, we identify the largest region of exponents  forced solely by the conditions \eqref{frostman} and \eqref{fourier}, and we construct
single extremal measures, of the smallest possible support dimension, for which
no improvement outside this region is possible.

To this end, we distinguish two cases
\[
\beta \le \alpha, \quad  \alpha  < \beta.
\]
Following Li--Liu \cite{LL}, we call the first case the \emph{geometric case} and the second the \emph{nongeometric case.}
Although Theorem~\ref{suff} is stated uniformly across these two regimes,
their sharpness mechanisms are markedly different, and it is therefore
necessary to treat them separately.
In the geometric case, the condition \eqref{frostman} plays a more significant role, enabling us to construct a measure whose support has Hausdorff dimension $\alpha$.
In contrast, the condition \eqref{fourier} becomes more dominant in the nongeometric case.

\subsubsection*{Geometric case}

Our first main theorem shows that, in the geometric regime, the sufficient
range \(\Delta_{\alpha,\beta}\) in Theorem~\ref{suff} is the optimal universal
range for the class \(\mathcal P_{\alpha,\beta}(\mathbb R^d)\).


\begin{thm}
\label{thm:conv-geo-i}
Suppose that \(0 < \beta \le \alpha < d\).
Then \(\mathfrak U_{\alpha,\beta}=\Delta_{\alpha,\beta}\).
Moreover, there exists \(\mu \in \mathcal{P}_{\alpha,\beta}(\mathbb{R}^d)\)
with \(\dim_{\mathcal{H}}(\operatorname{supp}\mu)=\alpha\) such that
\(\mathfrak T(\mu)=\mathfrak U_{\alpha,\beta}\).
\end{thm}

Since every measure in \(\mathcal P_{\alpha,\beta}(\mathbb R^d)\) is \(\alpha\)-Frostman, its support must have Hausdorff dimension at least \(\alpha\).
Thus, Theorem~\ref{thm:conv-geo-i} yields a singular measure whose support has the smallest possible dimension.
Once such a measure in \(\mathcal P_{\alpha,\beta}(\mathbb R^d)\) is available, it is not difficult, for any \(t\in (\alpha,d)\), to construct a measure \(\mu\in \mathcal P_{\alpha,\beta}(\mathbb R^d)\) with \(\dim_{\mathcal H}(\operatorname{supp}\mu)=t\) such that \(f\mapsto \mu*f\) has the same \(L^p\)--\(L^q\) boundedness properties as in Theorem~\ref{thm:conv-geo-i} (see Remark~\ref{rem:low}).

The proof of our theorem is based on the construction of certain measures.
In order to state our result, we need some definitions.

\begin{defn}
\label{near-ad}
Let $0 \le \alpha \le d$.
We say that $\mu \in \mathcal{P}(\mathbb{R}^d)$ is \emph{near $\alpha$-Ahlfors--David regular} if $\mu$ satisfies \eqref{frostman} and
\begin{equation}
\label{lower}
\mu(B(x, \rho)) \gtrsim_{\varepsilon} \rho^{\alpha + \varepsilon}, \quad \forall (x, \rho)\in \supp \mu\times (0, 1)
\end{equation}
for all $\varepsilon > 0$.   
\end{defn}

Throughout the paper, we abbreviate ``near $\alpha$-Ahlfors--David regular'' to ``near $\alpha$-AD regular''.
If a measure $\mu$ is near $\alpha$-AD regular and of $\alpha/2$-Fourier decay, then it follows that
\[
\dim_{\mathcal F}(\operatorname{supp}\mu)
=
\dim_{\mathcal H}(\operatorname{supp}\mu)
=
\underline{\dim}_{\mathcal M}(\operatorname{supp}\mu)
=
\overline{\dim}_{\mathcal M}(\operatorname{supp}\mu)
=
\alpha.
\]
See Lemma~\ref{lem:element}.
Here, $\dim_{\mathcal{F}}$, $\dim_{\mathcal{H}}$, $\underline{\dim}_{\mathcal M}$, and $\overline{\dim}_{\mathcal M}$ denote the Fourier, Hausdorff, lower Minkowski, and upper Minkowski dimensions, respectively.
Consequently, $\operatorname{supp}\mu$ is a Salem set, and $\mu$
realizes the Fourier dimension of its support. We refer to such a measure as a
Salem measure. 

The construction of the measure $\mu$ in the following proposition is crucial for the proof of Theorem \ref{thm:conv-geo-i}.
For a set $E \subset \mathbb{R}^d$ and $\delta > 0$, we write $(E)_\delta$ for the $\delta$-neighborhood of $E$ and $|E|$ for the Lebesgue measure of $E$.

\begin{prop}
\label{prop:conv-geo-factorization-i}
Let $0 < \beta \le \alpha < d$.
Then there exists \(\mu \in \mathcal{P}_{\alpha,\beta}(\mathbb{R}^d)\) with $\dim_{\mathcal{H}}(\operatorname{supp}\mu) = \alpha$ such that
\begin{equation}
\label{eq:conv-geo-factorization-i}
\mu = {\aracc{\mu}} * {\rdacc{\mu}},
\end{equation}
for some near $(\alpha - \beta)$-AD regular measure ${\aracc{\mu}}$ and some near $\beta$-AD regular measure ${\rdacc{\mu}}$ with $\beta/2$-Fourier decay.
Moreover, for $\varepsilon > 0$ we have
\begin{equation}
\label{eq:small-difference-set-i}
\bigl|(\operatorname{supp}{\aracc{\mu}} - \operatorname{supp}{\aracc{\mu}})_{2\delta}\bigr|
\lesssim_{\varepsilon} \delta^{\,d - (\alpha - \beta) - \varepsilon}, \quad  \forall\delta\in(0,1).
\end{equation}
\end{prop}

The two factors in Proposition~\ref{prop:conv-geo-factorization-i}
play complementary roles in the proof of Theorem~\ref{thm:conv-geo-i}.
The structured factor \({\aracc{\mu}}\) supplies the resonance quantified by
\eqref{eq:small-difference-set-i}, while the Salem-type factor
\({\rdacc{\mu}}\) provides the required Fourier decay without destroying the
ambient \(\alpha\)-dimensional geometry of \(\mu\); see
Section~\ref{sec2-2}.

\subsubsection*{Nongeometric case}
In the nongeometric regime \(\alpha<\beta\), the trapezoid
\(\mathrm{Trap}_\beta\) is not contained in \(\Delta_{\alpha,\beta}\), and the
sufficient region in Theorem~\ref{suff} is therefore the larger set
\(\mathrm{Pent}_{\alpha,\beta}\). The next theorem shows that this region is the
optimal universal range and is realized by a single extremal measure.

\begin{thm}
\label{thm:conv-ngeo-i}
Let $\alpha,$ $\beta$ satisfy \eqref{ab} with $\alpha<\beta$.
Then \(\mathfrak U_{\alpha,\beta}=\mathrm{Pent}_{\alpha,\beta}\).
Moreover, there exists \(\mu\in\mathcal P_{\alpha,\beta}(\mathbb R^d)\)
with \(\dim_{\mathcal H}(\operatorname{supp}\mu)=\beta\) such that
\(\mathfrak T(\mu)=\mathfrak U_{\alpha,\beta}\).
\end{thm}

Since \eqref{fourier} implies $\dim_{\mathcal{H}}(\operatorname{supp}\mu) \ge \beta$, Theorem \ref{thm:conv-ngeo-i} also provides a singular measure $\mu$ whose support has the smallest possible dimension.

We say that a set $E$ is \emph{near $s$-AD regular} if $E$ is the support of a near $s$-AD regular measure.
The following is the key proposition, which we use to prove Theorem \ref{thm:conv-ngeo-i}.

\begin{prop}
\label{prop:heavy-core-i}
Let $\alpha,$ $\beta$ satisfy \eqref{ab} with $\alpha< \beta$.
Then, for $s \in [0, 2\alpha - \beta]$, there exists $\mu \in \mathcal{P}_{\alpha,\beta}(\mathbb{R}^d)$ with $\dim_{\mathcal{H}}(\operatorname{supp}\mu) = \beta$ and a near $s$-AD regular set $F \subset \operatorname{supp}\mu$ such that
\begin{equation}
\label{eq:nearF}
\mu(B(x, \delta)) \gtrsim_\varepsilon \delta^{\alpha + \varepsilon},
\quad \forall (x,\delta) \in F\times (0,1)
\end{equation}
for any $\varepsilon>0$.
\end{prop}

The set \(F\) in Proposition~\ref{prop:heavy-core-i} is a near
\(s\)-AD regular \emph{heavy core} on which \(\mu\) carries nearly
\(\alpha\)-dimensional mass at every small scale. Testing on \(F\) yields the
sharp nongeometric necessary condition beyond that forced by the ambient
\(\beta\)-dimensional support.

\subsection{\texorpdfstring{Sharpness of the $L^2$ restriction estimate}{Sharpness of the L2 restriction estimate}}
\label{sec:restriction}

Modifying our constructions, we can address sharpness of $L^2$ adjoint restriction estimates associated with the measures $\mu\in \mathcal P_{\alpha, \beta}(\mathbb R^d).$

The adjoint Fourier restriction estimate
\begin{equation}
\label{st-}
\| \widehat{f \,d\mu} \|_{L^q(\mathbb{R}^d)} \lesssim_q \| f \|_{L^2(\mu)}, \quad  \forall  f \in L^2(\mu)
\end{equation}
has been extensively studied by various authors under the assumptions \eqref{frostman} and \eqref{fourier}.
As mentioned above, such estimates extend the classical Stein--Tomas theorem to general measures.
The following is due to the works of Mockenhaupt \cite{Mockenhaupt}, Mitsis \cite{Mitsis}, and Bak--Seeger \cite{BakSeeger}.

\begin{thm}
\label{thm:MMBS}
Let $\alpha, \beta $ satisfy \eqref{ab}, and let $q_*(\alpha, \beta, d) := (4d - 4\alpha + 2\beta)/\beta$.
Suppose $\mu \in \mathcal{P}_{\alpha, \beta}(\mathbb{R}^d)$.
Then, \eqref{st-} holds for $q \ge q_*(\alpha, \beta, d)$.
\end{thm}

For classical curved models, such as surface measure on the sphere or
arclength measure on a planar curve with nonvanishing curvature, sharpness of
the corresponding threshold follows from a Knapp-type example.
However, for general measures, the sharpness of the threshold exponent $q_*(\alpha,\beta,d)$ remained open for some time and has been studied extensively in recent years. In dimension one, Hambrook--\L{}aba \cite{HL13} and Hambrook \cite{Hambrook} showed that the range $q \ge q_*(\alpha,\beta,d)$ is sharp for the class of measures satisfying \eqref{frostman} and \eqref{fourier} in the geometric regime $\alpha \ge \beta$: for each $q < q_*(\alpha,\beta,d)$, they constructed a measure, depending on $q$, for which \eqref{st-} fails.

Later, Chen \cite{Chen} obtained a similar but more refined sharpness result on $\mathbb{R}$ by constructing a near $\alpha$-AD regular measure.
In particular, he showed that the estimate \eqref{st-} fails for a single measure $\mu\in \mathcal{P}_{\alpha,\beta}$ whenever $q<q_*(\alpha,\beta,d)$.

In higher dimensions $d\ge2$, Hambrook and \L{}aba \cite{HL16} proved that the exponent $q_*(\alpha,\beta,d)$ is sharp in the case $\alpha\ge\beta$ and $\beta\le d-1$.
However, all of these results concern the geometric case.
More recently, Fraser--Hambrook--Ryou \cite{FHR25-1} and Li--Liu \cite{LL} independently proved sharpness in dimension one in both the geometric and nongeometric cases, and Fraser--Hambrook--Ryou \cite{FHR25-2} subsequently established sharpness in all dimensions, again covering both the geometric and nongeometric cases.

Nevertheless, the higher-dimensional result of Fraser--Hambrook--Ryou
\cite{FHR25-2} does not directly yield a single measure in
\(\mathcal P_{\alpha,\beta}(\mathbb R^d)\) that fails throughout the full
subcritical range. Their construction is simultaneously sharp below the
threshold associated with auxiliary exponents \(\alpha_0>\alpha\) and
\(\beta_0>\beta\). Given a prescribed exponent
\(q_0<q_*(\alpha,\beta,d)\), these auxiliary exponents are chosen so that
\(q_0<q_*(\alpha_0,\beta_0,d)<q_*(\alpha,\beta,d)\). The resulting measure
belongs to \(\mathcal P_{\alpha,\beta}(\mathbb R^d)\) and yields failure for
every \(q<q_*(\alpha_0,\beta_0,d)\), but this argument does not directly
provide one fixed measure covering the whole range
\(q<q_*(\alpha,\beta,d)\).

To the best of our knowledge, constructing such a single measure, in the
spirit of Chen \cite{Chen}, remained open in higher dimensions and in the
nongeometric case.

The following theorem settles this issue by constructing, in every dimension and in both the geometric and nongeometric cases, a measure $\mu\in \mathcal{P}_{\alpha,\beta}(\mathbb{R}^d)$ for which \eqref{st-} fails precisely whenever $q<q_*(\alpha,\beta,d)$.

\begin{thm}
\label{thm:res-sharp}
Let $\alpha, \beta $ satisfy \eqref{ab}.
Then there exists $\mu \in \mathcal{P}_{\alpha,\beta}(\mathbb{R}^d)$ such that the estimate \eqref{st-} holds if and only if $q\ge q_*(\alpha,\beta,d)$, and
\begin{equation}
\label{eq:supp}
\dim_{\mathcal{H}}(\operatorname{supp} \mu) = \max\{\alpha, \beta\}.
\end{equation}
\end{thm}

As we have seen in the case of convolution, Theorem \ref{thm:res-sharp}  provides such a measure $\mu$ with $\supp \mu$ having the smallest possible dimension $\max\{\alpha, \beta\}$. Following the argument in {\it Remark} \ref{rem:low} below, one can also construct measures whose supports have Hausdorff dimension larger than \(\max\{\alpha,\beta\}\), while still implying the same optimality result.

Theorem~\ref{thm:res-sharp} follows from Theorem~\ref{thm:MMBS} once the
corresponding necessary conditions are realized by a single measure. The following theorem provides necessary conditions on the exponents $p$ and $q$ for the estimate
\begin{equation}
\label{eq:respq}
\|\widehat{f\,d\mu}\|_{L^q(\mathbb{R}^d)}
\lesssim
\|f\|_{L^p(\mu)}
\end{equation}
to hold. We denote $a_+=\max\{a, 0\}$ for $a\in \mathbb R$.

\begin{thm}
\label{thm:res-geo-i}
Let $\alpha, \beta $ satisfy \eqref{ab}.
Then there exists \(\mu \in \mathcal{P}_{\alpha,\beta}(\mathbb{R}^d)\) satisfying \eqref{eq:supp} such that \eqref{eq:respq} holds only if $q \ge (2d-2(\alpha-\beta)_+)/\beta$ and
\begin{equation}    \label{eq:reconpq}
q \ge \frac{2d-2\alpha+\beta}{\beta}p'.
\end{equation}
\end{thm}

Theorem~\ref{thm:res-geo-i} goes beyond the sharpness of the \(L^2(\mu)\)--\(L^q(\mathbb R^d)\) restriction estimate. For each admissible pair \((\alpha,\beta)\), it provides a measure in \(\mathcal P_{\alpha,\beta}(\mathbb R^d)\) for which the estimate \eqref{eq:respq} is constrained by the natural necessary conditions predicted for general \(L^p(\mu)\)--\(L^q(\mathbb R^d)\) restriction bounds. For this measure, the expected obstructions are genuinely realized not only in the case \(p=2\), but also for the general \(L^p\)--\(L^q\) problem. On the other hand, although these conditions are natural from the viewpoints of scaling and geometry, matching sufficient conditions for \(p>2\) appear to require additional
structure beyond Frostman regularity and Fourier decay, and are not pursued
here.

For the nongeometric case, $\beta> \alpha$, we prove Theorem \ref{thm:res-geo-i} using the following theorem, which provides near AD regular Salem measures.

\begin{thm}
\label{thm:near-AD-Salem}
For every $\alpha \in [0, d]$, there exists a near $\alpha$-AD regular measure $\mu\in \mathcal P(\mathbb R^d)$ with $\alpha/2$-Fourier decay.
\end{thm}

By contrast, the proof of Theorem~\ref{thm:res-geo-i} in the geometric case \(\beta \le \alpha\) is substantially more delicate.
In this regime, Theorem~\ref{thm:near-AD-Salem} alone is no longer sufficient, since one must simultaneously preserve the ambient \(\alpha\)-dimensional Frostman geometry and incorporate enough arithmetic resonance to force the sharp necessary conditions for \eqref{eq:respq}.
As will be seen in Section~\ref{rest-geo}, the measure \(\mu\) is therefore constructed in the convolution form $\mu={\aracc{\mu}} * {\rdacc{\mu}}$, where one factor provides the required Fourier decay, while the other carries the additive structure responsible for sharpness.
Thus, unlike the nongeometric case, the geometric case requires a genuinely hybrid argument, combining probabilistic and arithmetic ideas.

We record the support-enlargement argument in the following remark.

\begin{rem}
\label{rem:low} It is enough to describe the geometric case, since the nongeometric case is handled in the same manner.
Let $\alpha\ge \beta$ and \(\mu_0\) be the measure given by Theorem~\ref{thm:conv-geo-i}.
For \(t\in(\alpha,d)\), let \(\nu_t\in \mathcal P_{t,t}(\mathbb R^d)\) be a measure with \(\dim_{\mathcal H}(\operatorname{supp}\nu_t)=t\), whose existence is guaranteed by Theorem \ref{thm:near-AD-Salem}.
By translation, we may assume that \(\supp \nu_t\) is disjoint from \(\supp\mu_0\), and setting
\[
\mu_t=\frac12(\mu_0+\nu_t),
\]
we obtain a measure \(\mu_t\in\mathcal P_{\alpha,\beta}(\mathbb R^d)\) with \(\dim_{\mathcal H}(\operatorname{supp}\mu_t)=t\). 
Thus, by  Theorem \ref{suff}   \(f\mapsto \mu_t*f\) is bounded from $L^p$ to $L^q$  if $(1/p,1/q)\in \Delta_{\alpha,\beta}.$  Conversely, the counterexamples
used for \(\mu_0\) are nonnegative. For those test
functions, $
\mu_t*f \ge \frac12\,\mu_0*f.$  Hence,  failure of any  \(L^p\)--\(L^q\) bound for \(\mu_0\) continues to hold  for
\(\mu_t\), as desired.

The same support-enlargement argument applies separately to restriction estimates, starting with a measure furnished by Theorem~\ref{thm:res-sharp} (or by Theorem~\ref{thm:res-geo-i} for \eqref{eq:respq}), rather than with the convolution measure above. After adjoining a measure \(\nu_t\) with disjoint support as above, testing functions supported on the original support shows that every restriction counterexample for the original measure persists.
\end{rem}

\subsection{Further sharpness consequences}
\label{further}

The constructions developed in this paper are not limited to the sharpness problems for the estimates \eqref{conv} and \eqref{st-}. They also furnish extremal examples for questions involving the Fourier spectrum and the smoothing properties of convolution powers.

The following was proved by Carnovale--Fraser--de Orellana~\cite[Proposition~4.2]{CFdO24}.

\begin{prop}
Let $\mu$ be a finite Borel measure with compact support on $\mathbb{R}^d$.
Then for all $\theta \in [0, 1]$,
\begin{equation}
\label{eq:fourier-spectrum}
\dim_{\mathcal{F}}^\theta \mu \le \dim_{\mathcal{F}} \mu + d\theta.
\end{equation}
\end{prop}

Here $\dim_{\mathcal{F}}$ and $\dim_{\mathcal{F}}^\theta$ denote the Fourier dimension and Fourier spectrum of measures, respectively (see, for example, \cite[Section~2]{CFdO24}).
While the sharpness of this bound was already known in the case \(\dim_{\mathcal{F}} \mu = 0\), the analogous question for general prescribed positive Fourier dimension has been left open; see the discussion following \cite[Lemma~4.4]{CFdO24}.

Using Theorem~\ref{thm:near-AD-Salem}, we obtain the following, which establishes sharpness of the inequality \eqref{eq:fourier-spectrum}.

\begin{cor}
\label{f-spec}
For every $t \in [0,d]$, there exists a finite Borel measure with compact support on $\mathbb{R}^d$ such that $\dim_{\mathcal{F}} \mu = t$ and
\[
\dim_{\mathcal{F}}^\theta \mu = \dim_{\mathcal{F}} \mu + d\theta
\]
for all $\theta \in [0,1]$.
\end{cor}

For an integer \(m \ge 2\) and a finite Borel measure \(\mu\) on \(\mathbb{R}^d\), we consider the iterated convolution
\[
\mu^{*m} = \underbrace{\mu * \cdots * \mu}_{m\text{ times}}.
\]
The following can be proved by a standard argument based on Plancherel's theorem and energy integrals.
However, we include its proof for completeness in Section \ref{sec:iter-conv}.

\begin{prop}
\label{prop:m-fold-L2}
Suppose $\mu \in \mathcal{P}_{\alpha,\beta}(\mathbb{R}^d)$.  If $\max\{\alpha, \beta\} + (m-1)\beta > d$, then $\mu^{*m} \in L^2(\mathbb{R}^d)$.
\end{prop}

Using Proposition \ref{prop:conv-geo-factorization-i}, we can also prove that the result in Proposition \ref{prop:m-fold-L2} is sharp up to the endpoint.

\begin{cor}
\label{cor:m-fold-singular}
Let $m \ge 2$ be an integer.
For every $\alpha, \beta \in (0, d)$ with $\max\{\alpha,\beta\} + (m-1)\beta < d$, there exists $\mu \in \mathcal{P}_{\alpha,\beta}(\mathbb{R}^d)$ such that $\mu^{*m}$ is singular with respect to Lebesgue measure.
\end{cor}

Our approach belongs to the general probabilistic tradition of constructing fractal measures by means of multiscale random recursive procedures.
In this respect, it is conceptually related both to Bluhm's random recursive construction of Salem sets \cite{Bluhm} and to the probabilistic Fourier-decay scheme introduced by Laba--Pramanik \cite{LP} and subsequently developed by Hambrook--Laba \cite{HL13}, Chen \cite{Chen}, and Shmerkin--Suomala \cite{SS17, SS}.
In all of these works, one exploits independence across scales in order to build measures with favorable Fourier behavior.

\subsubsection*{Novelty of this paper}
However, the present paper differs from these earlier constructions in several essential ways.
Rather than randomizing translation vectors or working in a dyadic martingale framework based on repeated subdivision of cubes, we formulate our constructions in terms of random offspring systems on digit trees; see Section~\ref{sec3}.
Moreover, instead of relying on high-moment estimates for the Fourier transform,\footnote{Here ``high-moment estimates'' refers to bounds for quantities of the form $\mathbb E\bigl[|\widehat\mu(\xi)|^{2q}\bigr]$, $q\in\mathbb N$, which are then used to derive almost sure Fourier decay.} we derive Fourier decay from a refinement identity, a telescoping decomposition into scale-by-scale Fourier increments, and conditional Hoeffding inequalities.
This yields a more systematic framework than the earlier approaches and, in particular, allows us to obtain sharper and more structured examples than those previous methods seem able to produce.

A further distinctive feature of our method is the use of the higher-order kernel
\begin{equation}
\label{eq:r-fold}
\Phi := (\mathbbm{1}_{[0,1)^d})^{*r},
\end{equation}
where \((\mathbbm{1}_{[0,1)^d})^{*r}\) denotes the \(r\)-fold convolution of \(\mathbbm{1}_{[0,1)^d}\) with itself.
The kernel \(\Phi\) replaces the basic cube indicator (see
\eqref{eq:refinement-id-cube} and the discussion following it) and makes it
possible to access Fourier decay exponents \(\beta/2>1\), which are not naturally
available in the dyadic martingale setting; see \cite[Remark~14.2]{SS}.

The main novelty of the paper lies in the way this probabilistic framework is combined with arithmetic structure in order to construct single extremal measures for convolution and restriction problems.
In the nongeometric case, we combine two independent offspring systems so as to produce a heavy core carrying nearly \(\alpha\)-dimensional mass along a distinguished sequence of scales.
In the geometric case, the construction is organized through a factorization
\[
\mu={\aracc{\mu}} * {\rdacc{\mu}},
\]
where \({\aracc{\mu}}\) is additively structured and \({\rdacc{\mu}}\) is a random Salem-type factor.
This convolution decomposition has a decisive advantage: because the two factors are generated by different offspring systems, one can assign different roles to them and control these roles separately.
More precisely, \({\aracc{\mu}}\) is designed to carry the additive structure needed for resonance and small difference-set estimates, while \({\rdacc{\mu}}\) is constructed to provide Fourier decay and near AD regularity.
This separation of tasks ensures that the convolution measure \(\mu\)
remains \(\alpha\)-Frostman, while the two factors supply the required additive
structure and Fourier decay.
In this way, small difference-set estimates, block sparsity modulo \(Q_n\), two-partition sampling, and the implantation of a thin arithmetic subtree can be combined to produce sharp examples for both convolution and Fourier restriction estimates.
Thus, the originality of our method lies not merely in constructing Salem-type measures, but in integrating probabilistic recursive constructions with arithmetic structure so as to realize sharp convolution and restriction phenomena within the class \(\mathcal P_{\alpha,\beta}(\mathbb R^d)\).

\subsubsection*{Organization}
The remaining sections are organized as follows, in a sequence that differs somewhat from the order of presentation in the introduction.
In Section~2, we prove Theorem~\ref{suff}.
Then, assuming Proposition~\ref{prop:conv-geo-factorization-i} and Proposition~\ref{prop:heavy-core-i}, we establish our main results on convolution, namely Theorems~\ref{thm:conv-geo-i} and \ref{thm:conv-ngeo-i}.
In Section~\ref{sec3}, we introduce the basic setup that will be used throughout Sections~\ref{sec4} and~\ref{conv-geo}, and within this framework we prove Theorem~\ref{thm:near-AD-Salem}.
This preliminary construction also serves to illustrate the main ideas underlying the later arguments.
Section~\ref{sec4} is devoted to the proof of Proposition~\ref{prop:heavy-core-i}, while Section~\ref{conv-geo} is devoted to the proof of Proposition~\ref{prop:conv-geo-factorization-i}.
In Sections~\ref{sec6} and~\ref{rest-geo}, we adapt the constructions developed in Sections~\ref{sec4} and~\ref{conv-geo} to prove Theorem~\ref{thm:res-geo-i}, treating the geometric and nongeometric cases separately.
Finally, in Section~\ref{sec8}, we discuss several further sharpness consequences of our constructions.

\subsection*{Notation}

For a set $E$, we denote $E^d=\{ (s_1, \dots, s_d): s_\ell \in E, \ 1\le \ell \le d \}$.
 \section{\texorpdfstring{Sharpness of convolution bounds: \\ Proofs of Theorems \ref{suff}, \ref{thm:conv-geo-i}, and \ref{thm:conv-ngeo-i}}{Sharpness of convolution bounds}}

We first prove Theorem \ref{suff}.
The proof of Theorem \ref{stri-litt} in \cite{Strichartz, Littman} uses interpolation along an analytic family.
The proof of Theorem~\ref{suff} combines a Littlewood--Paley argument for
\(\Delta_{\alpha,\beta}\) with Fourier decay, Plancherel's theorem, Sobolev
embedding, duality, and interpolation for \(\mathrm{Trap}_\beta\). A final
interpolation yields the full convex hull.
We then prove Theorems \ref{thm:conv-geo-i} and \ref{thm:conv-ngeo-i} using Propositions \ref{prop:conv-geo-factorization-i} and \ref{prop:heavy-core-i}, respectively.

\subsection{Proof of Theorem \ref{suff}}
\label{sufficiency} 
Suppose \(\mu\in\mathcal P_{\alpha,\beta}(\mathbb R^d)\).
By interpolation, it suffices to establish \eqref{conv} separately for 
\(
(1/p,1/q)\in \Delta_{\alpha,\beta}\) and \((1/p,1/q)\in \mathrm{Trap}_\beta\). 

We first prove \eqref{conv} for
\((1/p,1/q)\in\Delta_{\alpha,\beta}\).
By Minkowski's inequality,
\eqref{conv} holds for \(1\le p=q\le\infty\). Thus, it suffices to establish
\eqref{conv} for
\((p,q)=(p_{\alpha,\beta},p_{\alpha,\beta}')\).
Interpolation between these estimates gives \eqref{conv} for
\((1/p,1/q)\in\Delta_{\alpha,\beta}\).

Choose \(\psi\in C_c^\infty((1/2,2))\) such that
\[
    \sum_{j\in\mathbb Z}\psi(2^{-j}t)=1,
    \qquad t>0.
\]
For $j\ge 1$, set $\psi_j(\xi) =\psi(2^{-j}|\xi|)$.  Define
$
    \psi_0(\xi)=1-\sum_{j\ge1}\psi_j(\xi).
$
Then \(\psi_0\in C_c^\infty(\mathbb R^d)\).
Let $P_j$ denote the Littlewood--Paley projection operator given by
\[
\widehat{ P_j f}(\xi)= \psi_j (\xi) \widehat f(\xi)
\]
for $j\ge 0$. Since $\mu$ is compactly supported, $\psi_0^\vee \ast \mu\in \mathcal S(\mathbb R^d)$.
Young's convolution inequality gives $\|\mu\ast (P_0 f)\|_q\lesssim \|f\|_p$ for $1\le p\le q\le \infty$.
Thus, it suffices to show
\begin{equation}
\label{pp}
\Big\| \sum_{j\ge 1}  \mu\ast (P_j f) \Big\|_{p_{\alpha, \beta}' }   \le C  \|f\|_{p_{\alpha, \beta} } .
\end{equation}

To this end, we claim that
\begin{equation}
\label{ppp}
\| \mu\ast (P_j f) \|_{p'} \lesssim 2^{(\frac{2(d-\alpha)+\beta}p-\beta-d+\alpha)j} \|f\|_p
\end{equation}
for $1\le p\le 2$.
By interpolation it is sufficient to show for $p=1, 2$.
For \(p=2\), Plancherel's theorem and \eqref{fourier} give
\[
\|  \mu\ast (P_j f)  \|_{2}^2= \int |\psi(2^{-j} |\xi|) \widehat\mu (\xi)|^2 | \widehat f(\xi)|^2 d\xi\lesssim 2^{- \beta j} \| f\|_2^2 .
\]

For the case $p=1$, since $ \mu\ast (P_j f) = \psi_j^\vee \ast \mu \ast f$, we have
\[
|\mu\ast (P_j f) (x)|\le  C 2^{dj}\|f\|_1 \int (1+ 2^{j}|x-y|)^{-N} d\mu(y)
\]
for any $N$.
For a fixed \(N>\alpha\), a dyadic annular decomposition and
\eqref{frostman} yield
\[
\sup_{x\in\mathbb R^d}
\int (1+2^j|x-y|)^{-N}\,d\mu(y)
\lesssim 2^{-\alpha j}.
\]
Consequently,
\(\|\mu\ast(P_jf)\|_\infty
\lesssim 2^{(d-\alpha)j}\|f\|_1\), which proves \eqref{ppp} for \(p=1\).

Choose \(\widetilde\psi\in C_c^\infty((1/4,4))\) such that
\(\widetilde\psi=1\) on \(\operatorname{supp}\psi\), and let
\(\widetilde P_j\) be the Fourier multiplier with symbol
\(\widetilde\psi(2^{-j}|\xi|)\), \(j\ge1\).  Since the exponent in \eqref{ppp} vanishes at \(p=p_{\alpha,\beta}\), and since
\(P_j\widetilde P_j=P_j\), we have
\[
\| \mu*(P_j f)\|_{p_{\alpha,\beta}'}
=\|\mu*(P_j\widetilde P_j f)\|_{p_{\alpha,\beta}'}
\lesssim \|\widetilde P_j f\|_{p_{\alpha,\beta}}.
\]
Note $1<p_{\alpha, \beta}<2< p_{\alpha, \beta}'<\infty$.  Combined with the Littlewood--Paley and Minkowski's inequalities,  the inequality gives
\begin{align*}
\Big\|  \sum_{j\ge 1} \mu*(P_j f)\Big\|_{p_{\alpha,\beta}'}
\lesssim
\Big( \sum_{j\ge 1} \|\widetilde P_j f\|_{p_{\alpha,\beta}}^2\Big)^{1/2} .
\end{align*}
By Minkowski's inequality and the Littlewood--Paley inequality, the right-hand side is bounded by a constant times $\|f\|_{p_{\alpha, \beta} }$.
Indeed, since $ 1< p_{\alpha, \beta}<2$, the right-hand side is bounded by $ \| (\sum_{j\ge 1} |P_j f|^2)^\frac12 \|_{p_{\alpha, \beta} } \lesssim \|f\|_{p_{\alpha, \beta} }.$
Consequently, \eqref{pp} follows, and hence \eqref{conv} holds for
\((1/p,1/q)\in\Delta_{\alpha,\beta}\).

We next prove \eqref{conv} for
\((1/p,1/q)\in\mathrm{Trap}_\beta\).
Since \(\mu\) is a
probability measure, \(|\widehat\mu(\xi)|\le 1\). Together with
\eqref{fourier}, this gives   $
|\widehat\mu(\xi)|^2\lesssim (1+|\xi|^2)^{-\beta/2}. $
Thus, by Plancherel's theorem and the  Sobolev embedding  we have 
\[
\|\mu*f\|_2
\lesssim
\left(\int_{\mathbb R^d}(1+|\xi|^2)^{-\beta/2}
|\widehat f(\xi)|^2\,d\xi\right)^{1/2}
=
\|f\|_{H^{-\beta/2}}
\lesssim
\|f\|_{r_\beta}
\]
for  \(0<\beta<d\). Applying the same estimate to the reflected measure and taking adjoints gives $
\|\mu*f\|_{r_\beta'}\lesssim\|f\|_2. $
Interpolation with the diagonal bounds yields \eqref{conv} for
\((1/p,1/q)\in\mathrm{Trap}_\beta\).
\qed

\subsection{Geometric case: Proof of Theorem \ref{thm:conv-geo-i}}
\label{sec2-2}

We begin by proving the following elementary lemma, which will be used later.

\begin{lem}
\label{lem:element}
Let \(0\le\alpha\le d\) and \(\mu\in\mathcal P(\mathbb R^d)\).
If \eqref{frostman} holds, we have
\begin{equation}
\label{supp-lower}
\bigl|(\operatorname{supp}\mu)_\delta\bigr| \gtrsim \delta^{d - \alpha}, \quad \forall \delta \in (0, 1).
\end{equation}
If \(\mu\) satisfies \eqref{lower}, then
\begin{equation}
\label{supp-upper}
\bigl|(\operatorname{supp}\mu)_{\delta}\bigr| \lesssim_{\varepsilon} \delta^{d - \alpha - \varepsilon}, \quad \forall \delta \in (0,1), \ \forall \varepsilon > 0.
\end{equation}
Moreover, if $\mu$ is near $\alpha$-AD regular, then
\begin{equation}
\label{eq:measure-dim-2}
\dim_{\mathcal{H}}\bigl(\operatorname{supp} \mu\bigr)
= \underline{\dim}_{\mathcal{M}}\bigl(\operatorname{supp}\mu\bigr)
= \overline{\dim}_{\mathcal{M}}\bigl(\operatorname{supp}\mu\bigr)
= \alpha.
\end{equation}
\end{lem}

\begin{proof}
We first prove \eqref{supp-lower}.
Fix $\delta\in(0,1)$ and choose a maximal $\delta$-separated set $\{x_i\}_{i=1}^N \subset \supp\mu$.
Thus, $\supp\mu \subset \bigcup_{i=1}^N B(x_i,\delta)$, and the balls $B(x_i,\delta/2)$ are pairwise disjoint.
Using \eqref{frostman}, we have $1\le \sum_{i=1}^N \mu(B(x_i,\delta)) \lesssim N \delta^\alpha$,
so $N \gtrsim \delta^{-\alpha}$.
Since $\bigcup_i B(x_i,\delta/2)\subset (\supp\mu)_\delta$, $|(\supp\mu)_\delta \gtrsim N \delta^d$.
Therefore, \eqref{supp-lower} follows.

To show \eqref{supp-upper}, let $\{y_j\}_{j=1}^M \subset \supp\mu$ be a maximal $2\delta$-separated set.
Since the balls $\{B(y_j,\delta)\}_{j=1}^M$ are pairwise disjoint, \eqref{lower} gives $1 \ge \sum_{j=1}^M \mu(B(y_j,\delta)) \gtrsim_\varepsilon M \delta^{\alpha+\varepsilon}$, so $M \lesssim_\varepsilon \delta^{-\alpha-\varepsilon}$.
Since $\supp\mu \subset \bigcup_{j=1}^M B(y_j,2\delta)$, it follows that $(\supp\mu)_{\delta} \subset \bigcup_{j=1}^M B(y_j,3\delta)$.
Therefore, $|(\supp\mu)_{\delta}| \lesssim M \delta^d$, which yields \eqref{supp-upper}.

If $\mu$ is near $\alpha$-AD regular, \eqref{frostman} implies $\dim_{\mathcal{H}}(\supp\mu) \ge \alpha$.
Also, \eqref{supp-upper} gives $\overline{\dim}_{\mathcal{M}}(\supp\mu )\le \alpha$.
Combining these inequalities with the inequality $\dim_{\mathcal{H}}(\supp\mu )\le \underline{\dim}_{\mathcal{M}}(\supp\mu)$, we obtain
\[
\alpha\le \dim_{\mathcal{H}}\bigl(\operatorname{supp} \mu\bigr)  \le \underline{\dim}_{\mathcal{M}}\bigl(\operatorname{supp}\mu\bigr)  \le   \overline{\dim}_{\mathcal{M}}(\supp\mu )\le \alpha.
\]
This gives \eqref{eq:measure-dim-2}.
\end{proof}

We now prove Theorem \ref{thm:conv-geo-i} using Proposition \ref{prop:conv-geo-factorization-i}.

\begin{proof}[Proof of Theorem \ref{thm:conv-geo-i}]
Fix $0<\beta\le \alpha<d$.
Proposition~\ref{prop:conv-geo-factorization-i} provides probability measures
$\mu,{\aracc{\mu}}$, and ${\rdacc{\mu}}$ such that
\[
\mu={\aracc{\mu}}\ast{\rdacc{\mu}}\in \mathcal{P}_{\alpha,\beta},
\qquad
\dim_{\mathcal H}(\operatorname{supp}\mu)=\alpha.
\]
Moreover, ${\aracc{\mu}}$ is near $(\alpha-\beta)$-AD-regular, and ${\rdacc{\mu}}$ is near $\beta$-AD-regular with $\beta/2$-Fourier decay.

Set
\[
\Gamma_\mu f:=f*\mu.
\]
Theorem~\ref{suff}, applied to every
\(\nu\in\mathcal P_{\alpha,\beta}(\mathbb R^d)\), gives
\(\Delta_{\alpha,\beta}\subset\mathfrak U_{\alpha,\beta}\).
Since \(\mu\in\mathcal P_{\alpha,\beta}(\mathbb R^d)\), the definition of
\(\mathfrak U_{\alpha,\beta}\) also gives
\(\mathfrak U_{\alpha,\beta}\subset\mathfrak T(\mu)\).
Thus, it remains to prove that
\(\mathfrak T(\mu)\subset\Delta_{\alpha,\beta}\), equivalently, that
\(\Gamma_\mu:L^p\to L^q\) can be bounded only if
\((1/p,1/q)\in\Delta_{\alpha,\beta}\).

Since $\Gamma_\mu$ is translation invariant, boundedness $\|\Gamma_\mu\|_{L^p \to L^q} < \infty$ necessarily requires that $p \le q$.
Thus, by duality, it suffices to show that $\|\Gamma_\mu\|_{L^p \to L^q} < \infty$ implies
\begin{equation}
\label{pq-con}
\beta + \frac{d-\alpha}{q} \ge \frac{d-\alpha+\beta}{p}.
\end{equation}
Indeed, $\|\Gamma_\mu\|_{L^p \to L^q} = \|\Gamma_\mu\|_{L^{q'} \to L^{p'}}$ by duality.
By replacing $p$ and $q$ with $q'$ and $p'$, respectively, \eqref{pq-con} yields an additional necessary condition
\begin{equation}
\label{pq-con1}
\frac{d-\alpha+\beta}{q}
\ge
\frac{d-\alpha}{p}.
\end{equation}
Note that equality in \eqref{pq-con} (resp.~\eqref{pq-con1}) holds when $(1/p,1/q)$ lies on the line segment $[C_{\alpha,\beta},(1,1)]$ (resp.~$[(0,0),C_{\alpha,\beta}]$) in Figure~\ref{fig1}.
The three conditions \eqref{pq-con}, \eqref{pq-con1}, and \(p\le q\)
are precisely the supporting half-space conditions defining
\(\Delta_{\alpha,\beta}\).

We now proceed to show \eqref{pq-con} assuming $\|\Gamma_\mu\|_{L^p \to L^q} < \infty$, i.e., \eqref{conv}.
Consider
\[
f_{\delta} := \mathbbm{1}_{(\supp {\aracc{\mu}} -\supp {\aracc{\mu}})_{2\delta}},
\]
with which we test the inequality \eqref{conv}.
We first observe that
\begin{equation}
\label{lower0}
(f_{\delta}*\mu)(x)  \gtrsim_{\varepsilon} \delta^{\beta+\varepsilon}
\end{equation}
for $\varepsilon>0$ whenever $x\in (\supp\mu)_{\delta}$.
Indeed, since \({\aracc{\mu}}\) and \({\rdacc{\mu}}\) are compactly supported
positive measures and \(\mu={\aracc{\mu}}\ast{\rdacc{\mu}}\), we have
\(\operatorname{supp}\mu
=\operatorname{supp}{\aracc{\mu}}+\operatorname{supp}{\rdacc{\mu}}\).
For $x\in (\supp\mu)_{\delta}$, we can choose $y_0 \in \supp {\aracc{\mu}}$ and $z_0 \in \operatorname{supp}{\rdacc{\mu}}$ such that $|x-(y_0+z_0)|<\delta$.
Thus, if $z \in B(z_0,\delta)$ and $y\in \supp {\aracc{\mu}}$,
\[
x-(y+z) = (y_0-y) + \big(x-(y_0+z_0)\big) + (z_0-z) \in (\supp {\aracc{\mu}}-\supp {\aracc{\mu}})_{2\delta}.
\]
Consequently, if $x \in (\operatorname{supp}\mu)_{\delta}$, $ f_{\delta}( x-(y+z))=1$ for $z \in B(z_0,\delta)$ and $y\in \supp {\aracc{\mu}}$.
Recalling $\mu={\aracc{\mu}}\ast {\rdacc{\mu}}$, we note $ (f_{\delta} * \mu)(x) = \iint f_{\delta}(x-(y+z)) \, d{\aracc{\mu}}(y)\, d{\rdacc{\mu}}(z)$.
Hence,
\begin{align*}
(f_{\delta} * \mu)(x)
\ge \int_{B(z_0,\delta)} \int 1 \, d{\aracc{\mu}}(y)\, d{\rdacc{\mu}}(z)
= {\rdacc{\mu}}(B(z_0,\delta))
\end{align*}
whenever $x \in (\operatorname{supp}\mu)_{\delta}$.
Therefore, \eqref{lower0} follows from the near $\beta$-AD regularity of ${\rdacc{\mu}}$.

Since $\mu$ is $\alpha$-Frostman and $\mu(\mathbb{R}^d)=1$, from \eqref{supp-lower} in Lemma \ref{lem:element}, we have $|(\operatorname{supp}\mu)_{\delta}|\gtrsim \delta^{d-\alpha}$.
Combining this and \eqref{lower0} yields $\|f_{\delta} * \mu\|_{L^q} \gtrsim_{\varepsilon} \delta^{\beta + (d-\alpha)/q + \varepsilon}$.
On the other hand, by \eqref{eq:small-difference-set-i} in Proposition \ref{prop:conv-geo-factorization-i}, we have
\[
\|f_{\delta}\|_{L^p}
= |(\supp {\aracc{\mu}}-\supp {\aracc{\mu}})_{2\delta}|^{1/p}
\lesssim_{\varepsilon} \delta^{(d-\alpha+\beta-\varepsilon)/p}.
\]
Therefore, we obtain
\[
\|\Gamma_\mu \|_{L^p\to L^q}\ge \|f_{\delta} * \mu\|_{L^q}/\|f_{\delta}\|_{L^p} \gtrsim_\epsilon   \delta^{\beta + (d-\alpha)/q + \varepsilon-(d-\alpha+\beta-\varepsilon)/p}.
\]
Letting $\delta \to 0$, we conclude
\[
\beta + \frac{d-\alpha}{q} + \varepsilon - \frac{d-\alpha+\beta-\varepsilon}{p} 
\ge 0.
\]
Finally, letting $\varepsilon\to 0$ gives \eqref{pq-con} as desired.
\end{proof}

\subsection{Nongeometric case: Proof of Theorem \ref{thm:conv-ngeo-i}}

We begin by proving the following lemma, which relates the Hausdorff dimension of $\operatorname{supp}\mu$ to the \(L^p\)--\(L^q\) boundedness of the convolution operator.

\begin{lem}[Dimension test for convolution]
\label{lem:conv-dim-test}
Let \(\mu\in\mathcal P(\mathbb R^d)\).
If $\dim_{\mathcal{H}}(\operatorname{supp}\mu) = \beta$, then \eqref{conv} holds only if $(1/p, 1/q) \in \Delta_{\beta, \beta}$.
\end{lem}

The proof of this lemma is based on an argument in \cite{KFK09} (also, see \cite{Oberlin03}).

\begin{proof}[Proof of Lemma~\ref{lem:conv-dim-test}]
Recall the definition of $\Delta_{\beta,\beta}$.
As in the proof of Theorem \ref{thm:conv-geo-i}, translation invariance gives
the necessary condition \(p\le q\). We first prove that, for \(1\le q<\infty\),
\eqref{conv} implies
\begin{equation}
\label{eq:conv-dim-necessary}
\beta + (d - \beta)/q \ge d/p
\end{equation}
The endpoint \(q=\infty\) will be handled directly by the same cube test below. Once \eqref{eq:conv-dim-necessary} is established for $1\le q\le \infty$,  by duality, one also obtains $d/q \ge (d - \beta)/p$ for $1 \le p \le \infty$. Together with the condition $p \le q$, these inequalities imply that $(1/p, 1/q)\in \Delta_{\beta, \beta}$.

For $\delta\in (0,1)$, let $V= [0, \delta)^d$ and
\[
\mathcal{Q}_\delta
=
\{\delta \mathbf m + V : \mathbf m \in \mathbb{Z}^d\}.
\]
We first show that the estimate \eqref{conv} implies
\begin{equation}
\label{eq:conv-dim-Sk}
\sum_{Q \in \mathcal{Q}_\delta} \mu(Q)^q \lesssim \delta^{d(q/p-1)}
\end{equation}
for $\delta \in (0,1)$.
To show this, we test the estimate \eqref{conv} with $ f = \mathbbm{1}_{V - V}$. Let \(Q \in \mathcal{Q}_\delta\).
If \(x \in Q\), then \(x-Q \subset V-V\), and hence $f(x-y)\ge \mathbbm{1}_Q(y)$.
Thus, it follows that $ (\mu * f)(x)\ge \mu(Q)$ for \(x \in Q\).
Integrating over $\mathbb R^d=\bigcup_{Q\in \mathcal Q_\delta} Q$, we have
\[
\|\mu * f\|_q^q
\ge \sum_{Q \in \mathcal{Q}_\delta} \int_Q (\mu * f)(x)^q\,dx
\ge \delta^{d} \sum_{Q \in \mathcal{Q}_\delta} \mu(Q)^q.
\]
Since \(\|f\|_p \sim \delta^{d/p}\), the estimate \eqref{conv} yields \eqref{eq:conv-dim-Sk}.

Fix \(s>\beta= \dim_{\mathcal{H}}(\operatorname{supp}\mu) \).
Then, for every \(N\), there exists a collection $\mathcal C$ of dyadic cubes covering \(\operatorname{supp}\mu\) such that
\[
|Q|^{1/d}  \le 2^{-N} \quad \forall Q \in \mathcal{C},
\qquad
\sum_{Q \in \mathcal{C}} |Q|^{s/d} < 1.
\]
Note $1\le \sum_{Q \in \mathcal{C}} \mu(Q)$.
Using the above inequality, by Hölder's inequality we have
\begin{align*}
1
\le
\Big(
    \sum_{Q \in \mathcal{C}} \mu(Q) 
\Big)^{q}
\le
\sum_{Q \in \mathcal{C}} \mu(Q)^q  |Q|^{-s(q-1)/d}.
\end{align*}
Regrouping the cubes in \(\mathcal{C}\) according to their side length $2^{-k}$ and using \eqref{eq:conv-dim-Sk} with $\delta=2^{-k}$, we have
\begin{align*}
1
\le
\sum_{k \ge N}
    2^{ks(q-1)}
    \sum_{Q \in \mathcal{C} \cap \mathcal{Q}_{2^{-k}}} \mu(Q)^q
&\lesssim \sum_{k \ge N} 2^{-k(d(q/p-1)-s(q-1))}.
\end{align*}
If \(d(q/p-1) > s(q-1)\), the last sum tends to \(0\) as \(N\to\infty\), which is a contradiction.
Hence, we obtain $d\left(1/{p}-1/q\right) \le s(1-1/q)$.
Letting \(s \to \beta\) gives \eqref{eq:conv-dim-necessary}.

The same test also gives the endpoint \(q=\infty\). If \(q=\infty\) and
\(1\le p\le \infty\), then for every \(Q\in\mathcal Q_\delta\),  $
    \mu(Q)\le \|\mu*f\|_\infty \lesssim \|f\|_p\sim \delta^{d/p}.$
Since every ball of radius \(r\) meets only \(O_d(1)\) cubes of side length
\(r\), it follows that
\[
    \mu(B(x,r))\lesssim r^{d/p},\qquad x\in\mathbb R^d,\quad 0<r<1.
\]
By the mass distribution principle,
\(\dim_{\mathcal H}(\operatorname{supp}\mu)\ge d/p\), and hence
\(\beta\ge d/p\). This is precisely \eqref{eq:conv-dim-necessary} when
\(q=\infty\).  
\end{proof}

We now prove Theorem \ref{thm:conv-ngeo-i} making use of Proposition \ref{prop:heavy-core-i} and Lemma \ref{lem:conv-dim-test}.

\begin{proof}[Proof of Theorem~\ref{thm:conv-ngeo-i}]
Let \(s = 2\alpha-\beta\).
By Proposition~\ref{prop:heavy-core-i}, there exist a measure \(\mu\in \mathcal P_{\alpha, \beta}(\mathbb R^d)\) with $\dim_{\mathcal{H}}(\operatorname{supp}\mu) = \beta$ and a near $s$-AD regular set \(F \subset \operatorname{supp}\mu\) such that \eqref{eq:nearF} holds.

Theorem~\ref{suff}, applied to every
\(\nu\in\mathcal P_{\alpha,\beta}(\mathbb R^d)\), gives
\(\mathrm{Pent}_{\alpha,\beta}\subset\mathfrak U_{\alpha,\beta}\).
Since \(\mu\in\mathcal P_{\alpha,\beta}(\mathbb R^d)\), the definition of
\(\mathfrak U_{\alpha,\beta}\) gives
\(\mathfrak U_{\alpha,\beta}\subset\mathfrak T(\mu)\).
Thus, it remains to prove that
\(\mathfrak T(\mu)\subset\mathrm{Pent}_{\alpha,\beta}\).
The line through \(C_{\alpha,\beta}\) and \(D_\beta\) is given by $
    \alpha+(d-2\alpha+\beta)b=da. $
Reflecting across the anti-diagonal, we obtain
\[
\operatorname{Pent}_{\alpha,\beta}
=
\{(a,b)\in\Delta_{\beta,\beta}:
\alpha+(d-2\alpha+\beta)b\ge da,\ 
\beta-\alpha+db\ge(d-2\alpha+\beta)a\}.
\]
Since $\dim_{\mathcal{H}}(\operatorname{supp}\mu) = \beta$, by Lemma~\ref{lem:conv-dim-test} $(1/p, 1/q)$ must belong to $\Delta_{\beta,\beta}$ for \eqref{conv} to hold.
Therefore, by duality, it is sufficient to prove that \eqref{conv} holds only if
\begin{equation}
\label{eq:conv-ngeo-necessary}
\alpha + (d - 2\alpha + \beta)/q \ge d/p.
\end{equation}

To this end, let us consider
\[
f_{\delta} = \mathbbm{1}_{B(0,2\delta)}.
\]
If \(x \in (F)_{\delta}\), there is \(x_0 \in F\) such that \(|x-x_0|<\delta\).
Clearly, \(B(x_0,\delta) \subset B(x,2\delta)\).
Hence $(\mu * f_{\delta})(x) = \mu(B(x,2\delta)) \ge \mu(B(x_0,\delta))$.
By \eqref{eq:nearF}, it follows that
\[
(\mu * f_{\delta})(x)
\gtrsim_{\varepsilon}
\delta^{\alpha+\varepsilon}, \quad  x \in (F)_{\delta}.
\]
Let \(\nu\) be a near \((2\alpha-\beta)\)-AD regular probability measure
with \(\operatorname{supp}\nu=F\). Applying \eqref{supp-lower} in
Lemma~\ref{lem:element} to \(\nu\), we get  $
    |(F)_\delta|\gtrsim \delta^{d-2\alpha+\beta}.
$
Hence, we obtain $
\|\mu * f_{\delta}\|_q
\gtrsim_{\varepsilon}
\delta^{\alpha + (d-2\alpha+\beta)/q + \varepsilon},
$
while $\|f_{\delta}\|_p \sim \delta^{d/p}$.
Therefore
\[
\|\Gamma_\mu \|_{L^p\to L^q}
\ge
{\|\mu * f_{\delta}\|_q}/{\|f_{\delta}\|_p}
\gtrsim_{\varepsilon}
\delta^{\alpha + (d-2\alpha+\beta)/q - d/p + \varepsilon}.
\]
Letting $\delta \to 0$ and then $\varepsilon \to 0$, we obtain \eqref{eq:conv-ngeo-necessary}.
\end{proof}
 \section{Basic Construction and Preliminary Tools}
\label{sec3}

In this section, we develop the basic framework used in the subsequent
constructions and prove Theorem~\ref{thm:near-AD-Salem}.
The argument has three components: a deterministic coding framework and
a criterion for near AD regularity; an almost-sure Fourier-decay
criterion based on a refinement identity and conditional Hoeffding
inequalities; and an AD-regular sampling lemma that allows the two
criteria to be satisfied simultaneously.

\subsection{Setup}
\label{subsec:basic-setup}

Let $(M_n)_{n\ge1}$ be a nondecreasing sequence of positive integers with $M_n\ge2$, to be chosen later according to the requirements of each construction.
We assume that the sequence $(M_n)_{n\ge1}$ satisfies the growth conditions
\begin{equation}
\label{eq:growth-condition}
\lim_{n\to\infty} M_n=\infty, \quad \log M_n=o(n).
\end{equation}
Define the associated sequence $({\mathfrak M}_n)_{n\ge 0}$ by
\begin{equation}
\label{eq:NM}
\mathfrak M_n=\prod_{k=1}^n M_k \ \ (n\ge1), \quad \mathfrak M_0=1.
\end{equation}
For the rest of the paper, we will repeatedly use the following elementary observation, which is immediate from the assumption \eqref{eq:growth-condition}.

\begin{lem}
\label{lem:absorption}
Let $C > 1$ and $\varepsilon > 0$.
Then $C^n \lesssim_\varepsilon {\mathfrak M}_n^{\varepsilon}$ and $\lim_{n\to\infty} C^{-n} M_n = 0$.
In particular, if $A_n \lesssim M_n^t$, then $\prod_{k=1}^n A_k \lesssim_\varepsilon {\mathfrak M}_n^{t + \varepsilon}$.
Similarly, if $B_n \gtrsim M_n^t$, then $\prod_{k=1}^n B_k \gtrsim_\varepsilon {\mathfrak M}_n^{t - \varepsilon}$.
\end{lem}

Next, let $r>\max\{1,\alpha/2, \beta/2\}$ be an integer.
For instance, one may take $r=d+1$.
For each $n\ge1$, define the digit sets $\mathcal{D}_n^{[r]}\subset\mathbb{R}^d$ by
\begin{equation}
\label{eq:digit}
\mathcal{D}_n^{[r]}=\{0,1,\ldots,r(M_n-1)\}^d.
\end{equation}
In other words, $\mathcal{D}_n^{[r]}$ is the $r$-fold sumset of $\mathcal{D}_n^{[1]}=\{0,1,\ldots,M_n-1\}^d$.
For simplicity, we write $\mathcal{D}_n$ in place of $\mathcal{D}_n^{[r]}$ whenever the value of $r$ does not need to be specified.
However, as in Section \ref{conv-geo} and Section \ref{rest-geo}, we will write $\mathcal{D}_n^{[r]}$ whenever it is necessary to make the dependence on $r$ explicit.

\subsubsection*{Set of words}
Let $\mathcal{W}$ denote the set of finite words whose alphabet at level $n$ is $\mathcal{D}_n^{[r]}$.
More precisely, we define
\begin{equation}
\label{eq:word-tree}
\mathcal{W} = \bigcup_{n=0}^\infty \mathcal{W}_n,
\end{equation}
where
\[
\mathcal{W}_n := \prod_{k=1}^n \mathcal{D}_k \quad (n\ge1), \quad \mathcal{W}_0 := \{\varnothing\}.
\]
Here, $\varnothing$ denotes the empty word, and $\mathcal{W}_n$ is the set of words of length $n$.
We also define
\begin{equation}
\label{eq:infinite-word}
\mathcal{W}_\infty = \prod_{n=1}^\infty \mathcal{D}_n.
\end{equation}
Elements of $\mathcal{W}_\infty$ are called infinite words.
Note that $\mathcal{W}_\infty$ is not a subset of $\mathcal{W}$.

\begin{defn}
\label{def:word}
For a finite word $w \in \mathcal{W}$, we write $|w|$ for its length.
If $w=(w_1, \dots, w_n) \in \mathcal{W}_n$ and $w_{n+1} \in \mathcal{D}_{n+1}$, we denote their concatenation by $(w,w_{n+1})=(w_1, \dots, w_n,w_{n+1}) \in \mathcal{W}_{n+1}$.
For a finite word $w=(w_1,\dots,w_n)$ and $0 \le m \le n$, we denote its prefix $(w_1,\dots,w_m)$ of length $m$ by $w|_m$, where $w|_0=\varnothing$.
Similarly, for an infinite word $\mathbf{w}=(w_1,w_2,\dots)$, we define $\mathbf{w}|_m := (w_1,\dots,w_m)$.
\end{defn}

\subsubsection*{Coding maps}
We associate each finite word $w = (w_1, \ldots, w_{|w|}) \in \mathcal{W}$ and each infinite word $\mathbf{w} = (w_n)_{n \ge 1} \in \mathcal{W}_\infty$ with points in $\mathbb R^d$ via the maps
\begin{align}
\label{eq:coding-maps-x}
    X(w) = \sum_{n=1}^{|w|} \frac{w_n}{{\mathfrak M}_n},
    \\
    \label{eq:coding-maps-p}
    \pi(\mathbf{w}) = \sum_{n=1}^\infty \frac{w_n}{{\mathfrak M}_n}.
\end{align}
We call the maps $X : \mathcal{W} \to \mathbb R^d$ and $\pi : \mathcal{W}_\infty \to \mathbb R^d$ the \emph{coding maps}.
We also define the \emph{$n$-th truncated coding map} $\pi_n : \mathcal{W}_\infty \to \mathbb R^d$ by
\begin{equation}
\label{eq:truncated-coding-map}
\pi_n(\mathbf{w}) = X(\mathbf{w}|_n) = \sum_{k=1}^n \frac{w_k}{{\mathfrak M}_k}.
\end{equation}
Since the $n$-th alphabet consists of integer vectors whose coordinates range from $0$ to $r(M_n - 1)$, a simple telescoping computation
\begin{equation}
\label{sumMk}
\sum_{k > n}^\infty \frac{r(M_k - 1)}{{\mathfrak M}_k}
= r \sum_{k > n} \left( \frac{1}{{\mathfrak M}_{k - 1}} - \frac{1}{{\mathfrak M}_k} \right)
= \frac{r}{{\mathfrak M}_n}
\end{equation}
gives
\begin{equation}
\label{eq:coding-basic-bounds}
\pi(\mathbf{w}) \in \pi_n(\mathbf{w}) + [0, {\mathfrak M}_n^{-1}\, r]^d.
\end{equation}
This, in particular, shows that the images of $X$, $\pi$, and $\pi_n$ are included in $[0, r]^d$.

\subsubsection*{Offspring system}
Suppose that $S$ is a map on $\mathcal{W}$ such that, for each $w \in \mathcal{W}$,
\[
S(w) \subset \mathcal{D}_{|w|+1}.
\]
We call such a map $S$ an \emph{offspring system on $\mathcal{W}$}.
Given an offspring system $S$, we recursively define the associated \emph{tree} $\mathcal{T}$ as follows.
First, let $\mathcal{T}_0 := \{\varnothing\}$.
Having defined $\mathcal{T}_n$, we define
\[
\mathcal{T}_{n+1}
=
\{(w,w_{n+1}) : w \in \mathcal{T}_n,\; w_{n+1} \in S(w)\}.
\]
We then define
\begin{equation}
\label{eq:associated-tree}
\mathcal{T} = \bigcup_{n=0}^\infty \mathcal{T}_n
\end{equation}
and
\[
\mathcal{T}_\infty
=
\{\mathbf{w} \in \mathcal{W}_\infty : \mathbf{w}|_n \in \mathcal{T},  \quad \forall n \ge 1\}.
\]

For each finite word $w \in \mathcal{T}$, let $[w]$ denote the set of infinite words $\mathbf{w} \in \mathcal{T}_\infty$ having $w$ as a prefix; that is,
\[
[w] = \{ \mathbf{w} \in \mathcal{T}_\infty : \mathbf{w}|_{|w|} = w \}.
\]
We call $[w]$ the \emph{cylinder} determined by $w$.
For the concatenated word $(w,a)$, we simply write $[w,a]$ for $[(w,a)]$.

\subsubsection*{Uniform offspring system and induced measure} An offspring system $S$ on $\mathcal{W}$ is said to be \emph{uniform} if, at each level, all offspring sets are nonempty and have the same cardinality.
More precisely, if $T_n$ denotes this common cardinality at level $n$ for a uniform offspring system $S$, then we have
\[
\# S(w) = T_n, \quad \forall w \in \mathcal{T}_{n - 1},
\]
for all $n\ge 1$.
In this case, we call $(T_n)_{n \ge 1}$ the \emph{profile} of the uniform offspring system $S$.

Then there exists a unique Borel probability measure $\tau$ on $\mathcal{T}_\infty$ satisfying
\begin{equation}
\label{eq:uniform-product-measure}
\tau([w]) = \frac{1}{\prod_{n=1}^{|w|} T_n}, \quad w \in \mathcal{T}.
\end{equation}
Indeed, since each level has exactly \(T_n\) offspring, the cylinder family \(\{[w]: w\in\mathcal T\}\) carries a natural consistent premeasure satisfying \eqref{eq:uniform-product-measure}.
Then, by the Carathéodory extension theorem, this premeasure extends uniquely to a
Borel probability measure on \(\mathcal T_\infty\). Alternatively, after choosing an enumeration of each offspring set
\(S(w)\), one obtains a noncanonical identification of
\(\mathcal T_\infty\) with
\(\prod_{n=1}^\infty\{1,\ldots,T_n\}\), under which \(\tau\) is the
product of the uniform probability measures on the factors.
Thus, we call the measure $\tau=\tau(S)$ the \emph{product measure} associated with the uniform offspring system $S$.

By a harmless abuse of notation, we use the same symbol $\tau$ for its trivial extension from $\mathcal{T}_\infty$ to $\mathcal{W}_\infty$, defined by
\[
\tau(E) = \tau(E \cap \mathcal{T}_\infty).
\]

Finally, let $\mu$ be the pushforward of $\tau$ under the coding map $\pi$:
\begin{equation}
\label{eq:induced-measure}
\mu = \pi_\sharp \tau.
\end{equation}
Then, since the image of $\pi$ is included in $[0, r]^d$, $\mu$ is a Borel probability measure supported on $[0, r]^d$.
We call this $\mu$ the \emph{measure induced by the offspring system $S$}.
We also define $\mu_n$ to be the pushforward of $\tau$ under the truncated coding map $\pi_n$:
\begin{equation}
\label{eq:induced-measuren}
\mu_n = (\pi_n)_\sharp \tau.
\end{equation}
Since $\pi_n$ maps each level-$n$ cylinder $[w]$ to a single point $X(w)$, we have
\begin{equation}
\label{eq:truncated-induced-measure}
\mu_n = \frac{1}{\prod_{k=1}^n T_k} \sum_{w \in \mathcal{T}_n} \delta_{X(w)},
\end{equation}
where \(\delta_x\) denotes the unit point mass at \(x\).
Note that $\mu_n$ converges to $\mu$ weakly since $\pi_n$ converges uniformly to $\pi$.

Having set up the basic construction, we next investigate how the geometric and analytic properties of the induced measure \(\mu\) depend on the parameters \((M_n)_{n\ge1}\) and \((T_n)_{n\ge1}\).
The sequence \((M_n)\) determines the spatial scales of the coding, while \((T_n)\) describes the branching profile of the underlying tree.
Together, they govern the mass distribution, regularity, and Fourier behavior of \(\mu\).

\subsection{Near AD regularity of the induced measure}

In this subsection, we prove Proposition \ref{prop:near-AD-criterion}, which provides a criterion for $\mu$ to be near $\alpha$-AD regular.
We first collect a couple of lemmas for our purpose.

\begin{lem}
\label{lem:multiplicity}
Let $X$ be defined by \eqref{eq:coding-maps-x}.
Then for every $n \ge 0$,
\begin{equation}
\label{eq:multiplicity}
\sup_{h \in {\mathfrak M}_n^{-1} \mathbb{Z}^d} \# \Big\{ w \in \mathcal{W}_n : X(w) = h \Big\} \le r^{dn}.
\end{equation}
\end{lem}
\begin{proof}
We prove this by induction on \(n\).
The case \(n=0\) is trivial, since $\mathcal{W}_0 = \{\varnothing\}$.

Now, assume that \eqref{eq:multiplicity} holds for some $n\ge 0$.
We need only to show
\begin{equation}  \label{eq:n+1}    \# \Big\{ z \in \mathcal{W}_{n + 1} : X(z) = {\mathfrak M}_{n + 1}^{-1}\, a \Big\}
\le r^{d(n+1)}
\end{equation}
for all \(a \in \mathbb{Z}^d\).
To this end, we write $z = (w, w_{n + 1}) \in \mathcal{W}_{n + 1}$ with \(w \in \mathcal{W}_{n}\) and \(w_{n+1} \in \mathcal{D}_{n + 1}\).
If $X(z) = {\mathfrak M}_{n + 1}^{-1}\, a$, by \eqref{eq:coding-maps-x} we have $ {\mathfrak M}_{n + 1}^{-1}\, a = X(w) + {\mathfrak M}_{n + 1}^{-1}\, w_{n + 1} $.
Hence, from \eqref{eq:NM} it follows that $w_{n + 1} \equiv a \bmod M_{n + 1}$, since $X(w)\in {\mathfrak M}_n^{-1} \mathbb{Z}^d$.
For each fixed \(w_{n+1}\), the induction hypothesis
\eqref{eq:multiplicity} bounds the number of possible prefixes \(w\)
by \(r^{dn}\). Therefore,
\[
\# \Big\{ z \in \mathcal{W}_{n + 1} : X(z) = {\mathfrak M}_{n + 1}^{-1}\, a \Big\}
\le r^{dn} \#\Big\{ w_{n + 1} \in \mathcal{D}_{n + 1} : w_{n + 1} \equiv a \bmod M_{n + 1} \Big\}.
\]
For each coordinate \(1 \le j \le d\), among the integers \(\{0,1,\dots,r(M_{n+1}-1)\}\), there are at most \(r\) numbers congruent to \(a_j \bmod{M_{n + 1}}\).
Thus,
\[
\#\Big\{ w_{n + 1} \in \mathcal{D}_{n + 1} : w_{n + 1} \equiv a \bmod M_{n + 1} \Big\} \le r^d.
\]
Consequently, \eqref{eq:n+1} follows.
\end{proof}

\begin{lem}
\label{lem:full-support}
Let $\mu$ be the induced measure given by \eqref{eq:induced-measure}.
Let \(\mathbf{w} \in \mathcal{T}_\infty\) and \(n \ge 0\).
Then we have
\begin{equation}    \label{eq:prod}
\mu\bigl(\pi(\mathbf{w}) + {\mathfrak M}_n^{-1} [-r,  r]^d\bigr)
\ge \big(\prod_{k=1}^n T_k\big)^{-1}.
\end{equation}
Moreover, $\operatorname{supp}\mu = \pi(\mathcal{T}_\infty)$.
\end{lem}
\begin{proof}
Let \(w = \mathbf{w}|_n\).
If \(\mathbf{v} \in [w]\), then \(\mathbf{v}|_n = w\).
Hence, by \eqref{eq:coding-basic-bounds}, $\pi(\mathbf{v}), \pi(\mathbf{w}) \in X(w) + [0, {\mathfrak M}_n^{-1}\, r]^d$.
Thus, it follows that
\[
\pi(\mathbf{v}) \in \pi(\mathbf{w}) + {\mathfrak M}_n^{-1} [-r,  r]^d.
\]
Recalling \(\mu = \pi_\sharp \tau\), we have $ \mu (\pi(\mathbf{w}) + {\mathfrak M}_n^{-1} [-r, r]^d) \ge \tau([w])$.
Therefore, \eqref{eq:prod} follows from \eqref{eq:uniform-product-measure}.

We now verify that $\operatorname{supp}\mu = \pi(\mathcal{T}_\infty)$.
Suppose \(x = \pi(\mathbf{w})\) for some $\mathbf w\in \mathcal{T}_\infty$ and \(U\) is any open neighborhood of \(x\).
Since $x + {\mathfrak M}_n^{-1} [-r, r]^d \subset U$ for sufficiently large \(n\), we have $ \mu(U) \ge \mu\bigl(x + {\mathfrak M}_n^{-1} [-r, r]^d\bigr) > 0$ by \eqref{eq:prod}.
Thus \(x \in \operatorname{supp}\mu\), and so $\pi(\mathcal{T}_\infty)\subset \operatorname{supp}\mu$.

For the reverse inclusion, note that each \(\mathcal{D}_n\) is finite, so \(\mathcal{W}_\infty = \prod_{n \ge 1} \mathcal{D}_n\) is compact, and
\[
\mathcal{T}_\infty
= \bigcap_{n \ge 1} \{ \mathbf{w} \in \mathcal{W}_\infty : \mathbf{w}|_n \in \mathcal{T}_n \}
\]
is closed; hence \(\mathcal{T}_\infty\) is compact.
Since \(\pi_n \to \pi\) uniformly by \eqref{eq:coding-basic-bounds} and each \(\pi_n\) is continuous, it follows that \(\pi\) is continuous.
Therefore, \(\pi(\mathcal{T}_\infty)\) is compact, and in particular closed.
Since \(\pi^{-1}(\mathbb{R}^d \setminus \pi(\mathcal{T}_\infty))\cap \mathcal{T}_\infty = \emptyset\) and $\tau(\mathcal{W}_\infty \setminus \mathcal{T}_\infty)=0$, $\mu(\mathbb{R}^d \setminus \pi(\mathcal{T}_\infty)) = \tau\!\left(\pi^{-1}(\mathbb{R}^d \setminus \pi(\mathcal{T}_\infty))\right) = 0$.
If \(x \notin \pi(\mathcal{T}_\infty)\), $x$ has an open neighborhood \(U \subset \mathbb{R}^d \setminus \pi(\mathcal{T}_\infty)\) and \(\mu(U)=0\).
Thus \(x \notin \operatorname{supp}\mu\).
Therefore, \(\operatorname{supp}\mu \subset \pi(\mathcal{T}_\infty)\).
\end{proof}

The following proposition states that if each offspring set is AD regular in the discrete sense, then the induced measure $\mu$ is near AD regular.

\begin{prop}
\label{prop:near-AD-criterion}
Let \(\mu\) be the measure induced by a uniform offspring system \(S\) on \(\mathcal W\), with digit parameters \((M_n)_{n\ge1}\) and profile \((T_n)_{n\ge1}\).
Suppose that there exists a constant $C_0 \ge 1$ such that
\begin{equation}
\label{eq:branch-scale}
r^d M_n^\alpha \le T_n \le C_0 r^d M_n^\alpha, \quad \forall n \ge 1;
\end{equation}
and, for every $n \ge 1$ and $w \in \mathcal{T}_{n - 1}$,
\begin{align}
        \#(S(w) \cap B(y, R)) \ge C_0^{-1} R^\alpha, \qquad& \forall y \in S(w), \ \forall R \in [1, M_n]; \label{eq:discrete-lower}\\
        \#(S(w) \cap B(y, R)) \le C_0 R^\alpha, \qquad & \forall y \in \mathbb{R}^d, \ \forall R \in [1, M_n] \label{eq:discrete-upper}.
\end{align}
Then $\mu$ is near $\alpha$-AD regular.
\end{prop}

\begin{proof}
We have to show that $\mu$ satisfies \eqref{frostman} and \eqref{lower}.
We verify \eqref{lower} first.

Let $(x, \rho)\in \supp \mu\times (0, 1)$.
By Lemma~\ref{lem:full-support}, we may write $x = \pi(\mathbf{w})$ for some $\mathbf{w} \in \mathcal{T}_\infty$.
Choose $n \ge 1$ such that
\[
{4r\sqrt{d}}\,\, {{\mathfrak M}_n^{-1}}  \le \rho < {4r\sqrt{d}}\,\, {\mathfrak M}_{n-1}^{-1}
\]
and set $w = \mathbf{w}|_{n-1}$, $(w, w_n) = \mathbf{w}|_n$.
Also set
\[
R = \frac{\rho\, {\mathfrak M}_n} {4r\sqrt{d}} \in [1,M_n].
\]

Let $a \in S(w) \cap B(w_n,R)$, and let $\mathbf{v} \in [(w,a)]$.
By \eqref{eq:truncated-coding-map} we have $\pi_n(\mathbf{v}) - \pi_n(\mathbf{w})=(a-w_n)/{\mathfrak M}_n.$
Since $a\in B(w_n,R)$, it follows that
\[
|\pi_n(\mathbf{v}) - \pi_n(\mathbf{w})| \le R/{\mathfrak M}_n=\rho/(4r\sqrt d).
\]
Moreover, \eqref{eq:coding-basic-bounds} gives
\[
|\pi(\mathbf w)-\pi_n(\mathbf w)| \le r\sqrt d\,{\mathfrak M}_n^{-1} \le \rho/4.
\]
On the other hand, \eqref{eq:coding-basic-bounds} gives $|\pi(\mathbf{v}) - \pi_n(\mathbf{v})|\le {r\sqrt{d}}/{{\mathfrak M}_n}.$
Thus, we see that
\begin{align*}
|\pi(\mathbf{v}) - x|
&\le 
|\pi(\mathbf{v}) - \pi_n(\mathbf{v})|
    + |\pi_n(\mathbf{v}) - \pi_n(\mathbf{w})|
    + |\pi_n(\mathbf{w}) - \pi(\mathbf{w})| <\rho.
\end{align*}
Hence $B(x,\rho)$ contains the cylinder image $\pi([(w,a)])$ for every $a \in S(w) \cap B(w_n,R)$.

Therefore, recalling that the measure $\mu$ is induced by the offspring system $S$ (see \eqref{eq:induced-measure}), we have
\[
\mu(B(x,\rho))= \tau(\pi^{-1}(B(x,\rho))) \ge \sum_{a \in S(w) \cap B(w_n,R)}  \tau([(w,a)]).
\]
Using \eqref{eq:uniform-product-measure}, \eqref{eq:discrete-lower}, and \eqref{eq:branch-scale} successively, we obtain
\[
\mu(B(x,\rho))
\ge C_0^{-1} R^\alpha {(\prod_{k=1}^n T_k)^{-1}}
\ge C_0^{-1} R^\alpha (C_0 r^d)^{-n} {\mathfrak M}_n^{-\alpha}.
\]
Since $R \sim \rho\, {\mathfrak M}_n$, this gives $\mu(B(x,\rho)) \gtrsim (C_0 r^d)^{-n} \rho^\alpha$.
Using Lemma~\ref{lem:absorption}, we obtain $\mu(B(x,\rho)) \gtrsim_{\varepsilon} \rho^\alpha \mathfrak M_{n-1}^{-\varepsilon}$ for every \(\varepsilon>0\).
Since \(\rho < 4r\sqrt d\,\mathfrak M_{n-1}^{-1}\), the inequality \eqref{lower} follows.

Next, we verify \eqref{frostman}.
Let $(x,\rho)\in \mathbb R^d\times (0,1)$.
We choose $n \ge 1$ such that $ {\mathfrak M}_n^{-1} \le \rho < {\mathfrak M}_{n-1}^{-1} $, so we have $1\le \rho\, {\mathfrak M}_n < M_n.$
Consider
\begin{equation}
\label{eq:parents}
H = \Bigl\{h \in {\mathfrak M}_{n-1}^{-1}\mathbb{Z}^d :
\bigl(h + {\mathfrak M}_{n-1}^{-1} [0, r ]^d\bigr) \cap B(x, \rho) \neq \emptyset \Bigr\}.
\end{equation}
Setting \(C_1=r\sqrt d+1\), for each $h \in H$, we let
\begin{equation}
\label{eq:children}
A_h = \{ a \in \mathcal{D}_n : h +  {\mathfrak M}_n^{-1}\, a \in B(x, 	C_1\rho) \}.
\end{equation}

Let $\mathbf{v} \in \mathcal{T}_\infty \cap \pi^{-1}(B(x,\rho))$.
Since $\pi(\mathbf v)\in \pi_{n-1}(\mathbf{v}) + [0, {\mathfrak M}_{n-1}^{-1}\, r]^d$ by \eqref{eq:coding-basic-bounds}, it follows that $\pi_{n-1}(\mathbf{v}) \in H$.
If we write
\[
\mathbf{v}|_{n-1} = v,  \quad \mathbf{v}|_n = (v,a),
\]
it is clear that $v\in \mathcal T_{n-1}$ and $a\in S(v)$.
If $\pi_{n-1}(\mathbf{v}) = X(v) = h$, we have $a \in A_h$.
Indeed, note that $h+\mathfrak M_n^{-1}a=\pi_n(\mathbf v)$.
Since $\mathbf{v} \in \pi^{-1}(B(x,\rho))$ and \(\rho\ge \mathfrak M_n^{-1}\),
\[
|h+\mathfrak M_n^{-1}a-x|
=
|\pi_n(\mathbf v)-x|
\le
|\pi_n(\mathbf v)-\pi(\mathbf v)|+|\pi(\mathbf v)-x|
\le
\frac{r\sqrt d}{\mathfrak M_n}+\rho
\le C_1\rho.
\]
Hence, this shows $a \in A_h$.

Combining all those together, we obtain
\begin{equation}    \label{ttt}
\mathcal{T}_\infty \cap \pi^{-1}(B(x,\rho))
\subset
\bigcup_{h \in H}
\bigcup_{v \in \mathcal{T}_{n-1} \cap X^{-1}(h)}
\bigcup_{a \in A_h \cap S(v)}
[(v,a)].
\end{equation}
Since $ \mu(B(x,\rho)) =\tau(\pi^{-1}(B(x,\rho)))$, it follows that
\begin{equation}    \label{muB}
\mu(B(x,\rho))
\le
\sum_{h\in H}  \,
\sum_{v\in \mathcal T_{n-1}\cap X^{-1}(h)}  \,
\sum_{a\in A_h\cap S(v)}
\tau([v,a]).
\end{equation}
Since $\rho < {\mathfrak M}_{n-1}^{-1}$, $H \subset B\big(x, {\mathfrak M}_{n-1}^{-1}(r\sqrt{d} + 1)\big)$ and hence $\#H \lesssim 1$.
Moreover, since $A_h \subset B\bigl({\mathfrak M}_n(x - h), C_1\rho\, {\mathfrak M}_n\bigr)$,
if \(C_1\rho\mathfrak M_n\le M_n\), then
\eqref{eq:discrete-upper} gives
\(\#(A_h\cap S(v))\le C_0(C_1\rho\mathfrak M_n)^\alpha\).
If \(C_1\rho\mathfrak M_n>M_n\), then the trivial bound and the
upper bound in \eqref{eq:branch-scale} give
\(\#(A_h\cap S(v))\le T_n\le C_0r^dM_n^\alpha
\le C_0r^d(C_1\rho\mathfrak M_n)^\alpha\).
Thus, in either case,
\(\#(A_h\cap S(v))\lesssim(\rho\mathfrak M_n)^\alpha\).
Therefore, using \eqref{eq:uniform-product-measure} and Lemma~\ref{lem:multiplicity},
\[
\mu(B(x,\rho))
\lesssim
\frac{(\#H)\,r^{d(n-1)}(\rho\mathfrak M_n)^\alpha}
{\prod_{k=1}^nT_k}.
\]
Since $T_k \ge r^d M_k^\alpha$ by \eqref{eq:branch-scale}, it follows that $\mu(B(x,\rho)) \lesssim \rho^\alpha$.
\end{proof}

The preceding argument also yields a more precise statement in terms
of the lower regularity dimension. Following
Käenmäki--Lehrbäck--Vuorinen \cite{KLV13}, we define the lower
regularity dimension \(\underline{\dim}_{\mathrm{reg}}\mu\) of a
locally finite Borel measure \(\mu\) to be the supremum of all
\(s\ge0\) such that
\begin{equation}
\label{eq:reg}
\mu(B(x,\rho))
\lesssim
\left(\frac{\rho}{R}\right)^s
\mu(B(x,R))
\end{equation}
for \(x\in\supp\mu\) and \(0<\rho<R<1\)
(see also \cite{KL17}).

\begin{prop}
\label{prop:lower-reg-dim}
Under the hypotheses of Proposition~\ref{prop:near-AD-criterion}, we have
\[
\underline{\dim}_{\mathrm{reg}}\mu=\alpha.
\]
\end{prop}

\begin{proof}
We first show that
\(\underline{\dim}_{\mathrm{reg}}\mu\le\alpha\).
Suppose that \eqref{eq:reg} holds for some \(s>\alpha\).
Fix \(x\in\operatorname{supp}\mu\) and \(R_0\in(0,1)\), and choose
\(\varepsilon\in(0,s-\alpha)\). Then \eqref{eq:reg} and
\eqref{lower} give
\[
\rho^{\alpha+\varepsilon}
\lesssim_\varepsilon
\mu(B(x,\rho))
\lesssim
\rho^s,
\qquad 0<\rho<R_0,
\]
which is impossible as \(\rho\downarrow0\). Hence
\(\underline{\dim}_{\mathrm{reg}}\mu\le\alpha\).

For the reverse inequality, it remains to prove \eqref{eq:reg} with
\(s=\alpha\)
for every \(x\in\operatorname{supp}\mu\) and
\(0<\rho<R<1\). If \(\rho\ge R/4\), the desired estimate follows
from the monotonicity of \(R\mapsto\mu(B(x,R))\). We may therefore
assume that \(\rho<R/4\).

For later use, we record a uniform upper estimate for descendant
measures. For \(v\in\mathcal T_k\), set
\(\mu_v:=\pi_\sharp(\tau|_{[v]})\) and
\(D_v(z):=\mathfrak M_k(z-X(v))\). The normalized measure
\(\tau([v])^{-1}(D_v)_\sharp\mu_v\)
is induced by the descendant offspring system below \(v\), with digit
parameters \((M_{k+j})_{j\ge1}\) and profile
\((T_{k+j})_{j\ge1}\). This system satisfies
\eqref{eq:branch-scale} and \eqref{eq:discrete-upper} with the same
constants. Therefore, the upper-bound argument in the proof of
Proposition~\ref{prop:near-AD-criterion} gives
\begin{equation}
\label{eq:muv-upper} 
\mu_v(B(z,t))
\lesssim
\min\{1,(t\mathfrak M_k)^\alpha\}\tau([v])
\end{equation}
uniformly in \(v\), \(z\in\mathbb R^d\), and \(t>0\).

Choose \(n\ge1\) so that
\(16r\sqrt d\,\mathfrak M_n^{-1}\le R
<16r\sqrt d\,\mathfrak M_{n-1}^{-1}\), and let
\(\mathcal V\) be the set of \(v\in\mathcal T_{n-1}\) satisfying
\(\pi([v])\cap B(x,\rho)\ne\varnothing\).
For each \(v\in\mathcal V\), choose
\(y_v\in\pi([v])\cap B(x,\rho)\), write
\(y_v=\pi(\mathbf w^{(v)})\), where
\(\mathbf w^{(v)}|_{n-1}=v\), and set
\(a_v=w_n^{(v)}\in S(v)\).

Set \(L:=\min\{M_n,(R-\rho)\mathfrak M_n/(4r\sqrt d)\}\).
The choice of \(n\) and the inequality \(\rho<R/4\) give
\(1\le L\le M_n\). If
\(a\in S(v)\cap B(a_v,L)\) and \(z\in\pi([v,a])\), then
\eqref{eq:coding-basic-bounds} gives
\(\lvert z-y_v\rvert
\le L/\mathfrak M_n+2r\sqrt d/\mathfrak M_n<R-\rho\).
Consequently,
\[
\pi([v,a])
\subset B(y_v,R-\rho)
\subset B(x,R).
\]

Using this inclusion, \eqref{eq:uniform-product-measure},
\eqref{eq:discrete-lower}, and the upper bound in
\eqref{eq:branch-scale}, successively, we obtain
\begin{align*}
\mu_v(B(x,R))
&\ge
\sum_{a\in S(v)\cap B(a_v,L)}
\tau([v,a])\\
&=
\#\bigl(S(v)\cap B(a_v,L)\bigr)
\frac{\tau([v])}{T_n}\\
&\gtrsim
\left(\frac{L}{M_n}\right)^\alpha
\tau([v])
\gtrsim
(R\mathfrak M_{n-1})^\alpha\tau([v]).
\end{align*}
For the last inequality, we used the estimate
\(L/M_n\gtrsim R\mathfrak M_{n-1}\), which follows from
\(\rho<R/4\) and
\(R\mathfrak M_{n-1}<16r\sqrt d\).

Applying \eqref{eq:muv-upper} with \(k=n-1\) and combining it with
the preceding lower bound, we obtain
\[
\mu_v(B(x,\rho))
\lesssim
\left(\frac{\rho}{R}\right)^\alpha
\mu_v(B(x,R)),
\qquad v\in\mathcal V.
\]

Finally, the level-\((n-1)\) cylinders form a partition of
\(\mathcal T_\infty\). Hence
\(\mu(B(x,\rho))
=\sum_{v\in\mathcal V}\mu_v(B(x,\rho))\), while
\(\sum_{v\in\mathcal V}\mu_v(B(x,R))
\le\mu(B(x,R))\).
Summing the preceding descendant estimate over
\(v\in\mathcal V\) gives \eqref{eq:reg} with \(s=\alpha\).
\end{proof}

\subsection{Probabilistic Fourier decay}

The objective of this subsection is to prove Proposition~\ref{prop:fourier-criterion}, which provides a basic probabilistic criterion for Fourier decay.
This result will later be used to show that, after a suitable randomization, the measure \(\mu\) admits a realization with the desired Fourier decay.

Our argument is related in spirit to earlier probabilistic constructions of Fourier-decaying fractal measures, but it differs from the usual dyadic martingale approach in a crucial way.
In that setting, the basic refinement identity is built on the cube indicator function
\begin{equation}
\label{eq:refinement-id-cube}
\mathbbm{1}_{[0, 1)^d}(x)
=
\sum_{u \in \{0, 1, \ldots, M - 1\}^d}
\mathbbm{1}_{[0, 1)^d}(Mx - u),
\end{equation}
whose Fourier decay is only of order \(1\).
To overcome this limitation, we replace \(\mathbbm{1}_{[0,1)^d}\) by a higher-order kernel $\Phi$ given by \eqref{eq:r-fold}, which still satisfies an identity of a form analogous to \eqref{eq:refinement-id-cube} (see Remark~\ref{rem:refine}).
This will make it possible to obtain Fourier decay of order exceeding \(1\).
Indeed, since \(\widehat{\Phi}(\xi)=(\widehat{\mathbbm{1}_{[0,1)^d}}(\xi))^r\), we have
\begin{equation}
\label{eq:fdecay}
|\widehat{\Phi}(\xi)| \lesssim (1+|\xi|)^{-r}.
\end{equation}
As will be seen later, this decay allows us to obtain the desired
\(\alpha/2\)- or \(\beta/2\)-Fourier decay, provided that
\(r>\max\{\alpha/2,\beta/2\}\).

Let \(p_n\) be the probability mass function on \(\mathcal{D}_n\) characterized by
\begin{equation}
\label{eq:pmf}
\sum_{u \in \mathcal{D}_n} p_n(u)\, z^u
= \bigg( \frac{1}{M_n^d}\sum_{u \in \{0, 1, \ldots, M_n - 1\}^d} z^u \bigg)^r,
\end{equation}
where \(z = (z_1, \ldots, z_d)\) and \(z^u = z_1^{u_1} \cdots z_d^{u_d}\) for \(u = (u_1, \ldots, u_d)\).
Equivalently, if \(Y_1, \ldots, Y_r\) are i.i.d. random vectors, each uniformly distributed on \(\{0, 1, \ldots, M_n - 1\}^d\), then
\begin{equation}
\label{eq:pmf-2}
p_n(u)
=
\mathbb{P}(Y_1 + \cdots + Y_r = u),
\qquad
u \in \mathcal{D}_n.
\end{equation}
We define
\begin{equation}
\label{eq:char}
m_n(\xi)
= \sum_{u \in \mathcal{D}_n} p_n(u)\, e^{-2\pi i\, u \cdot \xi},
\qquad
\xi \in \mathbb{R}^d.
\end{equation}

\begin{lem}[Refinement identity]\label{lem:refinement-id}
Let $n \ge 1$ and $\Phi_n(x) = {\mathfrak M}_n^d \Phi({\mathfrak M}_n x)$.
Then,
\begin{equation}    \label{eq:hatphi}
\widehat{\Phi_{n-1}}(\xi) = m_n(\xi/{\mathfrak M}_n) \widehat{\Phi_n}(\xi).
\end{equation}
\end{lem}
\begin{proof}
Taking the Fourier transform on both sides of \eqref{eq:refinement-id-cube} with $M = M_n$ gives
\[
\widehat{\mathbbm{1}_{[0, 1)^d}}(\xi)
= M_n^{-d} \sum_{u \in \{0, 1, \ldots, M_n - 1\}^d} e^{-2\pi i u \cdot \xi/M_n} \widehat{\mathbbm{1}_{[0, 1)^d}}\Big( \frac{\xi}{M_n} \Big).
\]
Since $\widehat{\Phi}(\xi) = (\widehat{\mathbbm{1}_{[0, 1)^d}}(\xi))^r$, raising this identity to the \(r\)-th power, we obtain
\begin{equation}
\label{eq:refine-id-2}
\widehat{\Phi}(\xi)
= \Big( \frac{1}{M_n^d} \sum_{u \in \{0, 1, \ldots, M_n - 1\}^d} e^{-2\pi i u \cdot \xi/M_n} \Big)^r \widehat{\Phi}\Big( \frac{\xi}{M_n} \Big).
\end{equation}
Now, substituting $z_j = e^{-2\pi i \xi_j/M_n}$, $j = 1, \ldots, d$, in \eqref{eq:char}, by \eqref{eq:pmf} we obtain
\[
m_n\Big( \frac{\xi}{M_n} \Big)
= \sum_{u \in \mathcal{D}_n} p_n(u)\, e^{-2\pi i u\cdot \xi/M_n}
= \Big( \frac{1}{M_n^d} \sum_{u \in \{0, 1, \ldots, M_n - 1\}^d} e^{-2\pi i u \cdot \xi/M_n}
\Big)^r.
\]
Combining this identity with \eqref{eq:refine-id-2} gives $\widehat{\Phi}(\xi) = m_n(\xi/M_n) \widehat{\Phi}(\xi/M_n)$.
Since $\widehat{\Phi_n}(\xi) = \widehat{\Phi}(\xi/{\mathfrak M}_n)$, substituting $\xi$ with $\xi/{\mathfrak M}_{n - 1}$ yields \eqref{eq:hatphi}.
\end{proof}

\begin{rem}
\label{rem:refine}
The identity \eqref{eq:hatphi} in Lemma~\ref{lem:refinement-id} can equivalently be written as follows:
\[
\Phi_{n - 1}(x) = \sum_{u \in \mathcal{D}_n} p_n(u) \Phi_n\left(x - \frac{u}{{\mathfrak M}_n}\right).
\]
This identity expresses the coarser-scale function \(\Phi_{n-1}\) as a weighted superposition of translated copies of the finer-scale function \(\Phi_n\).
\end{rem}

We now introduce a random offspring system on the set $\mathcal{W}$ of words.

\begin{defn}
\label{defn:RS}
Let $(\Omega,\mathcal{F},\mathbb{P})$ be a probability space.
We call an indexed family $S=(S_\zeta(w))_{w \in \mathcal{W},\, \zeta \in \Omega}$ a \emph{random offspring system} if the following hold:
\begin{itemize} [leftmargin=2em]
\item for each $\zeta \in \Omega$, the map $w \mapsto S_\zeta(w)$ is an offspring system on $\mathcal{W}$;
\item for every \(w\in\mathcal W\) and every
\(u\in\mathcal D_{|w|+1}\), the map
\(\zeta\mapsto\mathbbm{1}_{S_\zeta(w)}(u)\) is
\(\mathcal F\)-measurable.
\end{itemize}
We say that $S$ is a \emph{random uniform offspring system with profile $(T_n)_{n\ge 1}$} if $S$ is a random offspring system and, for every $\zeta \in \Omega$, the offspring system $S_\zeta$ is uniform with profile $(T_n)_{n\ge 1}$.
\end{defn}

Here and throughout, for notational simplicity, we suppress the dependence on the underlying random parameter \(\zeta\) whenever no confusion is likely to arise. All objects associated with a random uniform offspring system should nevertheless be understood as depending on \(\zeta\).

Given a random uniform offspring system \(S=(S_\zeta)_{\zeta\in\Omega}\), let \(\mu_\zeta\) denote the induced measure associated with \(S_\zeta\) for each \(\zeta\in\Omega\).
We then write
\[
\mu=(\mu_\zeta)_{\zeta\in\Omega},
\]
and regard \(\mu\) as the random measure induced by a random uniform offspring system \(S\).
In particular, whenever a random object appears inside \(\mathbb E\), \(\mathbb P\), or \(\sigma(\cdot)\), it is understood that the underlying randomness is with respect to \(\zeta\), unless otherwise specified.

\begin{prop}
\label{prop:fourier-criterion}
Let \(\mu\) be the random measure induced by a random uniform offspring system \(S\) on \(\mathcal W\), with digit parameters \((M_n)_{n\ge1}\) and profile \((T_n)_{n\ge1}\).
Recall that \(\mathcal D_n=\mathcal D_n^{[r]}\) and that \(p_n\) is
defined by \eqref{eq:pmf}, where \(r>\alpha/2\).
Suppose the following hold:
\begin{align}
\label{pro:con1}
&\text{There exists $C > 1$ such that $T_n \ge C M_n^\alpha$ for every $n \ge 1$.}
\\
\label{pro:con2}
&\text{The family of $\sigma$-algebras $\{ \sigma(S(w)) : w \in \mathcal{W} \}$ is independent.}
\\
\label{eq:correct-marginal}
&\mathbb E\!\left[
    \frac{1}{T_n}\mathbbm{1}_{S(w)}(u)
\right]
=p_n(u),
\qquad
\forall n\ge1,\quad
\forall w\in\mathcal W_{n-1},\quad
\forall u\in\mathcal D_n.
\end{align}
Then the random measure $\mu$ has $\alpha/2$-Fourier decay almost surely.
More precisely,
\begin{equation}
\label{eq:fourier-positive-probability}
\mathbb{P}\Big( \sup_{\xi \in \mathbb{R}^d}|\xi|^{\alpha/2}\,|\widehat{\mu}(\xi)| < \infty \Big) = 1.
\end{equation}
\end{prop}

For the proof of Proposition \ref{prop:fourier-criterion}, we use the following version of Hoeffding's inequality.

\begin{lem}
\label{lem:cond-hoeffding}
Let $(\Omega, \mathcal{F}, \mathbb{P})$ be a probability space.
Let \(\mathcal{G}\) be a sub-\(\sigma\)-algebra of $\mathcal{F}$, and let \(Y_1, \ldots, Y_N\) be complex-valued random variables which are conditionally independent given \(\mathcal{G}\).
Suppose that
\[
\mathbb{E}[Y_j \mid \mathcal{G}] = 0,
\qquad
|Y_j| \le 2
\quad
\text{a.s. \ for } 1 \le j \le N.
\]
Then, for every choice of \(\mathcal{G}\)-measurable complex coefficients \(a_1, \ldots, a_N\) and every \(t > 0\),
\[
\mathbb{P}\bigg(
\Big| \sum_{j=1}^N a_j Y_j \Big| \ge t
\,\bigg|\,
\mathcal{G}
\bigg)
\le
4 \exp\bigg(
- \frac{t^2}{16 \sum_{j=1}^N |a_j|^2}
\bigg)
\qquad
\text{a.s.}
\]
\end{lem}

\begin{proof}
Set \(Z = \sum_{j=1}^N a_j Y_j\).
If \(|Z| \ge t\), then either \(|\Re Z| \ge t/\sqrt{2}\) or \(|\Im Z| \ge t/\sqrt{2}\).
Define $A_j = \Re(a_j Y_j)$ and $B_j = \Im(a_j Y_j)$.
Since \(a_j\) is \(\mathcal{G}\)-measurable and the \(Y_j\) are conditionally independent given \(\mathcal{G}\), both families \((A_j)_{j=1}^N\) and \((B_j)_{j=1}^N\) are conditionally independent given \(\mathcal{G}\).
Moreover,
\[
\mathbb{E}[A_j \mid \mathcal{G}] = 0,
\qquad
\mathbb{E}[B_j \mid \mathcal{G}] = 0,
\qquad
|A_j|, |B_j| \le 2 |a_j|.
\]
Applying the real-valued conditional Hoeffding inequality (see, e.g., \cite[Theorem 2.8]{BLM} or \cite[Theorem 8]{MNT}) to \(\sum_j A_j\) and \(\sum_j B_j\), we obtain
\[
\mathbb{P}\bigg(
|\Re Z| \ge \frac{t}{\sqrt{2}}
\,\bigg|\,
\mathcal{G}
\bigg)
\le
2 \exp\bigg(
- \frac{t^2}{16 \sum_{j=1}^N |a_j|^2}
\bigg) \quad \text{a.s.},
\]
and the same bound with \(\Re Z\) replaced by \(\Im Z\).
The desired estimate follows by combining these two bounds.
\end{proof}

\begin{proof}[Proof of Proposition~\ref{prop:fourier-criterion}]
Since $\mu_n$ weakly converges to $\mu$, it follows that \(\widehat{\mu_n}(\xi) \to \widehat{\mu}(\xi)\).
Also, note that \(\widehat{\Phi_n}(\xi) = \widehat{\Phi}({\mathfrak M}_n^{-1} \xi) \to 1\).
Thus, we may write
\[
\widehat{\mu}(\xi) = \widehat{\Phi}(\xi) + \sum_{n=1}^\infty ( \widehat{\Phi_n}(\xi) \widehat{\mu_n}(\xi) - \widehat{\Phi_{n - 1}}(\xi) \widehat{\mu_{n - 1}}(\xi)),
\]
with $\mu_{0}=\delta_0$.
Using Lemma~\ref{lem:refinement-id}, we have
\begin{align*}
        \widehat{\mu}(\xi)
                = \widehat{\Phi}(\xi) + \sum_{n=1}^\infty \widehat{\Phi}({\mathfrak M}_n^{-1} \xi) {\mathfrak D_n}(\xi),
\end{align*}
where
\begin{equation}
\label{eq:Dn}
{\mathfrak D_n}(\xi) = \widehat{\mu_n}(\xi) - m_n ({\mathfrak M}_n^{-1} \xi)\,\widehat{\mu_{n-1}}(\xi).
\end{equation}

Suppose that there is a constant $C>0$ such that
\[
{\mathfrak M}_n^{\alpha/2}\sup_{\xi \in \mathbb{R}^d} |{\mathfrak D_n}(\xi)| \le C
\]
for all \(n\).
Then, using \eqref{eq:fdecay}, we have
\[
|\widehat{\mu}(\xi)| \lesssim |\xi|^{-r} + \sum_{n=1}^{\infty} (1+ {\mathfrak M}_n^{-1} |\xi|)^{-r} {\mathfrak M}_n^{-\alpha/2} \lesssim |\xi|^{-\alpha/2}.
\]
The second inequality is easy to show since ${\mathfrak M}_n$ grows super-exponentially and $r > \alpha/2$.
In particular, one may use $\mathfrak M_{n+1}\ge 2 \mathfrak M_n$.
Thus, it suffices to prove that
\begin{equation}
\label{eq:fourier-reduction}
\lim_{A \to \infty} \mathbb{P}\Big(  {\mathfrak M}_n^{\alpha/2}\!\sup_{\xi \in \mathbb{R}^d} |{\mathfrak D_n}(\xi)| \le A, \quad \forall n \ge 1 \Big) = 1.
\end{equation}

Let $\mathcal{F}_0$ be the trivial $\sigma$-algebra.
For \(k \ge 1\), let
\begin{equation}
\label{eq:filtration}
\mathcal{F}_{k} = \sigma(S(w) : w \in \mathcal{W},\, |w| < k).
\end{equation}
Recalling \eqref{eq:Dn} and \eqref{eq:truncated-induced-measure}, we have
\begin{align}
    \label{Dn}
        {\mathfrak D_n}(\xi)
         &= \frac{1}{\prod_{k=1}^{n - 1} T_k} \sum_{v \in \mathcal{T}_{n - 1}} e^{-2\pi i \xi \cdot X(v)} Y_v(\xi),
\end{align}
where
\[
Y_v(\xi) = \frac{1}{T_n} \sum_{u \in S(v)} e^{-2\pi i \xi \cdot u / {\mathfrak M}_n} - m_n(\xi/{\mathfrak M}_n).
\]

For each \(v\in\mathcal T_{n-1}\), the coefficient
\(e^{-2\pi i\xi\cdot X(v)}\) is
\(\mathcal F_{n-1}\)-measurable.
By the independence assumption \eqref{pro:con2}, the family
\(\{Y_v(\xi):v\in\mathcal T_{n-1}\}\) is conditionally independent
given \(\mathcal F_{n-1}\).
For every deterministic \(v\in\mathcal W_{n-1}\), the same assumption
also shows that \(Y_v(\xi)\) is independent of
\(\mathcal F_{n-1}\). Hence, by \eqref{eq:correct-marginal} and
\eqref{eq:char},
\begin{align*}
\mathbb E\bigl[Y_v(\xi)\mid\mathcal F_{n-1}\bigr]
&=\mathbb E\bigl[Y_v(\xi)\bigr]\\
&=
\sum_{u\in\mathcal D_n}p_n(u)
 e^{-2\pi i\xi\cdot u/\mathfrak M_n}
-m_n(\xi/\mathfrak M_n)
=0.
\end{align*}
It is also clear that \(|Y_v(\xi)|\le2\).
Since $S$ is a random uniform offspring system with profile $(T_n)_{n \ge 1}$, $\# \mathcal{T}_{n - 1}= \prod_{k=1}^{n - 1} T_k$.
Therefore, applying Lemma~\ref{lem:cond-hoeffding}, we obtain
\[
\mathbb{P}\bigl(|{\mathfrak D_n}(\xi)| \ge t \,\big|\, \mathcal{F}_{n-1}\bigr)
\le 4 \exp\Big(- \frac{1}{16} \Big(\prod_{k=1}^{n - 1} T_k\Big) t^2\Big).
\]

By taking expectation and using the tower property, this gives the inequality $ \mathbb{P}\bigl( |{\mathfrak D_n}(\xi)| \ge t \bigr) \le 4 \exp(- \frac{1}{16} (\prod_{k=1}^{n - 1} T_k) t^2).$
Now, using the assumption \eqref{pro:con1}, i.e., $T_k \ge C M_k^\alpha$ and taking \(t = B {\mathfrak M}_n^{-\alpha/2}\), we obtain
\begin{equation}
\label{eq:fixed-freq}
\mathbb{P}\bigl(
{\mathfrak M}_n^{\alpha/2} |{\mathfrak D_n}(\xi)| \ge B
\bigr)
\le
4 \exp\left(- c B^2 C^n M_n^{-\alpha}
\right)
\end{equation}
for some constant \(c > 0\).

Let us set
\[
\eta_n = \sum_{u \in \mathcal{D}_n} p_n(u)\,\delta_{u/{\mathfrak M}_n}.
\]
Recalling \eqref{eq:char}, we note that $\widehat{\eta_n}(\xi) = m_n({\mathfrak M}_n^{-1} \xi)$.
Thus, by \eqref{eq:Dn} we have
\[
{\mathfrak D_n} = (\mu_n - \mu_{n-1} * \eta_n)^{\wedge}.
\]
Since \(\mu_n - \mu_{n-1} * \eta_n\) is supported on \({\mathfrak M}_n^{-1}\mathbb{Z}^d \cap [0,r]^d\), the function \({\mathfrak D_n}\) is \({\mathfrak M}_n\)-periodic in each coordinate.
Also, the total variation of $(\mu_n - \mu_{n-1} * \eta_n)$ is bounded by $2$; hence its Fourier transform \({\mathfrak D_n}\) is \(O_{r,d}(1)\)-Lipschitz.
Thus, we have
\[
\sup_{\xi \in \mathbb{R}^d} |{\mathfrak D_n}(\xi)|
= \sup_{\xi \in [0, {\mathfrak M}_n]^d} |{\mathfrak D_n}(\xi)|
\le \sup_{\xi \in \Lambda_n} |{\mathfrak D_n}(\xi)| + {\mathfrak M}_n^{-\alpha/2},
\]
where \(\Lambda_n\) is a \(c_{r,d} {\mathfrak M}_n^{-\alpha/2}\)-net of \([0,{\mathfrak M}_n]^d\) with a sufficiently small constant $c_{r,d}$.
Thus, it follows that
\[
\mathbb{P} \Big(
{\mathfrak M}_n^{\alpha/2} \sup_{\xi \in \mathbb{R}^d} |{\mathfrak D_n}(\xi)|
>
B+1
\Big)\le \sum_{\xi \in \Lambda_n}  \mathbb{P} \Big(
{\mathfrak M}_n^{\alpha/2}   |{\mathfrak D_n}(\xi)| \ge B \Big).
\]
As a result, using the trivial bound \(\#\Lambda_n \lesssim {\mathfrak M}_n^{d(1+\alpha/2)}\) and \eqref{eq:fixed-freq}, we obtain
\[
\mathbb{P} \Big( {\mathfrak M}_n^{\alpha/2}
\sup_{\xi \in \mathbb{R}^d} |{\mathfrak D_n}(\xi)|
>
B+1
\Big)
\lesssim
{\mathfrak M}_n^{d(1+\alpha/2)}
\exp\left(
- c B^2 C^n M_n^{-\alpha}
\right).
\]
By the growth condition \eqref{eq:growth-condition} and Lemma \ref{lem:absorption}, the right-hand side is summable in \(n\) for every fixed \(B\), and its sum tends to \(0\) as \(B \to \infty\).
Therefore,
\[
\lim_{B \to \infty} \sum_{n=1}^{\infty} \mathbb{P}\Big({\mathfrak M}_n^{\alpha/2}\sup_{\xi \in \mathbb{R}^d} |{\mathfrak D_n}(\xi)| > B+1 \Big) = 0.
\]
This shows \eqref{eq:fourier-reduction} and completes the proof.
\end{proof}

\subsection{Proof of Theorem~\ref{thm:near-AD-Salem}}

To prove Theorem~\ref{thm:near-AD-Salem}, it remains to ensure that the hypotheses of Propositions~\ref{prop:near-AD-criterion} and \ref{prop:fourier-criterion} are satisfied simultaneously.
To this end, we utilize the following lemma.

Let \(\mathcal D_{r,M}=\{0,1,\dots,r(M-1)\}^d\), and let \(p_{r,M}\) denote the probability mass function of the sum \(Y_1+\cdots+Y_r\), where \(Y_1,\dots,Y_r\) are i.i.d. random vectors uniformly distributed on \(\{0,1,\dots,M-1\}^d\) (cf.~\eqref{eq:pmf-2}).

\begin{lem}[AD-regular sampling]\label{lem:AD-regular-sampling}
Let \(0 < \alpha < d\).
Then, for some constant \(C_0 = C_0(d,r,\alpha) \ge 1\) and some integer \(n_0 = n_0(d,r,\alpha) \ge 1\), the following holds whenever \(M=2^n\) with \(n\ge n_0\): there exists an integer \(T = T(d,r,\alpha,M)\) satisfying
\begin{equation}
\label{eq:T-scale}
r^d M^\alpha \le T \le C_0 r^d M^\alpha,
\end{equation}
and a random subset \(S \subset \mathcal D_{r,M}\) such that every realization \(E\) of \(S\) satisfies
\begin{gather}
        \#E = T; \label{eq:cloud-size}
        \\
        \#(E \cap B(x,R)) \ge C_0^{-1} R^\alpha,
        \qquad
        \forall x \in E,\ \forall R \in [1,M];  \label{eq:cloud-lower}
        \\
         \#(E \cap B(x,R)) \le C_0 R^\alpha,
        \qquad
        \forall x \in \mathbb R^d,\ \forall R \in [1,M]. \label{eq:cloud-upper}
\end{gather}
In addition,
\begin{equation}
\label{eq:cloud-inclusion-prob}
\mathbb P(u\in S)=T\,p_{r,M}(u),
\qquad
\forall u\in\mathcal D_{r,M}.
\end{equation}
\end{lem}

Postponing the proof of this lemma to the end of this section, we prove Theorem~\ref{thm:near-AD-Salem}.

\begin{proof}[Proof of Theorem~\ref{thm:near-AD-Salem}]
The endpoint cases $\alpha=0$ and $\alpha=d$ are immediate.
If \(\alpha = 0\), we take \(\mu = \delta_0\).
If \(\alpha=d\), take normalized Lebesgue measure on the unit ball.
This measure is \(d\)-AD regular, and the standard Bessel-function
formula for the Fourier transform of a ball gives
\[
|\widehat\mu(\xi)|
\lesssim(1+|\xi|)^{-(d+1)/2}.
\]
Consequently, we may assume \(0 < \alpha < d\).

Let \(C_0\) and \(n_0\) be given as in Lemma \ref{lem:AD-regular-sampling}.
For example, set
\[
M_n:=2^{n_0+\lceil\sqrt n\rceil},
\qquad n\ge1.
\]
Then each \(M_n\) is an admissible dyadic scale,
\(M_n\to\infty\), and
\(\log M_n=O(\sqrt n)=o(n)\), so
\eqref{eq:growth-condition} holds.
Applying Lemma~\ref{lem:AD-regular-sampling}, we can construct a random offspring system $S$.
Indeed, for each \(n \ge 1\), let \(T_n = T(d,r,\alpha,M_n)\) be an integer given in Lemma~\ref{lem:AD-regular-sampling}.
For $n \ge 1$ and $w \in \mathcal{W}_{n - 1}$, let $S(w) \subset \mathcal{D}_n$ be an independent copy of the random subset given by Lemma~\ref{lem:AD-regular-sampling} with $M=M_n$.
Consequently, we have a random uniform offspring system $S$ with profile $(T_n)_{n \ge 1}$, thanks to \eqref{eq:cloud-size}.

Let \(\mu\) be the random measure induced by \(S\).
We now apply Proposition~\ref{prop:fourier-criterion} to show that $\mu$ has the desired Fourier decay.
Note that the family of $\sigma$-algebras \(\{\sigma(S(w)) : w \in \mathcal{W}\}\) is independent by the construction.
We also note from \eqref{eq:T-scale} that $T_n \ge r^d M_n^\alpha$ with $r^d > 1$.
Thus, \eqref{pro:con1} and \eqref{pro:con2} in Proposition~\ref{prop:fourier-criterion} are verified.
Also, since \(p_n=p_{r,M_n}\), \eqref{eq:cloud-inclusion-prob} with $T=T_n$ and $M=M_n$ gives
\[
\mathbb{E}\bigg[ \frac{1}{T_n} \mathbbm{1}_{S(w)}(u) \bigg] = \frac{1}{T_n} \mathbb{P}(u \in S(w)) = p_n(u),
\]
for \(w\in\mathcal W_{n-1}\) and \(u\in\mathcal D_n\).
Hence, we have \eqref{eq:correct-marginal}.
Therefore, by Proposition~\ref{prop:fourier-criterion}, the random measure $\mu$ has $\alpha/2$-Fourier decay almost surely.

On the other hand, the conditions \eqref{eq:T-scale}, \eqref{eq:cloud-lower}, and \eqref{eq:cloud-upper} verify the conditions \eqref{eq:branch-scale}, \eqref{eq:discrete-lower}, and \eqref{eq:discrete-upper} in Proposition~\ref{prop:near-AD-criterion}, respectively.
Therefore, the random measure \(\mu\) is near \(\alpha\)-AD regular not merely almost surely, but in fact for every realization.
Hence, near $\alpha$-AD regularity and $\alpha/2$-Fourier decay hold simultaneously almost surely.
Choosing such a realization completes the proof.
\end{proof}

\subsection{Proof of the AD-regular sampling}

In order to prove Lemma~\ref{lem:AD-regular-sampling}, we first construct a random ``seed'' set in
\[
\mathcal D_{1,M}=\{0,1,\ldots,M-1\}^d
\]
with uniform inclusion probability, and then recover the probability mass function \(p_{r,M}\) by adding \(r-1\) independent uniform random vectors (see \eqref{eq:pmf} and \eqref{eq:pmf-2}).

\begin{lem}
\label{lem:ad-regular-seed}
Let \(0<\alpha<d\).
Then, for some constant $C_0=C_0(d,r,\alpha)\ge 1$ and some integer $n_0=n_0(d,r,\alpha)\ge 1$, the following holds whenever \(M=2^n\) with \(n\ge n_0\): there exist an integer $T=T(d,r,\alpha,M)$ satisfying \eqref{eq:T-scale}, and a random subset \(A\subset \mathcal D_{1,M}\) such that every realization \(E\) of \(A\) satisfies
\begin{gather}
\#E = T, 
\label{eq:cloud-size-s}
\\
\#(E\cap B(x,R)) \ge C_0^{-1} R^\alpha,
\qquad
\forall x\in E,\ \forall R\in[1,M]; 
\label{eq:cloud-lower-s}
\\
\#(E\cap B(x,R)) \le C_0 R^\alpha,
\qquad
\forall x\in\mathbb R^d,\ \forall R\in[1,M].
\label{eq:cloud-upper-s}
\end{gather}
In addition, we have
\begin{equation}
\label{eq:cloud-inclusion-prob-s}
\mathbb P(a\in A)=TM^{-d},
\qquad
\forall a\in\mathcal D_{1,M}.
\end{equation}
\end{lem}

\begin{proof}
Choose \(n_0 \ge 1\) such that $r^d \le 2^{(d-\alpha)n_0} \le 2^{d-\alpha} r^d$.
Fix \(M = 2^n\) with \(n \ge n_0\).
Every \(a \in \mathcal{D}_{1,M}\) has a unique binary expansion $a = \sum_{j=0}^{n-1} 2^j b_j$, where $b_j \in \{0,1\}^d$, so we identify
\[
\mathcal{D}_{1,M} \simeq\prod_{j=0}^{n-1} \{0,1\}^d
\]
via the map $a\mapsto b=(b_0, \dots, b_{n-1})$.

Define
\[
s_j :=
\begin{cases}
    \lceil \alpha(j+1) \rceil - \lceil \alpha j \rceil,
        & 0 \le j < n-n_0,
        \\
    \qquad\quad    d,
        & n-n_0 \le j < n.
\end{cases}
\]
Let \(N(m) := \sum_{j=0}^{m-1} s_j\) for \(0 \le m \le n\), with \(N(0)=0\).
Then, we have
\[
N(m) =
\begin{cases}
    \qquad\qquad\quad \lceil \alpha m \rceil,
        & 0 \le m \le n-n_0, \\
    \lceil \alpha(n-n_0) \rceil + d(m-n+n_0),
        & n-n_0 < m \le n.
\end{cases}
\]

For each \(j\), let \(\phi_j : \{0,1\}^d \to \{0,1\}^{d-s_j}\) be the projection onto the last \(d-s_j\) coordinates, and define a map
\[
\phi :
\mathcal{D}_{1,M}
\simeq
\prod_{j=0}^{n-1} \{0,1\}^d
\to
\Omega := \prod_{j=0}^{n-1} \{0,1\}^{d-s_j}
\]
by
\[
\phi((b_0,\dots,b_{n-1})) = (\phi_0(b_0),\dots,\phi_{n-1}(b_{n-1})).
\]

For $\zeta\in \Omega$, set
\[
A_\zeta = \phi^{-1}(\zeta),
\]
which is a subset of $\mathcal D_{1,M}$.
Each fiber has cardinality \(\#A_\zeta=2^{N(n)}\).
Let
\[
T := 2^{N(n)} .
\]
Since \(N(n)=\lceil \alpha(n-n_0)\rceil+dn_0\), the choice of \(n_0\) gives $r^d M^\alpha \le T \le 2^{d-\alpha+1} r^d M^\alpha$, which verifies \eqref{eq:T-scale}.
Now, we equip \(\Omega\) with the discrete \(\sigma\)-algebra and the uniform probability measure, and define a random subset \(A\) by \(A(\zeta):=A_\zeta\).
Then every realization $E$ of \(A\) is one of the fibers \(A_\zeta\), and hence \eqref{eq:cloud-size-s} holds.

We proceed to show \eqref{eq:cloud-upper-s} and \eqref{eq:cloud-lower-s}.
Let \(Q \subset [0,M)^d\) be a half-open dyadic cube of sidelength \(2^m\), \(0 \le m \le n\).
Under the binary identification, belonging to \(Q\) fixes the digits \(b_j\) for \(j \ge m\) and leaves \(b_0,\dots,b_{m-1}\) free.
Hence, $\#(A_\zeta \cap Q) \in \{0,\, 2^{N(m)}\}$.
Since $ \alpha m \le N(m) \le \alpha m + 1 + (d-\alpha)n_0$ for $m \in \{0,\dots,n\}$, we obtain
\[
\#(A_\zeta \cap Q) \lesssim 2^{\alpha m} = \ell(Q)^\alpha,
\]
and, whenever \(A_\zeta\cap Q\neq\emptyset\),
\[
\#(A_\zeta \cap Q) \gtrsim 2^{\alpha m} = \ell(Q)^\alpha.
\]

Using these upper and lower bounds, we now pass from dyadic cubes to balls.
Denote \(B_\infty(x,R) := x + [-R,R]^d\).
If \(2^{m-1} < R \le 2^m\), then \(B_\infty(x,R) \cap [0,M)^d\) can be covered by \(O_d(1)\) dyadic cubes of sidelength \(2^m\).
Hence
\[
\#(A_\zeta \cap B_\infty(x,R)) \lesssim R^\alpha.
\]
Conversely, if \(x \in A_\zeta\) and \(2^m \le R/\sqrt d < 2^{m+1}\), then the dyadic cube of sidelength \(2^m\) containing \(x\) is contained in \(B(x,R)\), so
\[
\#(A_\zeta \cap B(x,R)) \gtrsim R^\alpha.
\]
For \(1 \le R < \sqrt d\), the same lower bound is trivial since \(x \in A_\zeta\).
Since \(B(x,R) \subset B_\infty(x,R)\), the Euclidean upper bound also follows.
After enlarging \(C_0\) if necessary, every fiber \(A_\zeta=E\) satisfies \eqref{eq:cloud-upper-s} and \eqref{eq:cloud-lower-s}.
Therefore, every realization of \(A\) satisfies these estimates.

Finally, the fibers \(\{A_\zeta\}_{\zeta\in\Omega}\) form a partition of \(\mathcal D_{1,M}\), and each has size \(T\).
Hence \(\#\Omega=M^d/T\).
Since every \(a\in\mathcal D_{1,M}\) belongs to exactly one fiber,
\[
\mathbb P(a\in A)={1}/{\#\Omega}=TM^{-d}.
\]
This proves \eqref{eq:cloud-inclusion-prob-s}, and completes the proof.
\end{proof}

We now prove Lemma \ref{lem:AD-regular-sampling} by using Lemma \ref{lem:ad-regular-seed}.

\begin{proof}[Proof of Lemma~\ref{lem:AD-regular-sampling}]
Let \(n_0\) be as in Lemma~\ref{lem:ad-regular-seed}.
Fix \(M=2^n\) with \(n\ge n_0\), and let \(A\) be a random subset of \(\mathcal D_{1,M}\) and \(T\) be given by Lemma~\ref{lem:ad-regular-seed}.
Then \(T\) satisfies \eqref{eq:T-scale}, and every realization \(E\) of \(A\) satisfies \eqref{eq:cloud-size-s}, \eqref{eq:cloud-upper-s}, and \eqref{eq:cloud-lower-s}.

Let \(Y_1,\dots,Y_r\) be i.i.d. random vectors, each uniformly distributed on \(\mathcal D_{1,M}\), and independent of \(A\).
Let
\begin{equation}
\label{eq:vs}
V:=Y_2+\cdots+Y_r,
\qquad
S:=A+V.
\end{equation}
Then \(S\subset \mathcal D_{r,M}\).
Moreover, if \(E\) is a realization of \(A\) and \(v\) is a realization of \(V\), then the corresponding realization of \(S\) is \(E+v\).
Hence \(\#(E+v)=\#E\), and
\[
\#\bigl((E+v)\cap B(x,R)\bigr)=\#\bigl(E\cap B(x-v,R)\bigr).
\]
Therefore \eqref{eq:cloud-size} follows from \eqref{eq:cloud-size-s}, while \eqref{eq:cloud-lower} and \eqref{eq:cloud-upper} follow from \eqref{eq:cloud-lower-s} and \eqref{eq:cloud-upper-s}, respectively.

It remains to verify \eqref{eq:cloud-inclusion-prob}.
By \eqref{eq:vs}, we have
\begin{align*}
\mathbb P(u\in S)
= \sum_{a\in\mathbb Z^d}
    \mathbb P\bigl(a\in A,\ V=u-a\bigr)
= \sum_{a\in\mathbb Z^d}
    \mathbb P(a\in A)\,\mathbb P(V=u-a)
\end{align*}
for \(u\in \mathcal D_{r,M}\).
The last equality follows since \(A\) and \(V\) are independent.
Now, \(A\subset \mathcal D_{1,M}\) almost surely, so \(\mathbb P(a\in A)=0\) for \(a\notin \mathcal D_{1,M}\).
Therefore, by \eqref{eq:cloud-inclusion-prob-s}, we have $\mathbb P(a\in A)=TM^{-d}=T\,\mathbb P(Y_1=a)$ for $a\in\mathbb Z^d$, where \(\mathbb P(Y_1=a)=0\) for \(a\notin \mathcal D_{1,M}\).
Thus,
\begin{align*}
\mathbb P(u\in S)  
&= T \sum_{a\in\mathbb Z^d}
\mathbb P(Y_1=a)\,\mathbb P(Y_2+\cdots+Y_r=u-a) \\
&= T\,\mathbb P(Y_1+\cdots+Y_r=u)
= T\,p_{r,M}(u),
\end{align*}
where the last equality follows from the definition of \(p_{r,M}\) (cf. \eqref{eq:pmf-2}).
This proves \eqref{eq:cloud-inclusion-prob}, and hence the lemma.
\end{proof}
 \section{\texorpdfstring{Convolution: nongeometric case $\alpha<\beta$}{Convolution: nongeometric case alpha < beta}}
\label{sec4}

In this section we prove Proposition \ref{prop:heavy-core-i}.
Recall the associated parameters
\[
\alpha, \  \beta, \ s\in [0, 2\alpha - \beta]
\]
appearing in Proposition \ref{prop:heavy-core-i}.
To this end, we first construct two independent random uniform offspring systems, which are associated with the parameters $\beta$ and $s$, respectively, and then combine them into a single random offspring system, by which the desired measure $\mu$ is defined. The system associated with the parameter $\beta$ is responsible for the overall $\beta$-dimensional size and for the $\beta/2$-Fourier decay, while the one associated with $s$ produces a heavy core $F \subset \operatorname{supp}\mu$ on which $\mu$ has nearly $\alpha$-dimensional mass at every small scale.

\newcommand{\sT}{{}^sT}
\newcommand{\bT}{{}^\beta T}
\newcommand{\sS}{{}^s\!S}
\newcommand{\bS}{{}^\beta\! S}
\newcommand{\smu}{{}^s\!\mu}

\subsection{Construction of a random offspring system}

We keep using the same notation as in Subsection~\ref{subsec:basic-setup}, except that in this subsection we choose a sufficiently large positive integer $r$ such that
\begin{equation}
\label{eq:r-con}
r > \max\{1,  \alpha/2,  \beta/2\}.
\end{equation}
We assume $s > 0$; the case $s = 0$ can be handled by a simple modification of the argument below (see Remark \ref{rem:s=0}).

Choose a dyadic scale sequence $(M_n)_{n \ge 1}$ satisfying the growth condition \eqref{eq:growth-condition}.
Since this condition is unaffected by modifying finitely many initial terms, we may enlarge those terms if necessary and thereby assume that Lemma~\ref{lem:AD-regular-sampling} applies at every scale $M_n$.
Applying this lemma with the exponent $\alpha$ replaced successively by $\beta$ and $s$, we have the corresponding integers
\[
\bT_n = T(d,r,\beta,M_n), \quad \sT_n= T(d,r,s,M_n)
\]
given by Lemma~\ref{lem:AD-regular-sampling}.
Thus, $\bT_{n} \sim M_n^\beta$ and $\sT_{n} \sim M_n^s$.

Similarly, for $n \ge 1$ and $w \in \mathcal{W}_{n-1}$, let
\[
\bS(w), \ \sS(w) \subset \mathcal{D}_n
\]
denote the random subsets given by the AD-regular sampling at scale $M_n$ with $\alpha=\beta$ and $\alpha=s$, respectively, and assume that all these random sets are mutually independent.
As before, we are suppressing the random parameter $\zeta\in \Omega$ in the notation (see Definition \ref{defn:RS}).
Thus,
\[
\bS = (\bS(w))_{w \in \mathcal{W}}, \quad
\sS = (\sS(w))_{w \in \mathcal{W}}
\]
denote independent random uniform offspring systems on $\mathcal{W}$, with profiles $(\bT_{n} )_{n \ge 1}$ and $(\sT_{n} )_{n \ge 1}$, respectively.

Now, let $S$ denote the random offspring system
\begin{equation}
\label{swbs}
S(w) = \bS(w) \cup \sS(w), \quad  w\in \mathcal{W}.
\end{equation}
Note that $S$ is not necessarily uniform.
Let $\mathcal{T}$ be the tree associated with $S$, as defined in \eqref{eq:associated-tree}.
Choose a sufficiently small constant $c > 0$, to be determined later, and set
\begin{equation}
\label{lambdan}
\lambda_n = c(\sT_{n}  M_n^{-\alpha}).
\end{equation}
Since $\sT_{n} \sim M_n^s$ and $s \le 2\alpha - \beta < \alpha$, we have $0 < \lambda_n < 1/2$ for all $n$ provided that $c$ is sufficiently small.

There exists a unique Borel probability measure $\tau$ on $\mathcal{T}_\infty$ such that
\begin{equation}
\label{eq:recursion}
\tau([w, u]) = \left( \frac{1-\lambda_n}{\bT_{n} } \mathbbm{1}_{\bS(w)}(u) + \frac{\lambda_n}{\sT_{n} } \mathbbm{1}_{\sS(w)}(u) \right) \tau([w])
\end{equation}
for $w \in \mathcal{T}_{n - 1}$ and $u \in S(w)$.
Indeed, \eqref{eq:recursion} defines a consistent probability premeasure on the cylinder semiring $\{[w] : w \in \mathcal{T}\} \cup \{\emptyset\}$ of $\mathcal{T}_\infty$, and hence, by the Carath\'eodory extension theorem, determines a unique Borel probability measure on $\mathcal{T}_\infty$.
We extend $\tau$ trivially to $\mathcal{W}_\infty$ by setting $\tau(E) = \tau(E \cap \mathcal{T}_\infty)$.

For the measure $\tau$ defined above, we define the pushforward measures
\begin{equation}
\label{eq:mu}
\mu = \pi_\sharp \tau,   \quad   \mu_n = (\pi_n)_\sharp \tau.
\end{equation}
On the other hand, let ${}^s\mathcal{T}$ denote the tree associated with the random uniform offspring system $\sS$.
Note that ${}^s\mathcal{T} \subset \mathcal{T}$.
As before, we write
\[
{}^s\mathcal{T}_\infty = \{ \mathbf{w} \in \mathcal{W}_\infty : \mathbf{w}|_n \in {}^s\mathcal{T}, \; \forall n \ge 1 \}.
\]
Since $\sS$ is a random uniform offspring system with profile $(\sT_{n} )_{n \ge 1}$, we have the induced (random) measure ${}^s\!\mu$ of $\sS$ as in Subsection~\ref{subsec:basic-setup}.
We set
\begin{equation}
\label{eq:F}
F = \operatorname{supp} \smu.
\end{equation}
\begin{rem}
\label{rem:s=0}
When \(s=0\), instead of taking $\sS(w)$ from the AD-regular sampling construction, we define \({}^s\!S(w)=\{U(w)\}\), where \(U(w)\) is a \(\mathcal D_n\)-valued random variable with distribution \(p_n\), chosen independently for each \(w\in\mathcal W_{n-1}\).
In this case, \({}^s\!T_n=1\), and the induced measure \({}^s\!\mu\) is supported on a single branch; in particular, \(F=\operatorname{supp}({}^s\!\mu)\) is a singleton, and hence is trivially near \(0\)-AD regular.
Moreover, the argument below remains valid with this choice: the lower mass estimate follows exactly as before from \(\lambda_n=cM_n^{-\alpha}\), while in the Frostman bound one simply uses \(\#({}^s\!S(w)\cap B(y,R))\le 1=(R)^0\) for \(R\ge1\).
Thus the case \(s=0\) is handled in the same manner.
\end{rem}

In what follows, we present the proof of Proposition \ref{prop:heavy-core-i} in several steps.

\subsection{Heavy core}

We first show that the set $F$ defined above satisfies the desired properties.

\begin{prop}
\label{prop:conv-ngeo-heavy}
Let $\mu$ and $F$ be given by \eqref{eq:mu} and \eqref{eq:F}, respectively.
Then, $F$ is near $s$-AD regular and $F \subset \operatorname{supp} \mu$.
Moreover, \eqref{eq:nearF} holds.
More specifically, for any $\varepsilon>0$ we have
\begin{equation}    \label{Bxr}  \mu(B(x, r\sqrt{d} {\mathfrak M}_n^{-1})) \gtrsim_\varepsilon {\mathfrak M}_n^{-\alpha - \varepsilon}
\end{equation}
for $x\in F$ whenever $n\ge 1$.
\end{prop}

Before beginning the proof, we observe that it is sufficient for \eqref{eq:nearF} to show \eqref{Bxr}.

\begin{lem}
\label{lower-ad}
Suppose that \eqref{Bxr} holds for $n\ge 1$.
Then, for any $\varepsilon>0$, we have $ \mu(B(x, \delta)) \gtrsim_\varepsilon \delta^{\alpha + \varepsilon}$ whenever $\delta\in (0,1)$.
\end{lem}

\begin{proof}
Set $C=r\sqrt d$, and fix a target exponent $\varepsilon_0>0$.
Choose $\eta,\eta'>0$ so small that $(\alpha+\eta)(1+\eta')\le \alpha+\varepsilon_0.$ 
From \eqref{eq:growth-condition}, we have
\begin{equation}
\label{mn+1}
\frac{\log\mathfrak M_{n+1}}{\log\mathfrak M_n}
\longrightarrow1.
\end{equation}
For $\delta>0$ sufficiently small, choose $m\ge2$ such that
$
C\mathfrak M_m^{-1}\le\delta<C\mathfrak M_{m-1}^{-1}.
$
By \eqref{mn+1}, $\mathfrak M_m^{-1}\gtrsim_{\eta'}\delta^{1+\eta'}.$ 
Since $\mu(B(x,\delta))
\ge \mu(B(x,C\mathfrak M_m^{-1}))$, using \eqref{Bxr} with exponent $\eta$, we obtain
\begin{equation*}
\mu(B(x,\delta))
 \gtrsim_\eta \mathfrak M_m^{-(\alpha+\eta)}
 \gtrsim_{\eta,\eta'}
 \delta^{(\alpha+\eta)(1+\eta')}
 \ge \delta^{\alpha+\varepsilon_0}.
\end{equation*}
The remaining bounded range of $\delta$ follows from \eqref{Bxr} at a
fixed initial scale, after adjusting the implicit constant.
\end{proof}

\begin{proof}[Proof of Proposition~\ref{prop:conv-ngeo-heavy}]
When $s>0$, Lemma~\ref{lem:AD-regular-sampling} gives all the hypotheses
of Proposition~\ref{prop:near-AD-criterion}, so every realization of
$\smu$ is near $s$-AD regular. When $s=0$, Remark~\ref{rem:s=0} shows
directly that $F=\operatorname{supp}\smu$ is a singleton and hence is
near $0$-AD regular.

We now show $F \subset \operatorname{supp} \mu$.
Let $x \in F$.
Then there exists $\mathbf{a} \in {}^s\mathcal{T}_\infty $ such that $x = \pi(\mathbf{a})$.
By \eqref{eq:coding-basic-bounds}, the cylinder images $\pi([\mathbf{a}|_n])$ form a nested family of sets shrinking to $x$, and by \eqref{eq:recursion} and \eqref{lambdan} we have
\[
\mu(\pi([\mathbf{a}|_n])) \ge \tau([\mathbf{a}|_n]) \ge \prod_{k=1}^n \frac{\lambda_k}{\sT_{k}} = \prod_{k=1}^n (c M_k^{-\alpha}) = c^n {\mathfrak M}_n^{-\alpha}
\]
for every $n$.
Thus, it follows that every neighborhood of $x$ has positive $\mu$-measure, and hence $x \in \operatorname{supp}\mu$.

Furthermore, by \eqref{eq:coding-basic-bounds} we have $\pi([\mathbf{a}|_n]) \subset x + {\mathfrak M}_n^{-1}[-r,r]^d$.
Therefore, from the above inequality, it follows that
\[
\mu(B(x, r\sqrt{d} {\mathfrak M}_n^{-1})) \ge  c^n {\mathfrak M}_n^{-\alpha}
\]
for $x \in F$.
Note from Lemma~\ref{lem:absorption} that $c^n {\mathfrak M}_n^{-\alpha} \gtrsim_\varepsilon {\mathfrak M}_n^{-\alpha - \varepsilon}$.
Thus, we have \eqref{Bxr}.
Now, by Lemma~\ref{lower-ad}, \eqref{eq:nearF} follows.
This completes the proof.
\end{proof}

\subsection{Frostman bound}

In this subsection, we prove the following result, whose proof is similar to the corresponding part of the proof of Proposition~\ref{prop:near-AD-criterion}, where the $\alpha$-Frostman condition is established.

\begin{prop}
\label{prop:conv-ngeo-frostman}
Let $\mu$ be given by \eqref{eq:mu}.
Then, $\mu$ is $\alpha$-Frostman, i.e., \eqref{frostman} holds.
\end{prop}

\begin{proof}
Let $(x,\rho)\in \mathbb R^d\times (0,1)$.
We choose $n \ge 1$ such that
\[
{\mathfrak M}_n^{-1} \le \rho < {\mathfrak M}_{n-1}^{-1}.
\]
As in the proof of Proposition~\ref{prop:near-AD-criterion}, let the set $H$ and the set $A_h$ be defined by \eqref{eq:parents} and \eqref{eq:children}.
Consequently, by the same argument as before, we have \eqref{ttt} and \eqref{muB}.

Recall that $\bT_k\ge r^dM_k^\beta$ by \eqref{eq:T-scale}, while
\eqref{lambdan} gives $\frac{\lambda_k}{\sT_k}=cM_k^{-\alpha}.$ 
Set
\[
q_0:=\sup_{k\ge1}M_k^{\alpha-\beta}<1.
\]
Indeed, $(M_k)$ is nondecreasing, $M_k\ge2$, and $\alpha<\beta$.
Choose $c>0$ sufficiently small so that the preceding requirements hold
and
$
r^{-d}q_0+c\le r^{-d}.
$
Then
\[
\frac{1-\lambda_k}{\bT_k}
+\frac{\lambda_k}{\sT_k}
\le r^{-d}M_k^{-\beta}+cM_k^{-\alpha}
\le (r^{-d}q_0+c)M_k^{-\alpha}
\le r^{-d}M_k^{-\alpha}.
\]
Thus, by the recursion identity \eqref{eq:recursion}, we obtain
\[
\tau([w,u]) \le r^{-d} M_k^{-\alpha} \tau([w])
\]
for $k \ge 1$, $w \in \mathcal{T}_{k-1}$, and $u \in S(w)$.
Iterating this inequality yields
\begin{equation}
\label{taubd}
\tau([w]) \le r^{-d|w|} {\mathfrak M}_{|w|}^{-\alpha}, \quad  w \in \mathcal{T}.
\end{equation}

Now, we recall \eqref{muB} and follow the argument in the proof of Proposition~\ref{prop:near-AD-criterion}.
Let $h \in H$ and $v \in \mathcal{T}_{n-1} \cap X^{-1}(h)$.
From \eqref{eq:recursion}, we see
\begin{align*}
        \sum_{a \in A_h \cap S(v)} \tau([v,a])
        &\le \sum_{a \in A_h \cap \bS(v)} \tau([v,a]) + \sum_{a \in A_h \cap \sS(v)} \tau([v,a])
        \\[2pt]
&\le \tau([v]) \Big( (\bT_{n})^{-1}  \# (A_h \cap \bS(v)) +  c M_n^{-\alpha}  \#( A_h \cap \sS(v)) \Big).
\end{align*}
Since $A_h = \mathcal{D}_n \cap B({\mathfrak M}_n(x-h), C_1\rho\, {\mathfrak M}_n)$, for $ \kappa= \beta, s$, we have
\[ \# (A_h \cap {}^\kappa\! S(v)) \le \#\bigl({}^\kappa\! S(v) \cap B({\mathfrak M}_n(x-h), C_1\rho\, {\mathfrak M}_n)\bigr).
\]
Thus, by our construction based on Lemma \ref{lem:AD-regular-sampling} (see \eqref{eq:cloud-upper}), it follows that
\[
\# (A_h \cap {}^\kappa\! S(v))  \lesssim  (\rho\, {\mathfrak M}_n)^\kappa, \quad  \kappa= \beta, s.
\]
On the other hand, by \eqref{taubd} we have $\tau([v]) \le r^{-d(n-1)}{\mathfrak M}_{n-1}^{-\alpha}$.
Consequently, we obtain
\[
\sum_{a \in A_h \cap S(v)} \tau([v,a])
\lesssim
r^{-d(n-1)}
\Big( 
    {\mathfrak M}_{n-1}^{-\alpha} M_n^{-\beta} (\rho\, {\mathfrak M}_n)^\beta + {\mathfrak M}_n^{-\alpha} (\rho\, {\mathfrak M}_n)^s 
\Big)
\lesssim
r^{-d(n-1)}\rho^\alpha.
\]
For the last inequality, we used $s < \alpha < \beta$ and ${\mathfrak M}_n^{-1} \le \rho < {\mathfrak M}_{n-1}^{-1}$.

Therefore, since $\#H\lesssim 1$, combining \eqref{muB} and the preceding estimate with Lemma~\ref{lem:multiplicity}, we have $ \mu(B(x,\rho)) \lesssim \rho^\alpha$.
\end{proof}

\subsection{Dimension of the support}

Concerning $\dim_{\mathcal{H}}(\operatorname{supp}\mu) $, we prove the following.

\begin{prop}
\label{prop:conv-ngeo-support}
Let the measure $\mu$ be given by \eqref{eq:mu}.
Then, $\operatorname{supp}\mu = \pi(\mathcal{T}_\infty)$ and $\underline{\dim}_{\mathcal{M}} (\operatorname{supp}\mu) \le \beta$.
\end{prop}
\begin{proof}
We first show $\operatorname{supp}\mu = \pi(\mathcal{T}_\infty)$.
From the recursion identity \eqref{eq:recursion}, we have
\[
\min_{u \in S(w)}
\left(
\frac{1-\lambda_n}{\bT_{n} } \mathbbm{1}_{\bS(w)}(u)
+ \frac{\lambda_n}{\sT_{n} } \mathbbm{1}_{\sS(w)}(u)
\right)
\gtrsim M_n^{-\beta},
\]
which gives the inequality
\[
\tau([w,u]) \gtrsim M_n^{-\beta} \tau([w]).
\]
Hence, by this inequality, iteration, and Lemma~\ref{lem:absorption}, we have, for every $\varepsilon>0$ and every $w \in \mathcal{T}$,
\[
\tau([w]) \gtrsim_\varepsilon {\mathfrak M}_{|w|}^{-\beta-\varepsilon}.
\]

Let $x = \pi(\mathbf{w})$ with $\mathbf{w} \in \mathcal{T}_\infty$.
Then the cylinder images $\pi([\mathbf{w}|_n])$ form a nested sequence shrinking to $x$, and each has positive $\mu$-measure since $\mu(\pi([\mathbf{w}|_n])) \ge \tau([\mathbf{w}|_n]) > 0$ as shown above. 
Thus,  every neighborhood of $x$ has positive $\mu$-measure, so $x \in \operatorname{supp}\mu$. This yields $\pi(\mathcal{T}_\infty) \subset \operatorname{supp}\mu$.
The reverse inclusion $\operatorname{supp}\mu \subset \pi(\mathcal{T}_\infty)$ is established exactly as in the proof of Lemma~\ref{lem:full-support}, using the compactness of $\pi(\mathcal{T}_\infty)$.
Therefore, we conclude $\operatorname{supp}\mu = \pi(\mathcal{T}_\infty)$.

On the other hand, for every $n\ge 1$, we have
\[
\operatorname{supp}\mu = \pi(\mathcal{T}_\infty) \subset \bigcup_{w \in \mathcal{T}_n} \pi([w]).
\]
Since $s<\beta$, $\# \mathcal{T}_n \le \prod_{k=1}^n (\bT_{k} + \sT_{k}) \lesssim_\varepsilon {\mathfrak M}_n^{\beta + \varepsilon}$.
For $w \in \mathcal{T}_n$, $\pi([w])$ has diameter $\lesssim {\mathfrak M}_n^{-1}$ (see \eqref{eq:coding-basic-bounds}).
Therefore, $\operatorname{supp}\mu$ can be covered by $\lesssim_\varepsilon {\mathfrak M}_n^{\beta + \varepsilon}$ sets of diameter $\sim {\mathfrak M}_n^{-1}$.
This implies $\underline{\dim}_{\mathcal{M}} (\operatorname{supp}\mu) \le \beta$.
\end{proof}

\begin{rem}
The same covering argument also gives $\overline{\dim}_{\mathcal{M}}(\operatorname{supp}\mu) \le \beta$.
Indeed, if $\mathfrak M_{n+1}^{-1}\le\delta<\mathfrak M_n^{-1}$,
we use the cover at the finer scale $\mathfrak M_{n+1}^{-1}$.
The conclusion then follows from \eqref{mn+1}.
\end{rem}

\subsection{Fourier decay}

The following result guarantees the existence of a realization of the random measure with the desired Fourier decay.

\begin{prop}
\label{prop:conv-ngeo-fourier}
The random measure $\mu$ constructed above is of $\beta/2$-Fourier decay almost surely.
\end{prop}
\begin{proof}
The proof basically follows the lines of argument in the proof of Proposition~\ref{prop:fourier-criterion}.
The main difference is that $\tau$ is now defined by the system \eqref{swbs} and the recursion \eqref{eq:recursion}.

Recalling \eqref{eq:mu}, note that $\mu_n=(\pi_n)_\sharp\tau$ converges weakly to $\mu=\pi_\sharp\tau$.
Thus \(\widehat{\mu_n}(\xi) \to \widehat{\mu}(\xi)\).
Let ${\mathfrak D_n}(\xi)$ be defined by \eqref{eq:Dn}.
By the same argument as in the proof of Proposition~\ref{prop:fourier-criterion}, it is enough to prove that \eqref{eq:fourier-reduction} holds with $\beta$ in place of $\alpha$.

To this end, for $w \in \mathcal{W}_{n - 1}$ and $u \in \mathcal{D}_n$, we set
\begin{equation}    \label{kappauw}
\kappa(u\,|\, w) = \frac{1-\lambda_n}{\bT_{n} } \mathbbm{1}_{\bS(w)}(u) + \frac{\lambda_n}{\sT_{n} } \mathbbm{1}_{\sS(w)}(u).
\end{equation}
Then, we have $\tau([w,u]) = \kappa(u\,|\, w)\tau([w])$ by \eqref{eq:recursion}.
Also note that $\sum_{u \in \mathcal{D}_n} \kappa(u\,|\, w)=1$.
A computation analogous to that in the proof of Proposition~\ref{prop:fourier-criterion} gives
\[
{\mathfrak D_n}(\xi) = \sum_{v \in \mathcal{T}_{n-1}} \tau([v]) e^{-2\pi i \xi \cdot X(v)} Y_v(\xi),
\]
where
\[
Y_v(\xi) = \sum_{u \in \mathcal{D}_n} \kappa(u \mid v) e^{-2\pi i \xi \cdot u/{\mathfrak M}_n} - m_n(\xi/{\mathfrak M}_n).
\]

Let \(\mathcal F_0\) be the trivial \(\sigma\)-algebra, and let
\[
\mathcal F_k
:=
\sigma\bigl({}^\beta\! S(w),{}^s\!S(w): |w|<k\bigr)
\]
for \(k\ge1\).
As in the proof of Proposition~\ref{prop:fourier-criterion}, we verify
the conditional centering explicitly. The set \(\mathcal T_{n-1}\) and
the coefficients \(\tau([v])e^{-2\pi i\xi\cdot X(v)}\),
\(v\in\mathcal T_{n-1}\), are \(\mathcal F_{n-1}\)-measurable.
By the mutual independence imposed in the construction, the pairs
\(({}^\beta S(v),{}^sS(v))\), \(|v|=n-1\), are independent of
\(\mathcal F_{n-1}\) and are mutually independent. Consequently,
the variables \(Y_v(\xi)\),
\(v\in\mathcal T_{n-1}\), are conditionally independent given \(\mathcal F_{n-1}\). Moreover,
\[
\mathbb E[\kappa(u\mid v)\mid\mathcal F_{n-1}]
=\frac{1-\lambda_n}{{}^\beta T_n}
  \mathbb P(u\in{}^\beta S(v))
 +\frac{\lambda_n}{{}^sT_n}
  \mathbb P(u\in{}^sS(v))=p_n(u).
\]
(When \(s=0\), the same identity follows from \(U(v)\sim p_n\) in
Remark~\ref{rem:s=0}.) Hence
$\mathbb E[Y_v(\xi)\mid\mathcal F_{n-1}]=0$ and $|Y_v(\xi)|\le2$.
Therefore, using Lemma~\ref{lem:cond-hoeffding}, we obtain
\begin{equation}
\label{eq:hoeffding-result}
\mathbb{P}\bigl( |{\mathfrak D_n}(\xi)| \ge t \,\big|\, \mathcal{F}_{n-1} \bigr)
\le 4 \exp\Big( -\frac{t^2}{16 A_{n-1}}  \Big),
\end{equation}
where $A_k = \sum_{w \in \mathcal{T}_k} \tau([w])^2$.
Since $\tau([w,u]) = \kappa(u\,|\, w)\tau([w])$, we can write
\begin{equation}    \label{Ak}
A_k = \sum_{w \in \mathcal{T}_{k-1}} \tau([w])^2 \sum_{u \in S(w)} \kappa(u\,|\, w)^2.
\end{equation}

Recalling \eqref{kappauw}, we expand the square $\kappa(u\,|\, w)^2$.  Since $\#(\bS(w) \cap \sS(w)) \le \sT_{k}$, 
the cross term satisfies
$\frac{2\lambda_k(1-\lambda_k)}{\bT_k\,\sT_k}
\#(\bS(w)\cap\sS(w))
\le \frac{2\lambda_k(1-\lambda_k)}{\bT_k}. 
$
Thus, we get
\[
\sum_{u \in S(w)} \kappa(u\,|\, w)^2
\le  \frac{1-\lambda_k^2}{\bT_{k} } + \frac{\lambda_k^2}{\sT_{k}}.
\]
Since $s \le 2\alpha - \beta$, using \eqref{eq:T-scale} and \eqref{lambdan} and taking $c$ sufficiently small, we have
\[
\sum_{u \in S(w)} \kappa(u\,|\, w)^2
\le r^{-d} M_k^{-\beta} + C_0 r^d c^2 M_k^{-2\alpha+s} \le \theta M_k^{-\beta}
\]
for some $\theta < 1$.
Thus, combining this and \eqref{Ak} gives $A_k \le \theta M_k^{-\beta} A_{k-1}$.
Iterating this inequality, we obtain $A_k \le \theta^k {\mathfrak M}_k^{-\beta}$ for all $k$.

Substituting this into \eqref{eq:hoeffding-result}, setting
$t=B\mathfrak M_n^{-\beta/2}$, absorbing the harmless fixed factor
arising from the index shift into $c$, and using the tower property,
we obtain
\[
\mathbb{P}\bigl( {\mathfrak M}_n^{\beta/2} |{\mathfrak D_n}(\xi)| \ge B  \bigr) \le 4 \exp\bigl( -c B^2 \theta^{-n} M_n^{-\beta} \bigr).
\]
Once we have this inequality, the rest of the proof is exactly the same as in the proof of Proposition~\ref{prop:fourier-criterion}.
We therefore record only the remaining details.

Indeed, the function ${\mathfrak D_n}$ is ${\mathfrak M}_n$-periodic in each coordinate and $O_{r,d}(1)$-Lipschitz.
Let $\Lambda_n$ denote a $c_{r,d} {\mathfrak M}_n^{-\beta/2}$-net in $[0,{\mathfrak M}_n]^d$, whose cardinality $\#\Lambda_n \lesssim {\mathfrak M}_n^{d(1+\beta/2)}$.
By the same argument as before, we obtain
\[
\mathbb{P}\Big( {\mathfrak M}_n^{\beta/2} \sup_{\xi \in \mathbb{R}^d} |{\mathfrak D_n}(\xi)| > (B+1)  \Big)
\lesssim {\mathfrak M}_n^{d(1+\beta/2)} \exp\left( -c B^2 \theta^{-n} M_n^{-\beta} \right).
\]
Since $\log M_n=o(n)$, $\theta^{-n}M_n^{-\beta}$ grows exponentially,
whereas $\log\mathfrak M_n=o(n^2)$. Thus the right-hand side is summable
in $n$, and its sum tends to zero as $B\to\infty$. The union bound now
gives the analogue of \eqref{eq:fourier-reduction} with $\beta$ in place
of $\alpha$.
\end{proof}

\subsection{Completion of the proof of Proposition~\ref{prop:heavy-core-i}}

By Proposition~\ref{prop:conv-ngeo-fourier}, we may choose a realization so that $\mu$ has $\beta/2$-Fourier decay.
Since $\beta \le \dim_{\mathcal{F}}(\operatorname{supp}\mu) \le \dim_{\mathcal{H}}(\operatorname{supp}\mu)\le \underline{\dim}_{\mathcal{M}}(\operatorname{supp}\mu)$, by Proposition~\ref{prop:conv-ngeo-support} we have
\[
\dim_{\mathcal{H}}(\operatorname{supp}\mu) = \beta.
\]
This and Proposition~\ref{prop:conv-ngeo-frostman} show that $\mu \in \mathcal{P}_{\alpha,\beta}(\mathbb{R}^d)$.
Finally, Proposition~\ref{prop:conv-ngeo-heavy} gives the desired near $s$-AD regular set $F$ satisfying \eqref{eq:nearF}.
\qed
 \section{\texorpdfstring{Convolution: Geometric case $\beta\le \alpha$}{Convolution: Geometric case beta <= alpha}}
\label{conv-geo}

In this section, we prove Proposition \ref{prop:conv-geo-factorization-i}.
In the case $\alpha = \beta$, Proposition~\ref{prop:conv-geo-factorization-i} is an immediate consequence of Theorem~\ref{thm:near-AD-Salem}.
Indeed, one may take ${\aracc{\mu}}$ to be the Dirac mass at the origin and ${\rdacc{\mu}}$ to be the measure given by Theorem~\ref{thm:near-AD-Salem}.
Therefore, throughout this section we assume that $\beta < \alpha$.

To construct measures ${\aracc{\mu}}$ and ${\rdacc{\mu}}$ satisfying the factorization \eqref{eq:conv-geo-factorization-i}, the support of ${\aracc{\mu}}$ should be additively structured, as suggested by the estimate \eqref{eq:small-difference-set-i}.
Heuristically, one may regard \(\operatorname{supp}{\aracc{\mu}}\) as exhibiting approximate additive structure at the relevant scales, in the sense that it is well approximated by a lattice or, more generally, by a generalized arithmetic progression.

The factor ${\rdacc{\mu}}$ will be constructed, as in the previous section, so as to be a near $\beta$-AD regular Salem measure.  The additional requirement is that the convolution $\mu = {\aracc{\mu}} * {\rdacc{\mu}}$ must remain $\alpha$-Frostman, so ${\rdacc{\mu}}$ must be arranged in a manner that it avoids resonance with the additive structure of ${\aracc{\mu}}$.
For this purpose, we replace Lemma~\ref{lem:AD-regular-sampling} by Lemma~\ref{lem:two-partition-sampling}.
(See the comment immediately following {\it Proof of Proposition~\ref{prop:conv-geo-factorization-i}}.)

Throughout the construction below, symbols carrying the accent \(\aracc{\cdot}\) will denote the auxiliary parameters, sets, trees, and product measures associated with the arithmetic factor \({\aracc{\mu}}\), while symbols carrying the accent \(\rdacc{\cdot}\) will denote the corresponding objects associated with the random Salem factor \({\rdacc{\mu}}\).

\subsection{Coding setup} Let $(M_n)_{n\ge1}$ satisfy the conditions in Section~\ref{sec3},
including \eqref{eq:growth-condition}, and let
$({\mathfrak M}_n)_{n\ge0}$ be defined by \eqref{eq:NM}.
The sequence $(M_n)$ will be specified in the proof of
Proposition~\ref{prop:conv-geo-factorization-i} below.
We fix a sufficiently large positive integer $r$ so that \eqref{eq:r-con} holds.
Recall the set $\mathcal{D}_n^{[s]}$ defined by \eqref{eq:digit}.
In what follows we need to consider the cases $s = 1, r, r + 1$.
Here the superscript \(s\) denotes the digit-range parameter and is
unrelated to the heavy-core parameter \(s\) used in
Section~\ref{sec4}.

For $n\ge1$, let $\mathcal{W}_n^{[s]}$ denote the set of words of
length $n$ whose alphabet at level $k$ is $\mathcal{D}_k^{[s]}$
for $1\le k\le n$.
That is,
\[
\mathcal{W}_0^{[s]} = \{\varnothing\},
\qquad
\mathcal{W}_n^{[s]} = \prod_{k=1}^n \mathcal{D}_k^{[s]},
\qquad
\mathcal{W}^{[s]} = \bigcup_{n=0}^\infty \mathcal{W}_n^{[s]}.
\]
We also define the set of infinite words
\[
\mathcal{W}_\infty^{[s]} = \prod_{n=1}^\infty \mathcal{D}_n^{[s]}.
\]
Let $X^{[s]} : \mathcal{W}^{[s]} \to [0, s]^d$ and $\pi^{[s]} : \mathcal{W}_\infty^{[s]} \to [0, s]^d$ denote the associated coding maps defined by \eqref{eq:coding-maps-x} and \eqref{eq:coding-maps-p}, respectively.
Also, $\pi_n^{[s]} : \mathcal{W}_\infty^{[s]} \to [0, s]^d$ denotes the $n$-th truncated coding map given by \eqref{eq:truncated-coding-map}.
Since the defining formulas are the same for every $s$, we simply write $X$, $\pi$, and $\pi_n$ when the choice of $s$ is clear from the context.

\subsection{Additively structured factor}

Let $\aracc{M}_n$ and $Q_n$ be positive integers to be specified later such that $\aracc{M}_n Q_n < M_n$.
We define $\aracc{\mathcal D}_n \subset \mathcal{D}_n^{[1]}$ by
\begin{equation}
\label{eq:grid}
\aracc{\mathcal D}_n = Q_n \{0, 1, \ldots, \aracc{M}_n - 1\}^d.
\end{equation}
We also define an offspring system $\aracc{S}$ on $\mathcal{W}^{[1]}$ by
\begin{equation}
\label{eq:grid-offspring}
\aracc{S}(w) = \aracc{\mathcal D}_{|w| + 1}.
\end{equation}
Since the offspring set depends only on the length of the word, the associated $\aracc{\mathcal T}_n$, $\aracc{\mathcal T}$, and $\aracc{\mathcal T}_{\infty}$ are given by the product sets:
\begin{equation}
\label{tildetrees}
\aracc{\mathcal T}_n = \prod_{k=1}^n \aracc{\mathcal D}_k, \quad \aracc{\mathcal T} = \bigcup_{n=0}^\infty \aracc{\mathcal T}_n, \quad \aracc{\mathcal T}_{\infty} = \prod_{n=1}^\infty \aracc{\mathcal D}_n.
\end{equation}
Since $\aracc{S}$ is a uniform offspring system with profile $(\aracc{T}_n)_{n \ge 1}:=(\aracc{M}_n^d)_{n \ge 1}$, we may define, exactly as in Section~\ref{subsec:basic-setup}, the product measure $\aracc{\tau}$ on $\aracc{\mathcal T}_{\infty}$.
We extend $\aracc{\tau}$ to $\mathcal{W}^{[1]}$ in the trivial way, and then define the induced measures
\begin{equation}
\label{eq:lambda}
{\aracc{\mu}} = \pi_\sharp \aracc{\tau}, \quad {\aracc{\mu}}_n = (\pi_n)_\sharp \aracc{\tau}.
\end{equation}
Indeed, since \(\aracc{S}\) is uniform, the associated product measure \(\aracc{\tau}\) on \(\aracc{\mathcal T}_\infty=\prod_{n=1}^\infty \aracc{\mathcal D}_n\) is simply the countable product of the uniform probability measures on the sets \(\aracc{\mathcal D}_n\).
Thus, for a cylinder determined by \(\aracc{w}=(u_1,\dots,u_m)\in \aracc{\mathcal T}_m\), we have
\[
\aracc{\tau}([\aracc{w}])=\prod_{k=1}^m \aracc{M}_k^{-d}
\]
(cf. \eqref{eq:uniform-product-measure}).
Consequently, combining this and \eqref{eq:truncated-coding-map}, the truncated induced measure \({\aracc{\mu}}_n=(\pi_n)_\sharp\aracc{\tau}\) is given by
\begin{equation}
\label{tildemu_n}
{\aracc{\mu}}_n
=
\bconv_{k=1}^n
\Big(  \aracc{M}_k^{-d} \sum\nolimits_{u\in \aracc{\mathcal D}_k}\delta_{u/\mathfrak M_k}
\Big),
\end{equation}
where ${\mathlarger{\mathlarger {\mathlarger {\conv}}}}_{k=1}^n g_k= g_1\ast g_2\ast \dots \ast g_n$.
Indeed, this follows by expanding the convolution and using the identity \(\delta_a*\delta_b=\delta_{a+b}\), together with the fact that \(\aracc{\tau}\) is the product of the uniform measures on the sets \(\aracc{\mathcal D}_k\).

As a result, it follows from \eqref{tildemu_n} that
\begin{equation}
\label{supp-tildemu}
\supp {\aracc{\mu}}_n
=
\Big\{
\sum_{k=1}^n \frac{u_k}{\mathfrak M_k}
:
u_k\in \aracc{\mathcal D}_k
\Big\}.
\end{equation}

\begin{lem}
\label{lem:small-difference-criterion}
Suppose $\aracc{M}_n \lesssim M_n^{(\alpha - \beta)/d}$ for all $n\ge 1$.
Then \eqref{eq:small-difference-set-i} holds.
\end{lem}
\begin{proof}
Let $E = \operatorname{supp}{\aracc{\mu}}$ and $E_n = \operatorname{supp}{\aracc{\mu}}_n$.
Applying \eqref{eq:coding-basic-bounds} with $r = 1$, we have
\begin{equation}
\label{eq:approx-supp}
(E - E)_{{\mathfrak M}_n^{-1}} \subset (E_n - E_n)_{(1 + 2\sqrt{d}){\mathfrak M}_n^{-1}}.
\end{equation}
By \eqref{supp-tildemu}, we have $E_n - E_n = \{ \sum_{k=1}^n \frac{h_k}{{\mathfrak M}_k} : h_k \in \aracc{\mathcal D}_k - \aracc{\mathcal D}_k\}$.
Thus, it follows that
\begin{equation}    \label{enen}
\#(E_n - E_n) \le \prod_{k=1}^n \#(\aracc{\mathcal D}_k - \aracc{\mathcal D}_k).
\end{equation}
Moreover, by the lattice structure of $\aracc{\mathcal D}_k$ and the assumption $\aracc{M}_n \lesssim M_n^{(\alpha - \beta)/d}$ we have $\#(\aracc{\mathcal D}_k - \aracc{\mathcal D}_k) \le 2^d \# \aracc{\mathcal D}_k = 2^d \aracc{M}_k^d \lesssim M_k^{\alpha - \beta}$.
Thus,
\[
\prod_{k=1}^{n} \#(\aracc{\mathcal D}_k - \aracc{\mathcal D}_k) \le \prod_{k=1}^{n} CM_k^{\alpha - \beta}
\]
for a constant $C$.
Therefore, by Lemma~\ref{lem:absorption}, we obtain
\[
\#(E_n - E_n) \lesssim_\varepsilon {\mathfrak M}_n^{\alpha - \beta + \varepsilon}.
\]

Since $E_n-E_n\subset{\mathfrak M}_n^{-1}\mathbb Z^d$, distinct
points are separated by at least ${\mathfrak M}_n^{-1}$.
Hence, by \eqref{eq:approx-supp},
$|(E-E)_{{\mathfrak M}_n^{-1}}|
\lesssim {\mathfrak M}_n^{-d}\#(E_n-E_n)$.
Combining this and the inequality above, we obtain
\begin{equation}    \label{e-e}
|(E - E)_{{\mathfrak M}_n^{-1}}|  \lesssim_\varepsilon {\mathfrak M}_n^{-d + \alpha - \beta + \varepsilon}
\end{equation}
for any $\varepsilon>0$. 
As in the proof of Lemma \ref{lower-ad}, the inequality \eqref{e-e}
implies \eqref{eq:small-difference-set-i}, thanks to the growth
condition \eqref{eq:growth-condition}. Indeed, set $a=d-\alpha+\beta>0$, and fix $\varepsilon>0$.
Choose $0<\varepsilon_1<a$ and $\varepsilon_2>0$ so that
$(1-\varepsilon_2)(a-\varepsilon_1)\ge a-\varepsilon$.
For $\delta>0$ sufficiently small, choose $n\ge1$ such that
$\mathfrak M_{n+1}^{-1}\le2\delta<\mathfrak M_n^{-1}$.
By \eqref{mn+1},
$\mathfrak M_n^{-1}\lesssim_{\varepsilon_2}
\delta^{1-\varepsilon_2}$.
Hence, using \eqref{e-e} with $\varepsilon=\varepsilon_1$, we obtain
\[
|(E-E)_{2\delta}|
\le |(E-E)_{\mathfrak M_n^{-1}}|
\lesssim_{\varepsilon_1}
\mathfrak M_n^{-a+\varepsilon_1}
\lesssim_{\varepsilon_1,\varepsilon_2}
\delta^{(1-\varepsilon_2)(a-\varepsilon_1)}
\le \delta^{a-\varepsilon}.
\]
This proves \eqref{eq:small-difference-set-i}.
\end{proof}

We now investigate near AD-regularity of the measure ${\aracc{\mu}}$ induced by uniform offspring systems.

\begin{prop}
\label{prop:seq-lower}
Let $S = (S(w))_{w \in \mathcal{W}^{[s]}}$ be a uniform offspring system on $\mathcal{W}^{[s]}$ with profile $(T_n)_{n \ge 1}$.
Let $\eta$ be the measure induced by $S$ (see \eqref{eq:induced-measure}).
If $T_n \lesssim M_n^\gamma$, then for any $\varepsilon>0$ we have
\begin{equation}
\label{etaB}
\eta\big(B(x,\rho)\big) \gtrsim \rho^{\gamma+\varepsilon}, \quad \forall (x,\rho)\in \supp \eta\times (0,1).
\end{equation}
\end{prop}

\begin{proof}
Let $x \in \operatorname{supp}\eta$.
Then $x = \pi(\mathbf{w})$ for some $\mathbf{w} \in \mathcal{T}_\infty$ by Lemma~\ref{lem:full-support}.
Then, using \eqref{eq:coding-basic-bounds}, we have $ \eta\big( B(x, s\sqrt{d}{\mathfrak M}_n^{-1}) \big) \ge \eta\big( \pi([\mathbf{w}|_n]) \big)$.
Combining this with \eqref{eq:uniform-product-measure} gives
\[
\eta\big( B(x, s\sqrt{d}{\mathfrak M}_n^{-1}) \big) \ge \frac{1}{\prod_{k=1}^n T_k}.
\]
Since $T_n \lesssim M_n^\gamma$, using Lemma~\ref{lem:absorption}, we have $\eta\big( B(x, s\sqrt{d}{\mathfrak M}_n^{-1}) \big) \gtrsim_\varepsilon {\mathfrak M}_n^{-\gamma - \varepsilon}$ for $x \in \supp \eta$ and $n \ge 1$. Thus, the argument of Lemma~\ref{lower-ad} applies verbatim, with
$(\mu,F,\alpha,r)$ replaced by
$(\eta,\operatorname{supp}\eta,\gamma,s)$, and yields \eqref{etaB}.
\end{proof}

To verify Frostman-type upper bounds, we will use a simple sparsity condition on the offspring sets.
The following notion measures how thinly a subset of \(\mathbb Z^d\) is distributed at a given block scale, and will allow us to control the number of points contained in Euclidean balls.

\begin{defn}
\label{def:q-block-sparse}
Let $q \in \mathbb{N}$.
We call a set of the form
\[
b + \{0, 1, \ldots, q - 1\}^d, \quad b \in q \mathbb{Z}^d
\]
a \emph{$q$-block}.
A set $A \subset \mathbb{Z}^d$ is \emph{$q$-block sparse} if it intersects each $q$-block in at most one point.
\end{defn}

\begin{prop}
\label{prop:frostman-block}
Let $\gamma \in(0,d)$ and $(b_n)_{n \ge 1}$ be a sequence of integers.
Let $\eta$ be the measure induced by the uniform offspring system $S$ with profile $(T_n)_{n \ge 1}$.
Suppose that
\[
T_n \ge s^d M_n^\gamma, \quad b_n \gtrsim M_n^{1 - \gamma/d}, \quad \forall n\ge 1.
\]
Suppose the set $S(w)$ is $b_n$-block sparse for $w \in \mathcal{T}_{n - 1}$ and $n \ge 1$.
Then $\eta$ is $\gamma$-Frostman.
\end{prop}
\begin{proof}
We need to show
\[
\eta\big(B(x,\rho)\big) \lesssim \rho^\gamma
\]
for $ (x,\rho)\in \mathbb R^d\times (0,1)$.
To do this, it suffices to verify \eqref{eq:discrete-upper} with \(\alpha\) replaced by \(\gamma\), since the proof of the Frostman estimate in Proposition~\ref{prop:near-AD-criterion} then applies verbatim.

To this end, fix $n \ge 1$, $w \in \mathcal{T}_{n - 1}$, and let $y \in \mathbb{R}^d$ and $R \in [1, M_n]$.
By $b_n$-block sparsity the set $S(w) \cap B(y, R)$ can be covered by $\lesssim (1 + R/b_n)^d$ many $b_n$-blocks, and each such block contains at most one point of $S(w)$.
Hence
\[
\#(S(w) \cap B(y, R)) \lesssim (1 + R/b_n)^d.
\]
If $1 \le R < b_n$, then $(1 + R/b_n)^d \lesssim 1 \lesssim R^\gamma$.
If $b_n \le R \le M_n$, then $(1 + R/b_n)^d \lesssim b_n^{-d} R^d \lesssim M_n^{-d + \gamma} R^d \le R^\gamma$.
Consequently, \eqref{eq:discrete-upper} follows.
\end{proof}

Putting Propositions \ref{prop:seq-lower} and \ref{prop:frostman-block} together, we obtain the following corollary concerning the measure ${\aracc{\mu}}$.

\begin{cor}
\label{cor:factor-lambda}
Let \((\aracc{M}_n)_{n\ge1}\) and \((Q_n)_{n\ge1}\) be sequences of positive integers such that \(\aracc{M}_n Q_n < M_n\) for all \(n\ge1\).
Suppose that $ Q_n \gtrsim M_n^{1 - (\alpha - \beta)/d}$ and
\[
M_n^{(\alpha - \beta)/d} \le \aracc{M}_n \le C M_n^{(\alpha - \beta)/d}
\]
for a constant $C \ge 1$.
Then ${\aracc{\mu}}$ is near $(\alpha - \beta)$-AD regular.
\end{cor}
\begin{proof}
Recall that the measure ${\aracc{\mu}}$ is induced by the uniform offspring system $\aracc{S}$ on $\mathcal{W}^{[1]}$ with profile $(\aracc{T}_n)_{n \ge 1}=(\aracc{M}_n^d)_{n \ge 1}$.
Since $\aracc{T}_n \lesssim M_n^{\alpha-\beta}$, by Proposition~\ref{prop:seq-lower} we see that ${\aracc{\mu}}$ is near $(\alpha-\beta)$-lower regular, i.e., \eqref{etaB} holds with $\eta={\aracc{\mu}}$ and $\gamma=\alpha-\beta$ for $\varepsilon>0$.

Also, it is clear from Definition \ref{def:q-block-sparse} that the grid set $\aracc{\mathcal D}_n$ is $Q_n$-block sparse.
Consequently, so is $\aracc{S}(w) \subset \aracc{\mathcal D}_n$ for $w\in \aracc{\mathcal T}_{n-1}$.
Since $Q_n \gtrsim M_n^{1 - (\alpha - \beta)/d}$, by Proposition~\ref{prop:frostman-block} it follows that ${\aracc{\mu}}$ is $(\alpha-\beta)$-Frostman.

Therefore, ${\aracc{\mu}}$ is near $(\alpha - \beta)$-AD regular.
\end{proof}

\subsection{Convolution of offspring systems and Salem factor}
\label{sec:convolution-offspring}

The next step is to construct the measure \({\rdacc{\mu}}\), which is a near \(\beta\)-AD regular Salem measure.
For this purpose, it suffices to construct a random uniform offspring system satisfying the hypotheses of Propositions~\ref{prop:seq-lower} and \ref{prop:frostman-block}, to which Proposition~\ref{prop:fourier-criterion} also applies.

In the geometric construction, however, one must additionally ensure that the convolution $\mu={\aracc{\mu}} * {\rdacc{\mu}}$ remains \(\alpha\)-Frostman.
Since the additively structured factor \({\aracc{\mu}}\) is supported on sets lying in the lattice \(Q_n\mathbb Z^d\), the offspring sets of \(\rdacc{S}\) must be arranged so as to avoid collisions modulo \(Q_n\).
The following notion is introduced precisely for this purpose.

\begin{defn}
Let \(q, Q \in \mathbb{N}\) satisfy \(q \mid Q\).
We call a set of the form
\[
b + \{0, 1, \ldots, q - 1\}^d + Q \mathbb{Z}^d, \quad b \in q \mathbb{Z}^d
\]
a \emph{$q$-block modulo $Q$}.
We say that a set $A \subset \mathbb{Z}^d$ is \emph{$q$-block sparse modulo $Q$} if $A$ meets each $q$-block modulo $Q$ in at most one point.
\end{defn}

Note that $q$-block sparsity modulo $Q$ is a stronger condition than $q$-block sparsity.
We recall that the accents \(\aracc{\cdot}\) and \(\rdacc{\cdot}\) indicate the auxiliary objects associated with \({\aracc{\mu}}\) and \({\rdacc{\mu}}\), respectively.

\begin{lem}
\label{lem:residue-sep-add}
Let \(q,Q \in \mathbb{N}\) satisfy \(q \mid Q\).
Let \(\aracc{A},\rdacc{A} \subset \mathbb{Z}^d\) be finite sets such that $\aracc{A} \subset Q\mathbb{Z}^d$, and \(\rdacc{A}\) is \(q\)-block sparse modulo \(Q\).
Then the addition map
\[
(\aracc{u}, \rdacc{u}) \mapsto \aracc{u} + \rdacc{u}
\]
from $\aracc{A} \times \rdacc{A}$ to $ \aracc{A}+\rdacc{A}$ is bijective.
Moreover, \(\aracc{A} + \rdacc{A}\) is $q$-block sparse.
\end{lem}
\begin{proof}
We first show that the map $(u, v) \mapsto u + v$ is bijective. Surjectivity follows from the definition of the sumset, so it remains
to prove injectivity.
Suppose first that $\aracc{u} + \rdacc{u} = \aracc{v} + \rdacc{v}$ with \(\aracc{u}, \aracc{v}\in \aracc{A}\) and \(\rdacc{u}, \rdacc{v} \in \rdacc{A}\).
Since \(\aracc{u} , \aracc{v} \in Q\mathbb{Z}^d\), we have
\[
\rdacc{u} - \rdacc{v} = \aracc{v}  - \aracc{u}  \in Q\mathbb{Z}^d.
\]
Thus, \(\rdacc{u}\) and \(\rdacc{v}\) lie in the same \(q\)-block modulo \(Q\).
Now, \(q\)-block modulo $Q$ sparsity of \(\rdacc{A}\) implies \(\rdacc{u}=\rdacc{v}\), hence \(\aracc{u} =\aracc{v} \).

Secondly, we verify that \(\aracc{A}+\rdacc{A}\) is \(q\)-block sparse.
Let
\[
B=b+\{0,1,\ldots,q-1\}^d,
\qquad b\in q\mathbb Z^d,
\]
be a \(q\)-block, and suppose that
\(z_1=\aracc{u}+\rdacc{u}\) and \(z_2=\aracc{v}+\rdacc{v}\)
are two points of \(\aracc{A}+\rdacc{A}\) lying in \(B\).
Since \(\aracc{u},\aracc{v}\in Q\mathbb Z^d\), both \(\rdacc{u}\) and
\(\rdacc{v}\) belong to the same \(q\)-block modulo \(Q\), namely
\(B+Q\mathbb Z^d\).
Thus \(\rdacc{u}=\rdacc{v}\) by the assumed modulo-\(Q\) sparsity.
Consequently,
\[
\aracc{u}-\aracc{v}=z_1-z_2.
\]
The left-hand side belongs to \(Q\mathbb Z^d\), whereas the right-hand
side belongs to \(\{-(q-1),\ldots,q-1\}^d\).
Since \(q\le Q\), both sides vanish. Hence \(z_1=z_2\).
\end{proof}

We are now ready to define the convolution of two uniform offspring systems.  Let $\rdacc{S}$ be a uniform offspring system on $\mathcal{W}^{[s]}$ with profile $(\rdacc{T}_{n})_{n \ge 1}$, and recall that $\aracc{S}$ is defined by \eqref{eq:grid-offspring}.
Let $(q_{n})_{n \ge 1}$ be a sequence of positive integers such that $q_n \mid Q_n$.

\subsubsection{Convolution of offspring systems}
\label{sec:convolution}
Suppose that \(\rdacc{S}(\rdacc{w})\) is \(q_n\)-block sparse modulo \(Q_n\) for all \(\rdacc{w}\in \mathcal W^{[s]}_{n-1}\) and \(n\ge1\).\footnote{Equivalently, \(\rdacc{S}(\rdacc{w})\) is \(q_{|\rdacc{w}|+1}\)-block sparse modulo \(Q_{|\rdacc{w}|+1}\) for all \(\rdacc{w}\in\mathcal W^{[s]}\).}
Let
\[
\rdacc{\mathcal T}=\bigcup_{n=0}^\infty \rdacc{\mathcal T}_n
\]
be the tree associated with the system \(\rdacc{S}\).
We recursively define a uniform offspring system \(S=\aracc{S}*\rdacc{S}\) on \(\mathcal W^{[1+s]}\) and its associated tree \(\mathcal T\) as follows.

Set $\mathcal T_0=\{\varnothing\}$.
Assume that \(\mathcal T_{n-1}\) has been defined and that every \(w\in\mathcal T_{n-1}\) admits a unique representation
\[
w=\aracc{w}+\rdacc{w}, \qquad   (\aracc{w}, \rdacc{w}) \in \aracc{\mathcal T}_{n-1}\times \rdacc{\mathcal T}_{n-1}.
\]

$\bullet$
For \(w\in\mathcal T_{n-1}\), define
\begin{equation}
\label{eq:sas}
S(w)=(\aracc{S}*\rdacc{S})(w):=\aracc{S}(\aracc{w})+\rdacc{S}(\rdacc{w}),
\end{equation}
and
\[
\mathcal T_n:=\{(w,u): w\in\mathcal T_{n-1},\ u\in S(w)\}.
\]
Since \(\aracc{S}(\aracc{w})=\aracc{\mathcal D}_n\subset Q_n\mathbb Z^d\) and \(\rdacc{S}(\rdacc{w})\) is \(q_n\)-block sparse modulo \(Q_n\), Lemma~\ref{lem:residue-sep-add} shows that the addition map $(\aracc{u},\rdacc{u})\mapsto \aracc{u}+\rdacc{u}$ from \(\aracc{S}(\aracc{w})\times \rdacc{S}(\rdacc{w})\) to \(S(w)\) is bijective.
Hence
\[
\#S(w)=\#\aracc{S}(\aracc{w})\,\#\rdacc{S}(\rdacc{w})=\aracc{T}_n\rdacc{T}_n,
\]
and every element of \(\mathcal T_n\) again admits a unique representation $w'=\aracc{w}'+\rdacc{w}'$ with \(\aracc{w}'\in \aracc{\mathcal T}_n\) and \(\rdacc{w}'\in \rdacc{\mathcal T}_n\).
This shows that the recursive construction is well defined.

$\bullet$ For \(w\in \mathcal W_{n-1}^{[1+s]}\setminus \mathcal T_{n-1}\), let \(S(w)\) be the set consisting of the first \(\aracc{T}_n\rdacc{T}_n\) elements of \(\mathcal D_n^{[1+s]}\) in lexicographic order.
This auxiliary definition, which does not affect \(\mathcal T_n\), is introduced only to make \(S\) an offspring system on the whole space \(\mathcal W^{[1+s]}\).

Therefore, \(S=\aracc{S}*\rdacc{S}\) is a uniform offspring system on \(\mathcal W^{[1+s]}\) with profile
\[
(T_n)_{n\ge1}:=(\aracc{T}_n\rdacc{T}_n)_{n\ge1},
\]
and, by construction, its associated tree is
\[
\mathcal T=\bigcup_{n=0}^\infty \mathcal T_n.
\]

Moreover, we have the following.

\begin{prop}
\label{prop:conv-offspring-system}
Let \(\mu\) denote the induced measure of the uniform offspring system \(S=\aracc{S}\ast\rdacc{S}\) with profile \((T_n)_{n\ge 1}=(\aracc{T}_{n}\rdacc{T}_{n})_{n \ge 1}\).
Then $\mu = {\aracc{\mu}} * {\rdacc{\mu}}$.
\end{prop}
\begin{proof}
Recalling \eqref{eq:uniform-product-measure}, let $\aracc{\tau}, \rdacc{\tau}$ be the product measures on \(\aracc{\mathcal T}_{\infty}, \rdacc{\mathcal T}_\infty\) inducing the measures \({\aracc{\mu}}, {\rdacc{\mu}}\), respectively.
Let \(\tau\) denote the product measure on \(\mathcal{T}_\infty\) that induces the measure \(\mu\). Define
\[
\Theta : \aracc{\mathcal T}_\infty  \times \rdacc{\mathcal T}_\infty \to \mathcal{T}_\infty,
\qquad
\Theta(\aracc{\mathbf{w}},\rdacc{\mathbf{w}}) = \aracc{\mathbf{w}} + \rdacc{\mathbf{w}}.
\]

By the recursive construction and Lemma~\ref{lem:residue-sep-add}, \(\Theta\) is a bijection.
Indeed, by the construction of \(\mathcal T\), every finite word \(w\in\mathcal T_n\) admits a unique representation \(w=\aracc{w}+\rdacc{w}\) with \(\aracc{w}\in\aracc{\mathcal T}_n\) and \(\rdacc{w}\in\rdacc{\mathcal T}_n\).
Applying this to the prefixes of \(\mathbf w\in\mathcal T_\infty\), we obtain unique words \(\aracc{w}^{(n)}\in\aracc{\mathcal T}_n\) and \(\rdacc{w}^{(n)}\in\rdacc{\mathcal T}_n\) such that
\[
\mathbf w|_n=\aracc{w}^{(n)}+\rdacc{w}^{(n)}.
\]
By uniqueness, these decompositions are compatible under restriction to shorter prefixes, and therefore determine unique infinite words $\aracc{\mathbf{w}}\in\aracc{\mathcal T}_\infty$, $\rdacc{\mathbf{w}}\in\rdacc{\mathcal T}_\infty$ satisfying $\mathbf w=\aracc{\mathbf{w}}+\rdacc{\mathbf{w}}$.
Hence \(\Theta\) is bijective.
Since each finite prefix of the output depends only on the corresponding
input prefixes, \(\Theta\) is continuous in the product topology and
hence Borel measurable.

Moreover, let $\aracc{w}, \rdacc{w}$ be finite words of the same length.  
By the coordinatewise definition of $\Theta$ and the uniqueness of the
prefix decomposition,
\[
\Theta([\aracc{w}]\times[\rdacc{w}])
=
[\aracc{w}+\rdacc{w}].
\]
Also, note that $ \aracc{\tau}([\aracc{w}])=(\prod_{k=1}^{|\aracc{w}|} \aracc{T}_{k})^{-1}$, $\rdacc{\tau}([\rdacc{w}]) = (\prod_{k=1}^{|\aracc{w}|} \rdacc{T}_{k})^{-1} $, and $ \tau([\aracc{w} + \rdacc{w}]) = (\prod_{k=1}^{|\aracc{w}|} \aracc{T}_{k}\rdacc{T}_{k})^{-1}$.
Thus, $\tau([\aracc{w} + \rdacc{w}]) = \aracc{\tau}([\aracc{w}])\rdacc{\tau}([\rdacc{w}])$. 
Since cylinder sets generate the Borel $\sigma$-algebra of
$\mathcal T_\infty$, the preceding identities imply
\[
\tau = \Theta_\sharp(\aracc{\tau} \otimes \rdacc{\tau}).
\]

Let \(\aracc{\pi} : \aracc{\mathcal T}_\infty \to [0,1]^d\), \(\rdacc{\pi} : \rdacc{\mathcal T}_\infty \to [0,s]^d\), and \(\pi : \mathcal{T}_\infty \to [0, 1+s]^d\) be the coding maps associated with the systems $\aracc{S}$, $\rdacc{S}$ and $S$, respectively.
From the linearity of coding maps, it is clear that $\pi(\aracc{\mathbf{w}} + \rdacc{\mathbf{w}})= \aracc{\pi}(\aracc{\mathbf{w}}) + \rdacc{\pi}(\rdacc{\mathbf{w}})$.
Hence,
\[
\pi \circ \Theta = \Sigma \circ (\aracc{\pi},\rdacc{\pi}),
\]
where \(\Sigma(x, y) = x + y\).
Combining these identities, we obtain
\[
\mu
= \pi_\sharp \tau
= \pi_\sharp \Theta_\sharp(\aracc{\tau} \otimes \rdacc{\tau})
= \Sigma_\sharp \bigl( (\aracc{\pi})_\sharp \aracc{\tau} \otimes (\rdacc{\pi})_\sharp \rdacc{\tau} \bigr)
= \Sigma_\sharp ({\aracc{\mu}} \otimes {\rdacc{\mu}}).
\]
Finally, note that $\Sigma_\sharp ({\aracc{\mu}} \otimes {\rdacc{\mu}}) = {\aracc{\mu}} * {\rdacc{\mu}}$.
Thus, the desired identity $\mu = {\aracc{\mu}} * {\rdacc{\mu}}$ follows.
\end{proof}

\subsubsection{\texorpdfstring{Properties of the measures ${\rdacc{\mu}}$ and ${\aracc{\mu}}\ast{\rdacc{\mu}}$}{Properties of the measures mu-rd and mu-ar * mu-rd}}
Combining Propositions~\ref{prop:seq-lower}, \ref{prop:frostman-block}, \ref{prop:conv-offspring-system}, and Proposition~\ref{prop:fourier-criterion}, we obtain the following corollary, which we will use to prove Proposition~\ref{prop:conv-geo-factorization-i}.

In order to obtain measures satisfying the desired properties, we specify the sequences $(\aracc{M}_n)_{n\ge 1}$, $(Q_n)_{n\ge 1}$, and $(q_n)_{n\ge 1}$ so that
\begin{gather}
M_n^{(\alpha - \beta)/d}
\le
\aracc{M}_n \le C M_n^{(\alpha-\beta)/d} \quad \text{for some constant } C;
\label{Mn-}
\\
Q_n \sim M_n^{1-(\alpha-\beta)/d},     \quad     Q_n \aracc{M}_n < M_n;
\label{Qn}
\\
\quad \ q_n \mid Q_n,     \quad          q_n \sim M_n^{1-\alpha/d}.
\label{qn}
\end{gather}

Let $\aracc{S}$ be the offspring system on $\mathcal{W}^{[1]}$ defined in \eqref{eq:grid-offspring}, and let ${\aracc{\mu}}$ be its induced measure.

\begin{cor}
\label{cor:factor-nu}
Let $(\aracc{M}_n)_{n\ge 1}$, $(Q_n)_{n\ge 1}$, and $(q_n)_{n\ge 1}$ be sequences of integers satisfying \eqref{Mn-}, \eqref{Qn}, and \eqref{qn}.
Let ${\aracc{\mu}}$ be the measure induced by the system $\aracc{S}$ defined in \eqref{eq:grid-offspring}.
Let $(b_n)_{n\ge 1}$ be a sequence of integers such that
\begin{equation}    \label{bn}
b_n \sim M_n^{1-\beta/d}.
\end{equation}
Let ${\rdacc{S}}$ be a random uniform offspring system on $\mathcal{W}^{[r]}$ with profile $(\rdacc{T}_n)_{n \ge 1}$ satisfying
\begin{equation}      \label{rmn}
(r+1)^d M_n^\beta \le \rdacc{T}_n \le C (r+1)^d M_n^\beta
\end{equation}
for some constant $C \ge 1$.
Suppose that the family $\{ {\rdacc{S}}(w) : w \in \mathcal{W}^{[r]} \}$ is independent (as set-valued random variables) and the following hold for $w \in \mathcal{W}_{n-1}^{[r]}$ and $n \ge 1$:
\begin{gather}
\text{every realization of  $\rdacc{S}(w)$ is $b_n$-block sparse and}
\label{barS}
\\
\text{$q_n$-block sparse modulo $Q_n$.}
\nonumber
\\
\text{$\displaystyle \mathbb{E}\Big[ \frac{1}{\rdacc{T}_n} \mathbbm{1}_{{\rdacc{S}}(w)}(u) \Big] = p_n(u)$ for all $u \in \mathcal{D}_n^{[r]}$.}
\label{ET}
\end{gather}
Then, every realization of the random measure ${\rdacc{\mu}}$ induced by the system $\rdacc{S}$ is near $\beta$-AD regular, and ${\rdacc{\mu}}$ is of $\beta/2$-Fourier decay almost surely.
Let $\mu$ be the measure induced by the convolution offspring system $\aracc{S}\ast \rdacc{S}$ (see Section \ref{sec:convolution}).
Then, $\mu = {\aracc{\mu}} * {\rdacc{\mu}}$ and $\mu$ is $\alpha$-Frostman.
Moreover, $\mu \in \mathcal{P}_{\alpha,\beta}(\mathbb{R}^d)$ almost surely.
\end{cor}

Here, \(p_n\) denotes the probability mass function on $\mathcal{D}_n^{[r]}$ defined in the same manner as in \eqref{eq:pmf}.

\begin{proof}
By \eqref{rmn} and \eqref{barS}, Propositions~\ref{prop:seq-lower} and \ref{prop:frostman-block} with \(\gamma=\beta\) show that every realization of \({\rdacc{\mu}}\) is near \(\beta\)-AD regular.

Since the family $\{ {\rdacc{S}}(w) : w \in \mathcal{W}^{[r]} \}$ is independent, the family of $\sigma$-algebras $\{ \sigma({\rdacc{S}}(w)) : w \in \mathcal{W}^{[r]} \}$ is independent.
Thus, Proposition~\ref{prop:fourier-criterion} applies once we verify the remaining hypotheses.
The size condition \eqref{pro:con1} follows from \eqref{rmn}.
The condition \eqref{eq:correct-marginal} follows from \eqref{ET}; note that this is stated for all \(w\in\mathcal W_{n-1}^{[r]}\), and hence in particular for all \(w\in\rdacc{\mathcal T}_{n-1}\).
Therefore \({\rdacc{\mu}}\) is of \(\beta/2\)-Fourier decay almost surely.

By Proposition~\ref{prop:conv-offspring-system}, we have $\mu={\aracc{\mu}} * {\rdacc{\mu}}$.
It remains to show that \(\mu\) is \(\alpha\)-Frostman.
Let \(w\in\mathcal T_{n-1}\), and write uniquely
$w=\aracc{w}+\rdacc{w}$ with \(\aracc{w}\in\aracc{\mathcal T}_{n-1}\) and \(\rdacc{w}\in\rdacc{\mathcal T}_{n-1}\).
Then
\[
(\aracc{S}*\rdacc{S})(w)=\aracc{S}(\aracc{w})+\rdacc{S}(\rdacc{w}).
\]
Since \(\aracc{S}(\aracc{w})\subset Q_n\mathbb Z^d\) and every realization of \(\rdacc{S}(\rdacc{w})\) is \(q_n\)-block sparse modulo \(Q_n\), Lemma~\ref{lem:residue-sep-add} shows that \((\aracc{S}*\rdacc{S})(w)\) is \(q_n\)-block sparse. Recalling $\aracc{T}_n= \aracc{M}_n^d$, note that the profile of \(\aracc{S}*\rdacc{S}\) is \((\aracc{M}_n^d\rdacc{T}_n)_{n\ge1}\), and by \eqref{Mn-} and \eqref{rmn},
\[
\aracc{M}_n^d\rdacc{T}_n \ge (r+1)^d M_n^\alpha.
\]

Since \(q_n\sim M_n^{1-\alpha/d}\), Proposition~\ref{prop:frostman-block} implies that the induced measure \(\mu\) is \(\alpha\)-Frostman.
Since $|\widehat{{\aracc{\mu}}}|\le 1$, $|\widehat\mu|=|\widehat{{\aracc{\mu}}}\,\widehat{{\rdacc{\mu}}}|\le |\widehat{{\rdacc{\mu}}}|$, so $\mu$ is of $\beta/2$-Fourier decay almost surely.
Thus, $\mu \in \mathcal{P}_{\alpha,\beta}(\mathbb{R}^d)$ almost surely.
\end{proof}

\subsection{Proof of Proposition~\ref{prop:conv-geo-factorization-i}}

To carry out the probabilistic construction required in Corollary~\ref{cor:factor-nu}, we make use of the following lemma. The lemma follows from a standard network-flow argument; we include
the proof in Section~\ref{sec:twopart} below.

\begin{lem}[Two-partition sampling]\label{lem:two-partition-sampling}
Let \(\mathcal{D}\) be a finite set equipped with a probability mass function \(p\).
Suppose that \(\mathcal{D}\) admits two partitions
\[
\mathcal{D}
=
\bigsqcup_{B \in \mathcal{B}} B
=
\bigsqcup_{R \in \mathcal{R}} R.
\]
Suppose that
\[
\max_{B \in \mathcal{B}} \sum_{u \in B} p(u) \le   {1}/{T}, \qquad
\max_{R \in \mathcal{R}} \sum_{u \in R} p(u) \le  {1}/{T}
\]
for some \(T \in \mathbb{N}\).
Then,
\[
\Omega
:=
\Big\{
S \subset \mathcal{D}
:
\#S = T,
\max_{B \in \mathcal{B}} \#(S \cap B) \le 1,\
\max_{R \in \mathcal{R}} \#(S \cap R) \le 1
\Big\}  \neq \emptyset.
\]
Moreover, there exists a probability measure \(\mathbb{P}\) on \(\Omega\) such that, for every \(u \in \mathcal{D}\),
\begin{equation}
\label{eq:two-partition-inclusion-probability}
\mathbb{P}\bigl(\{S \in \Omega : u \in S\}\bigr) = T p(u).
\end{equation}
\end{lem}

Assuming this lemma for the moment, we prove Proposition~\ref{prop:conv-geo-factorization-i}.

\begin{proof}[Proof of Proposition~\ref{prop:conv-geo-factorization-i}]
Take a sufficiently large natural number $M$, and set $M_n = M + n$.
Then, $(M_n)_{n\ge 1}$ satisfies the growth condition \eqref{eq:growth-condition}.
We take the parameters appearing in the construction of the systems $\aracc{S}$ and $\rdacc{S}$ as follows:
\begin{equation}
\label{parameter}
\begin{gathered}
\aracc{M}_n = \ceil*{M_n^{(\alpha - \beta)/d}}, 
\quad
q_n = \floor*{(100r)^{-2} M_n^{1 - \alpha/d}}, 
\quad 
Q_n = q_n \floor*{(50r)^2 M_n^{\beta/d}},
\\
\rdacc{T}_n = \ceil*{(r + 1)^d M_n^\beta}, 
\quad
b_n = \floor*{(2r)^{-1} M_n^{1 - \beta/d}}.
\end{gathered}
\end{equation}
Clearly, \(q_n\mid Q_n\), and the floor and ceiling estimates give
\(Q_n\sim M_n^{1-(\alpha-\beta)/d}\) and
\(Q_n\aracc{M}_n/M_n\to 1/4\). Thus, for \(M\) sufficiently large,
\(Q_n\aracc{M}_n<M_n\), and the sequences \((\aracc{M}_n)\),
\((Q_n)\), and \((q_n)\) satisfy \eqref{Mn-}, \eqref{Qn}, and
\eqref{qn}, while \((b_n)\) and \((\rdacc{T}_n)\) satisfy \eqref{bn}
and \eqref{rmn}, respectively.
In particular, Corollary~\ref{cor:factor-lambda} yields that the measure \({\aracc{\mu}}\) induced by the offspring system \(\aracc{S}\) is near \((\alpha-\beta)\)-AD regular.
Moreover, by Lemma~\ref{lem:small-difference-criterion}, \({\aracc{\mu}}\) satisfies \eqref{eq:small-difference-set-i}.

We now construct a random uniform offspring system \(\rdacc{S}\) satisfying the hypotheses of Corollary~\ref{cor:factor-nu}.
For this purpose, we consider two partitions $\mathcal{B}_n$ and $\mathcal{R}_n$ of $\mathcal{D}_n^{[r]}$.
The partition $\mathcal{B}_n$ is obtained by decomposing $\mathcal{D}_n^{[r]}$ into $b_n$-blocks, while $\mathcal{R}_n$ is obtained by decomposing $\mathcal{D}_n^{[r]}$ into $q_n$-blocks modulo $Q_n$.

Note that
\[
p_n(u)\le M_n^{-d}
\qquad
\forall\,u\in\mathcal D_n^{[r]}
\]
by \eqref{eq:pmf}.
Since $\# B \le b_n^d$ for each $B \in \mathcal{B}_n$, we have
\[
\sum_{u \in B} p_n(u) \le b_n^d M_n^{-d} \le{(2r)^{-d}} M_n^{-\beta} \le 1/{\rdacc{T}_n}.
\]
Likewise, for each \(R\in\mathcal R_n\), we have
\(\#R\le q_n^d(1+rM_n/Q_n)^d\). Moreover, since \(\beta<\alpha\),  the definition of \(Q_n\) gives 
\({q_n}/{M_n}+{rq_n}/{Q_n}   
\le (50r)^{-1}M_n^{-\beta/d}
\)
for \(M\) sufficiently large.     Therefore,
\[
\sum_{u \in R} p_n(u)
\le q_n^d (1 + rM_n / Q_n)^d M_n^{-d}
\le {(50r)^{-d}} M_n^{-\beta}
\le 1/\rdacc{T}_n.
\]
The final comparisons with \(1/\rdacc{T}_n\) in the estimates for
\(B\) and \(R\) hold for \(M\) sufficiently large, since
\((2r)^{-d}<(r+1)^{-d}\).
Hence Lemma~\ref{lem:two-partition-sampling} applies with
\[
(\mathcal D,\mathcal B,\mathcal R,T,p)
=
(\mathcal D_n^{[r]},\mathcal B_n,\mathcal R_n,\rdacc{T}_n,p_n).
\]
Thus, for each \(n\ge1\) and each \(w\in\mathcal W_{n-1}^{[r]}\), we obtain an independent random set \(\rdacc{S}(w)\subset \mathcal D_n^{[r]}\), and hence a random uniform offspring system \(\rdacc{S}\) on \(\mathcal W^{[r]}\).
By construction, the family \(\{\rdacc{S}(w):w\in\mathcal W^{[r]}\}\) is independent, and both \eqref{barS} and \eqref{ET} hold.

Therefore, Corollary~\ref{cor:factor-nu} applies: every realization of
\({\rdacc{\mu}}\) is near \(\beta\)-AD regular, and
\({\rdacc{\mu}}\) is of \(\beta/2\)-Fourier decay almost surely.
We fix a realization of \({\rdacc{\mu}}\) for which this Fourier decay
holds.
For this fixed realization, we set
\[
\mu={\aracc{\mu}} * {\rdacc{\mu}}.
\]

Then, by Corollary~\ref{cor:factor-nu}, $\mu$ is $\alpha$-Frostman and moreover $\mu \in \mathcal{P}_{\alpha,\beta}(\mathbb{R}^d)$.
Since ${\aracc{\mu}}$ is near $(\alpha - \beta)$-AD regular and ${\rdacc{\mu}}$ is near $\beta$-AD regular, Lemma~\ref{lem:element} gives
\[
\overline{\dim}_{\mathcal{M}}(\operatorname{supp}{\aracc{\mu}})=\alpha-\beta,
\qquad
\overline{\dim}_{\mathcal{M}}(\operatorname{supp}{\rdacc{\mu}})=\beta.
\]
On the other hand, since \(\mu\) is \(\alpha\)-Frostman, we have $\dim_{\mathcal{H}}(\operatorname{supp}\mu)\ge \alpha$. Since both supports are compact,
$
\operatorname{supp}\mu
=\operatorname{supp}({\aracc{\mu}}*{\rdacc{\mu}})
=\operatorname{supp}{\aracc{\mu}}
+\operatorname{supp}{\rdacc{\mu}}.
$
Consequently,
\begin{align*}
\dim_{\mathcal{H}}(\operatorname{supp}\mu)
&\le
\overline{\dim}_{\mathcal{M}}(\operatorname{supp}{\aracc{\mu}}+\operatorname{supp}{\rdacc{\mu}})
\\
&\le
\overline{\dim}_{\mathcal{M}}(\operatorname{supp}{\aracc{\mu}})
+
\overline{\dim}_{\mathcal{M}}(\operatorname{supp}{\rdacc{\mu}})
=\alpha.
\end{align*}
Here the second inequality follows from the product-covering estimate
at a common scale.
Therefore, $\dim_{\mathcal{H}}(\operatorname{supp}\mu)=\alpha$.
This completes the proof.
\end{proof}

The use of the partition \(\mathcal B_n\) above is stronger than what is needed for the proof of Theorem~\ref{thm:conv-geo-i}.
Its role is to guarantee the \(b_n\)-block sparsity of \(\rdacc{S}(w)\), and hence the \(\beta\)-Frostman property of \({\rdacc{\mu}}\).
By contrast, the proof of Theorem~\ref{thm:conv-geo-i} only requires the arithmetic sparsity encoded by \(\mathcal R_n\), together with the resulting Fourier decay and the \(\alpha\)-Frostman property of \(\mu={\aracc{\mu}}*{\rdacc{\mu}}\).
Thus, for that theorem alone, one could sample only with respect to \(\mathcal R_n\), and the network-flow / total-unimodularity argument would be unnecessary.
The full two-partition sampling lemma is needed only when one also wishes to enforce the Frostman regularity of the factor \({\rdacc{\mu}}\), as in Proposition~\ref{prop:conv-geo-factorization-i}.

\medskip

We conclude this section by proving Lemma~\ref{lem:two-partition-sampling} using a standard network-flow argument based on total unimodularity (see, for example, \cite[Theorem~8.4]{Schrijver}).

\subsection{Proof of Lemma~\ref{lem:two-partition-sampling}}
\label{sec:twopart}

Set $y_u = T\,p(u)$, $u \in \mathcal{D}$.
Then \(y=(y_u)_{u\in\mathcal{D}}\) belongs to the polytope
\[
P
:=
\Big\{
x \in [0,1]^{\mathcal{D}}
:
\sum_{u\in\mathcal{D}} x_u = T; \
\sum_{u\in B} x_u \le 1,  \  \forall B\in\mathcal{B}; \
\sum_{u\in R} x_u \le 1,  \ \forall R\in\mathcal{R}
\Big\}.
\]
Hence \(P\neq\emptyset\), and \(y\) is a convex combination of extreme points of \(P\).

Build a directed network with source \(s\) and sink \(t\), with arcs
\(s\to B\) for \(B\in\mathcal B\) and arcs \(R\to t\) for
\(R\in\mathcal R\).
For each \(u\in\mathcal D\), let \(e_u\) be an arc from the unique
\(B\in\mathcal B\) containing \(u\) to the unique
\(R\in\mathcal R\) containing \(u\).
Parallel arcs are allowed, and every arc has capacity \(1\).
Any \(x\in P\) extends uniquely to an \(s\)-\(t\) flow of value \(T\)
by setting
\[
f(e_u)=x_u,
\qquad
f(s\to B)=\sum_{u\in B}x_u,
\qquad
f(R\to t)=\sum_{u\in R}x_u,
\]
and conversely any such flow restricts to an \(x\in P\).
Therefore \(P\) is affinely equivalent to the flow polytope of value \(T\) in this network.

The constraint matrix of this flow polytope is a directed node-arc incidence matrix, hence totally unimodular by \cite[Theorem~8.4]{Schrijver}.
Since all right-hand sides and bounds are integral, it follows from \cite[Theorem~8.1]{Schrijver} that every extreme point of \(P\) is integral, hence \(\{0,1\}\)-valued.
Thus each extreme point corresponds to a set \(S\subset\mathcal{D}\) with \(\#S=T\) meeting every \(B\in\mathcal{B}\) and every \(R\in\mathcal{R}\) in at most one point, i.e. \(S\in\Omega\).

Writing \(y\) as a convex combination of extreme points of \(P\), we obtain
\[
y
=
\sum_{S\in\Omega}\lambda_S\,\mathbbm{1}_S,
\qquad
\lambda_S \ge 0,
\qquad
\sum_{S\in\Omega}\lambda_S = 1.
\]
Defining \(\mathbb{P}(\{S\}) = \lambda_S\) yields a probability measure on \(\Omega\) satisfying \eqref{eq:two-partition-inclusion-probability}.
 \section{\texorpdfstring{Restriction: nongeometric case $\alpha <\beta$}{Restriction: nongeometric case alpha < beta}}
\label{sec6}

We begin with the following lemma, which is a standard consequence of the relation between Fourier energy and the Hausdorff dimension of a measure; see Wolff \cite[Corollary~8.7]{Wolff} or Chen \cite[Proposition~2]{Chen14}.
We therefore state it without proof.

\begin{lem}
\label{lem:res-dim-test}
Let \(\mu\) be a nonzero finite Borel measure on \(\mathbb{R}^d\), and suppose that
\(\widehat{\mu}\in L^q(\mathbb{R}^d)\) for some \(1\le q\le\infty\).
Then $\dim_{\mathcal{H}}(\operatorname{supp}\mu) \ge \min\{d, 2d/q \}$.
\end{lem}

\begin{proof}[Proof of Theorem \ref{thm:res-geo-i} when $\alpha<\beta$]
In order to prove Theorem \ref{thm:res-geo-i} for the nongeometric case, we need to show that there exists a measure \(\mu\in \mathcal{P}_{\alpha,\beta}(\mathbb{R}^d)\) with $ \dim_{\mathcal{H}}(\operatorname{supp} \mu)=\beta$ such that the estimate \eqref{eq:respq} can hold only if $q \ge 2d/\beta$ and \eqref{eq:reconpq} is satisfied.
However, thanks to Lemma~\ref{lem:res-dim-test}, it is sufficient to show existence of such a measure for which \eqref{eq:respq} holds only if \eqref{eq:reconpq} is satisfied. Indeed, applying \eqref{eq:respq} with \(f=1\), we get \(\widehat{\mu}\in L^q(\mathbb R^d)\); since \(\mu\) has compact support and \(\dim_{\mathcal H}(\operatorname{supp}\mu)=\beta\), Lemma~\ref{lem:res-dim-test} implies that \(q\ge 2d/\beta\).

We construct the desired measure $\mu$ by making use of Theorem~\ref{thm:near-AD-Salem}.
Let \(\sigma\) be a near \(\beta\)-AD regular probability measure with \(\beta/2\)-Fourier decay, provided by Theorem~\ref{thm:near-AD-Salem}.
By dilation we may assume \(\operatorname{supp}\sigma \subset [0,1]^d\).
For \(r>0\) and \(y\in\mathbb{R}^d\), write
\[
Q_{r,y} = y + [0,r]^d,
\]
and let \(\sigma_{r,y}\) be the pushforward measure of \(\sigma\) under the map \(x \mapsto y+rx\).
Then \(\sigma_{r,y}\) is a probability measure with \(\operatorname{supp}\sigma_{r,y}\subset Q_{r,y}\).

Let $\rho_n = \left\lceil 2^{n(2\alpha-\beta)/d} \right\rceil$.
Note that $\rho_n\ge 1$.
We set
\[
Z_n = (2^{n}{\rho_n})^{-1}\{0,1, \dots,\rho_n-1\}^d.
\]
Then \(\# Z_n = \rho_n^d \sim 2^{n(2\alpha-\beta)}\).
Let us consider a sequence $( t_n)_{n\ge 1}$ in $\mathbb R^d$ given by $t_1 = 0$ and $ t_n = (100d\sum_{k<n} 2^{-k})e_1$, $n \ge 2.$
Here $e_1=(1,0, \dots, 0)\in \mathbb R^d$.
Set
\begin{equation}      \label{nun}
\nu_n
=
\frac{1}{\# Z_n}\sum_{z\in Z_n} \sigma_{2^{-2n},\,t_n+ z}.
\end{equation}
We consider a probability measure
\[
\mu
= \frac{1}{W}\sum_{n=1}^\infty n^{-2}{2^{-\beta n}}\nu_n,
\]
where $W = \sum_{n=1}^\infty n^{-2}{2^{-\beta n}}$.
Clearly $\mu\in \mathcal P(\mathbb R^d)$.
We need to verify that \(\mu\) has the desired properties.

First, by the choice of the sequence $(t_n)_{n\ge 1}$, the supports of the \(\nu_n\) are pairwise disjoint, and each \(\operatorname{supp}\nu_n\) is a finite union of affine copies of \(\operatorname{supp}\sigma\).   Set $
t_\infty:=100d\sum_{k=1}^\infty 2^{-k}e_1=100d\,e_1.$
Since \(t_n\to t_\infty\) and
\(\operatorname{supp}\nu_n\subset B(t_n,C_d2^{-n})\), the only
accumulation point of \(\bigcup_n\operatorname{supp}\nu_n\) outside
that union is \(t_\infty\). Hence
\[
\operatorname{supp}\mu
=\overline{\bigcup_{n\ge1}\operatorname{supp}\nu_n}
=\{t_\infty\}\cup\bigcup_{n\ge1}\operatorname{supp}\nu_n.
\]
Each \(\operatorname{supp}\nu_n\) has Hausdorff dimension \(\beta\).
Therefore, by the countable stability of Hausdorff dimension,
\[
\dim_{\mathcal H}(\operatorname{supp}\mu) = \beta.
\]

We next verify that \(\mu \in \mathcal{P}_{\alpha,\beta}(\mathbb{R}^d)\).
Let $\mathcal F$ denote the Fourier transform.
Since $\mathcal F(\sigma_{2^{-2n},\,t_n+ z})(\xi)= e^{-2\pi i \xi\cdot (t_n+z)}\,\widehat{\sigma}(2^{-2n}\xi)$, \(|\widehat{\nu_n}(\xi)| \le |\widehat{\sigma}(2^{-2n}\xi)|\).
Recalling that $\sigma$ is of $\beta/2$-Fourier decay, we have
\[
|\widehat{\mu}(\xi)|
\lesssim
\sum_{n=1}^\infty \frac{2^{-\beta n}}{n^2}\,
|2^{-2n}\xi|^{-\beta/2}
\lesssim
|\xi|^{-\beta/2}.
\]

We now show that \(\mu\) satisfies \eqref{frostman}. Fix
\(x\in\mathbb{R}^d\) and \(\rho\in(0,1)\).
Indeed, for every \(n\ge1\), $
|t_{n+1}-t_n|=100d\,2^{-n}$ and $ \operatorname{supp}\nu_n \subset t_n+[0,2^{-n}+2^{-2n}]^d.$ 
Consequently, the ball \(B(x,\rho)\) meets
\(\operatorname{supp}\nu_n\) for at most one index \(n\) satisfying
\(2^{-n}\ge\rho\).
Thus
\begin{equation}
\label{eq:mu-ball-decomposition}
\mu(B(x,\rho))
\lesssim
\sum_{2^{-n}<\rho} n^{-2}{2^{-\beta n}}
+
\sup_{2^{-n}\ge\rho}n^{-2}{2^{-\beta n}}\nu_n(B(x,\rho)).
\end{equation}
The first sum is \(\lesssim \rho^\alpha\) since $\alpha\le \beta$.

For the second sum, let \(\delta_n = (2^{n}{\rho}_n)^{-1}\).
Recall that \(\nu_n\) is the average of the measures \(\sigma_{2^{-2n},\,t_n+ z}\), whose supports are contained in the cubes \(Q_{2^{-2n},\,t_n+z}\) with mutual spacing \(\delta_n\).
Since $2\alpha-\beta<d$, we have \(\delta_n \gg 2^{-2n}\). Thus, a ball of radius \(\rho\) meets at most one such cube if
\(\rho\le2^{-2n}\), \(O_d(1)\) such cubes if
\(2^{-2n}<\rho\le\delta_n\), and
\(O_d((\rho/\delta_n)^d)\) such cubes if
\(\delta_n<\rho\le2^{-n}\). 
Recalling \eqref{nun}, we see
\[
\nu_n(B(x,\rho))
\lesssim
\begin{cases}
    \rho_n^{-d}(\rho/2^{-2n})^\beta, & 0<\rho\le 2^{-2n},\\[1mm]
    \qquad  \rho_n^{-d}, & 2^{-2n}<\rho\le \delta_n,\\[1mm]
    \quad  (\rho/2^{-n})^d, & \delta_n<\rho\le 2^{-n}.
\end{cases}
\]
For the first case, the \(\beta\)-Frostman bound for \(\sigma\) gives
$
\sigma_{2^{-2n},\,t_n+z}(B(x,\rho)) 
\lesssim(\rho2^{2n})^\beta, 
$
since $\sigma_{2^{-2n},\,t_n+z}(B(x,\rho))
=\sigma(B((x-t_n-z)2^{2n},\rho2^{2n}))$. 
Using \(\rho_n^d\sim2^{n(2\alpha-\beta)}\),  a direct computation in the three cases above gives
\[
\sup_{2^{-n}\ge\rho}n^{-2}{2^{-\beta n}}\nu_n(B(x,\rho))\lesssim \rho^\alpha.
\]
Combining this with \eqref{eq:mu-ball-decomposition}, we conclude that
\(\mu\) is \(\alpha\)-Frostman.

It remains to show that \eqref{eq:respq} can hold only if
\eqref{eq:reconpq} holds.
To this end, we consider
\[
f_n = W \mathbbm{1}_{\operatorname{supp}\nu_n}.
\]
Since the supports of the \(\nu_n\) are disjoint,
\(f_n\,d\mu=n^{-2}2^{-\beta n}\,d\nu_n\). 
Thus, $\widehat{f_n\,d\mu}=n^{-2}{2^{-\beta n}}\,\widehat{\nu_n}$.
Also, we note that
\[
\widehat{\nu_n}(\xi)
=
e^{-2\pi i t_n\cdot \xi}\,\widehat{\sigma}(2^{-2n}\xi)\,S_n(\xi),
\]
where
\[
S_n(\eta)
=
\frac{1}{\# Z_n}\sum_{z\in Z_n} e^{-2\pi i z\cdot \eta}
=
\prod_{j=1}^d
\Big(
\frac{1}{\rho_n}\sum_{m=0}^{\rho_n-1} e^{-2\pi i m\eta_j/(2^n\rho_n)}
\Big).
\]
Equivalently,
\[
S_n(\eta)
=
\prod_{j=1}^d
\mathsf D_{\rho_n}\bigl(\eta_j/(2^n\rho_n)\bigr),
\]
where
\begin{equation}
\label{eq:d-kernel}
\mathsf D_N(t) = \frac{1}{N} \sum_{j=0}^{N - 1} e^{-2\pi i j t}, \quad N \in \mathbb{N}.
\end{equation}
By the Dirichlet kernel formula, we have an elementary lower bound
\begin{equation}
\label{eq:Dirichlet-lower}
\operatorname{dist}(t, \mathbb{Z}) < {1}/{2N}
\implies
|\mathsf D_N(t)| \ge 1/2.
\end{equation}
Hence, \(|S_n(\eta)|\gtrsim_d1\) whenever
\(\operatorname{dist}(2^{-n}\eta_j,\rho_n\mathbb{Z})\le\tfrac14\)
for all \(1\le j\le d\).
Since $\sigma$ is a compactly supported probability measure, one may choose \(c>0\) so that \(|\widehat{\sigma}(\eta)| \ge \tfrac12\) on \([-c,c]^d\).
We set
\[
\mathbf I_n^d
=
\left\{
\xi \in [-c 2^{2n}, c 2^{2n}]^d :
\operatorname{dist}(2^{-n}\xi_j, \rho_n\mathbb{Z}) \le \tfrac14,
\ \forall j
\right\}.
\]
By the preceding two lower bounds,
\(|\widehat{\nu_n}(\xi)|\gtrsim1\) on \(\mathbf I_n^d\).
Since $\widehat{f_n\,d\mu}= n^{-2}{2^{-\beta n}} \widehat{\nu_n}$, we have
\[
\|\widehat{f_n\,d\mu}\|_{L^q(\mathbb{R}^d)}
\gtrsim
n^{-2}{2^{-\beta n}}\,|\mathbf I_n^d|^{1/q}.
\]

Note that the set $\mathbf I_n := \{ t \in [-c 2^{2n}, c 2^{2n}] : \operatorname{dist}(2^{-n} t, \rho_n\mathbb{Z}) \le \tfrac14 \}$ is a union of intervals of length \(\sim 2^{n}\), separated by \(\rho_n2^{n}\).
Since \(2\alpha-\beta<d\), $
{2^{2n}}/{\rho_n2^n} 
\sim2^{n(1-(2\alpha-\beta)/d)}\rightarrow\infty
$ as $n\to\infty$. 
Thus, for all sufficiently large \(n\), the set \(\mathbf I_n\) contains
\(\gtrsim2^n/\rho_n\) intervals of length \(\sim2^n\). Consequently,
$|\mathbf I_n^d|
\gtrsim(2^{2n}/\rho_n)^d
\sim2^{n(2d-2\alpha+\beta)}.$ 
Consequently, we have
\[
\|\widehat{f_n\,d\mu}\|_{L^q(\mathbb{R}^d)}
\gtrsim
n^{-2}{2^{-\beta n}} 2^{n(\frac{2d-2\alpha+\beta}{q})}.
\]

On the other hand, since $\nu_n$ is a probability measure, we have $\|f_n\|_{L^p(\mu)} \lesssim (n^{-2}2^{-\beta n})^{1/p} \lesssim 2^{-\beta n/p}$.
Therefore, using \eqref{eq:respq}, we have, for some constant $C>0$,
\[
C       \gtrsim
n^{-2}
2^{n\left(\frac{2d-2\alpha+\beta}{q}-\frac{\beta}{p'}\right)}.
\]
By letting $n\to \infty$, this implies \eqref{eq:reconpq}.
\end{proof}
 \section{\texorpdfstring{Restriction: Geometric Case $\beta\le \alpha$}{Restriction: Geometric Case beta <= alpha}}
\label{rest-geo}

In this section, we prove Theorem \ref{thm:res-geo-i} in the geometric case, which is restated as the following proposition. We prove it by establishing Propositions \ref{prop:first-necessary-verify} and \ref{prop:second-necessary-verify} after constructing the measure $\mu$.

\begin{prop}
\label{prop:res-geo}
Let \(0 < \beta \le \alpha < d\).
Then there exists \(\mu \in \mathcal{P}_{\alpha,\beta}(\mathbb{R}^d)\) with \(\dim_{\mathcal{H}}(\operatorname{supp}\mu)=\alpha\) such that \eqref{eq:respq} holds only if $q \ge 2(d-\alpha+\beta)/\beta$ and \eqref{eq:reconpq} holds.
\end{prop}

In the proof of Proposition \ref{prop:first-necessary-verify}, we construct the desired measure \(\mu\) in the convolution form
\[
\mu={\aracc{\mu}} * {\rdacc{\mu}},
\]
as in Section~\ref{conv-geo}.
The factor \({\aracc{\mu}}\) will be deterministic and additively structured, and will be responsible for the geometric resonance needed in the sharpness argument.
The factor \({\rdacc{\mu}}\) will be obtained from a random construction so as to be a near \(\beta\)-AD regular Salem-type measure.
In order to produce the required resonance while preserving the random features of the construction, we further implant a thin arithmetic subtree into the tree defining \({\rdacc{\mu}}\).

\subsection{Parameters}

We begin by choosing parameters of suitable size for our construction.
The purpose of these choices will become clear later in this section.
Choose a large integer $R_0$ and set
\[
R_n = R_0 + n \quad (n \ge 1).
\]
Let $A = 10$ and $B = 100r(r + 1)$.
We set
\begin{equation}
\label{choice-p}
\aracc{M}_n= \ceil*{2R_n^{\frac{\alpha - \beta}{d}}}, 
\quad 
{\rdacc{M}_n}=\floor*{\frac{1}{2} R_n^{\frac{\beta}{2d}}}, 
\quad
L_n= \ceil*{B R_n^{\frac{\beta}{2d}}}, 
\quad 
{\rdacc{q}_n}= \floor*{\frac{1}{AB} R_n^{\frac{d-\alpha}{d}}}.
\end{equation}
Here and throughout this section, as in Section \ref{conv-geo}, we use the accents \(\aracc{\cdot}\) and \(\rdacc{\cdot}\) to indicate that the corresponding quantities or sets are associated with the measures \({\aracc{\mu}}\) and \({\rdacc{\mu}}\), respectively.

We also set
\begin{equation}
\label{choice-pp}
M_n = A \aracc{M}_n {\rdacc{M}_n} L_n {\rdacc{q}_n}, \quad T_n = \ceil*{(r + 1)^d M_n^\beta},  \quad
{\rdacc{Q}_n} = L_n {\rdacc{q}_n}, \quad \aracc{Q}_n = {\rdacc{M}_n} {\rdacc{Q}_n}.
\end{equation}
Then we have
\begin{equation}
\label{eq:Mn-asymp}
M_n \sim R_n \sim n.
\end{equation}
Note that the growth condition \eqref{eq:growth-condition} holds thanks to \eqref{eq:Mn-asymp}.
A simple computation shows the following relations between those parameters:
\begin{equation}
\label{eq:criteria-parameters}
  \qquad  \ \   \aracc{M}_n^d \ge  M_n^{\alpha - \beta}, \qquad T_n \ge r^d M_n^{\beta}, \qquad \aracc{M}_n^d T_n \ge (r + 1)^d M_n^\alpha,
\end{equation}
\begin{equation}
\label{eq:thin-subtree}
\rdacc{M}_n^d \le M_n^{\beta/2},
\end{equation}
\begin{equation}
\label{eq:light-blocks}
(\rdacc{M}_n L_n)^d \ge (r+1)^d T_n, 
\end{equation}
\begin{equation}
\label{eq:branches-asymp}
\aracc{M}_n^d \sim M_n^{\alpha - \beta}, \qquad T_n \sim M_n^\beta, \qquad \rdacc{M}_n^d \sim M_n^{\beta/2},
\end{equation}
\begin{equation}
\label{eq:gap-asymp}
{\rdacc{Q}_n} \sim M_n^{1 - (\alpha - \beta/2)/d}, \qquad \aracc{Q}_n \sim M_n^{1 - (\alpha - \beta)/d}.
\end{equation}

\subsection{Arithmetic structure}
\label{sec:as}

In order to make the Fourier transform of the measure large at certain frequencies, one needs to create a resonance structure.
A simple approach to this is to impose an arithmetic structure.
In this subsection, we consider the arithmetic structure that will later be used to prove sharpness of the restriction estimate.

We first consider two grid sets $\aracc{\mathcal D}_n$ and $\rdacc{\mathcal D}_n$ given by \eqref{eq:grid} and
\begin{equation}
\label{eq:grid-}
\rdacc{\mathcal D}_n = {\rdacc{Q}_n} \{0, 1, \ldots, {\rdacc{M}_n} - 1\}^d,
\end{equation}
respectively.
Let ${\aracc{\mu}}$, ${\aracc{\mu}}_n$ be given as in Section \ref{conv-geo}; see \eqref{eq:grid-offspring}, \eqref{tildetrees}, and \eqref{eq:lambda}.
That is to say, ${\aracc{\mu}}$ is the measure induced by the uniform offspring system $\aracc{S}$ on $\mathcal{W}^{[1]}$ and ${\aracc{\mu}}_n$ is the $n$th truncated induced measure.
In particular, we have \eqref{tildemu_n}.

Let
\[
\aracc{\sigma}_n = \bconv_{k=n + 1}^\infty \bigg( {\aracc{M}_k^{-d}} \sum_{u \in \aracc{\mathcal D}_k} \delta_{u/{\mathfrak M}_k} \bigg).
\]
The infinite convolution converges weakly, so \(\aracc{\sigma}_n\) is well defined. Recalling \eqref{eq:coding-basic-bounds}, we see $\operatorname{supp}\aracc{\sigma}_n \subset [0, 1/{\mathfrak M}_n]^d$, and it is clear from \eqref{tildemu_n} and \eqref{eq:truncated-coding-map} (see also \eqref{eq:truncated-induced-measure}) that
\begin{equation}
\label{eq:res-lambda}
{\aracc{\mu}} = {\aracc{\mu}}_n * \aracc{\sigma}_n.
\end{equation}

We also define
\begin{equation}
\label{eq:res-eta-n}
\rdacc{\eta}_n = \bconv_{k=1}^n \bigg( \frac{1}{T_k} \sum_{u \in {\rdacc{\mathcal D}}_k} \delta_{u/{\mathfrak M}_k} \bigg).
\end{equation}
Note that \(\rdacc{\eta}_n\) is not the \(n\)-th truncated induced measure of \({\rdacc{\mu}}\).
Rather, it is an auxiliary sub-probability measure encoding only the arithmetic pattern \(\rdacc{\mathcal D}_1,\dots,\rdacc{\mathcal D}_n\).
By contrast, \({\aracc{\mu}}_n\) is the genuine \(n\)-th truncated induced measure of \({\aracc{\mu}}\).

\subsection{Resonance estimates}

In this subsection, we obtain bounds concerning the Fourier transforms of ${\aracc{\mu}}_n * \rdacc{\eta}_n$ and ${\aracc{\mu}}_n$.

By \eqref{choice-pp}, $M_k = A \aracc{Q}_k \aracc{M}_k$, and hence $\aracc{Q}_k / {\mathfrak M}_k = 1 / (A \aracc{M}_k {\mathfrak M}_{k - 1})$.
Thus, recalling \eqref{eq:grid} and \eqref{tildemu_n}, we have
\begin{equation}
\label{eq:res-lambda-hat}
\widehat{{\aracc{\mu}}_n}(\xi) = \prod_{k=1}^n \prod_{\ell=1}^{d} \mathsf D_{\aracc{M}_k}\left( \frac{\xi_\ell}{A \aracc{M}_k {\mathfrak M}_{k - 1}} \right),
\end{equation}
where $ \mathsf D_N$ denotes the normalized Dirichlet kernel given by \eqref{eq:d-kernel}.
Similarly, since ${\rdacc{Q}}_k / {\mathfrak M}_k = 1 / (A {\rdacc{M}}_k \aracc{M}_k {\mathfrak M}_{k - 1})$, from \eqref{eq:res-eta-n} and \eqref{eq:grid-} we have
\begin{equation}
\label{eq:res-eta-hat}
\widehat{\rdacc{\eta}_n}(\xi) = \prod_{k=1}^n \frac{\rdacc{M}_k^d}{T_k} \prod_{\ell=1}^d \mathsf D_{{\rdacc{M}}_k} \bigg( \frac{\xi_\ell}{A {\rdacc{M}}_k \aracc{M}_k {\mathfrak M}_{k - 1}} \bigg).
\end{equation}
Using the identity $\mathsf D_{N_1} (N_2 t) \mathsf D_{N_2} (t) = \mathsf D_{N_1N_2}(t)$ for $N_1, N_2 \in \mathbb{N}$, and combining \eqref{eq:res-lambda-hat} and \eqref{eq:res-eta-hat}, we obtain
\begin{equation}
\label{eq:res-lambda-eta}
\widehat{{\aracc{\mu}}_n * \rdacc{\eta}_n}(\xi) = \prod_{k=1}^n \frac{\rdacc{M}_k^d}{T_k} \prod_{\ell=1}^d \mathsf D_{{\rdacc{M}}_k \aracc{M}_k} \left( \frac{\xi_\ell}{A {\rdacc{M}}_k \aracc{M}_k {\mathfrak M}_{k - 1}} \right).
\end{equation}

The next lemma will be used to prove Proposition \ref{prop:first-necessary-verify}.

\begin{lem}
\label{lem:first-necessary}
For \(k\ge1\), let \(L_k:=A\rdacc{M}_k\aracc{M}_k\mathfrak M_{k-1}\).
For \(n\ge1\), set
\[
\mathbf I_n
=
\Big\{
t \in \mathbb{R} : \operatorname{dist}(t/L_k, \mathbb{Z}) < (2{\rdacc{M}}_k\aracc{M}_k)^{-1}, \ \forall 1 \le k \le n
\Big\}.
\]
Then the following hold:
\begin{align}
\label{mmn}
|\widehat{{\aracc{\mu}}_n * \rdacc{\eta}_n}(\xi)|
&\ge 2^{-dn} \prod_{k=1}^n (\rdacc{M}_k^d / T_k), \quad \forall \xi \in \mathbf I_n^d,
\\
\label{omegan}
|\mathbf I_n^d \cap [-c\theta^n {\mathfrak M}_n, c\theta^n {\mathfrak M}_n]^d|
&\gtrsim_{\varepsilon, c, \theta} {\mathfrak M}_n^{\,d-\alpha+\beta/2-\varepsilon}, \quad \forall c, \theta \in (0, 1).
\end{align}
\end{lem}

\begin{proof}
If \(\xi \in \mathbf I_n^d\), it is clear for \(1 \le \ell \le d\) that $\operatorname{dist}\!( {\xi_\ell}/L_k, \mathbb{Z}) <1/(2{\rdacc{M}}_k\aracc{M}_k)$ for \(1 \le k \le n\). Hence, \eqref{eq:res-lambda-eta} and \eqref{eq:Dirichlet-lower} give \eqref{mmn}.

Now, fixing \(c,\theta \in (0,1)\), we proceed to prove \eqref{omegan}.  Since \(\theta^n\mathfrak M_n/\mathfrak M_1\to\infty\), there exists
\(n_0=n_0(c,\theta)\ge2\) such that $
4\mathfrak M_1\le c\theta^n\mathfrak M_n,$ $n\ge n_0.$ 
For the finitely many \(n<n_0\), \eqref{omegan} follows after decreasing
the implicit constant. Thus, we may  assume \(n\ge n_0\), and let
\(m\in\{1,\ldots,n-1\}\) be the largest integer such that 
\begin{equation}    \label{4M}
4{\mathfrak M}_m \le c\theta^n {\mathfrak M}_n.
\end{equation}
Define
\[
E_m
=
\Big\{
t+ \sum_{k = 1}^m L_k a_k : t\in  [-{A}/{8}, {A}/{8}], \quad a_k \in \{0, 1, \ldots, {\rdacc{Q}}_k - 1\}
\Big\}.
\]

We first claim that
\begin{equation}    \label{em}
E_m \subset \mathbf I_n \cap [-c\theta^n {\mathfrak M}_n, c\theta^n {\mathfrak M}_n].
\end{equation}
Let $x\in E_m$, i.e.,
\[
x = \sum_{k = 1}^m L_k a_k + t
\]
for some $ a_k \in \{0, 1, \ldots, {\rdacc{Q}}_k - 1\}$ and $ |t| \le {A}/{8}.$
Since \(|a_k| \le {\rdacc{Q}}_k\) and \(\sum_{k = 1}^m {\mathfrak M}_k \le 2{\mathfrak M}_m\), recalling \eqref{choice-pp}, we see $\sum_{k = 1}^m L_k {\rdacc{Q}}_k = \sum_{k = 1}^m {\mathfrak M}_k \le 2{\mathfrak M}_m$.
Thus, by \eqref{4M} it follows that
\[
|x|
\le 4{\mathfrak M}_m
\le c \theta^n {\mathfrak M}_n.
\]

Thus, it suffices for \eqref{em} to show $x\in \mathbf I_n$.
Note \(L_j/L_k \in \mathbb{N}\) for every \(j \ge k\) provided that \(k \le m\).
If \(k \le m\),
\[
\operatorname{dist}\!\left(\frac{x}{L_k}, \mathbb{Z}\right)
\le \sum_{j < k} |a_j| \frac{L_j}{L_k} + \frac{|t|}{L_k} \le \frac{\sum_{j < k} {\mathfrak M}_j}{A {\rdacc{M}}_k \aracc{M}_k {\mathfrak M}_{k - 1}} + \frac{1}{8{\rdacc{M}}_k \aracc{M}_k}
<
\frac{1}{2{\rdacc{M}}_k\aracc{M}_k}.
\]
If \(k > m\), then \({\mathfrak M}_m \le {\mathfrak M}_{k - 1}\), so
\[
\operatorname{dist}\!\left(\frac{x}{L_k}, \mathbb{Z}\right)
\le \frac{|x|}{L_k}
\le \frac{4{\mathfrak M}_m}{A {\rdacc{M}}_k\aracc{M}_k {\mathfrak M}_{k - 1}}
\le \frac{4}{A {\rdacc{M}}_k\aracc{M}_k}
< \frac{1}{2{\rdacc{M}}_k\aracc{M}_k}.
\]
Consequently \(x \in \mathbf I_n\), and the claim follows.

Distinct points in the set
\[
E_m'=\Big\{ \sum_{k = 1}^m L_k a_k:
a_k \in \{0, 1, \ldots, {\rdacc{Q}}_k - 1\}\Big\}
\]
are separated by more than \(A/4\).
Indeed, let $\gamma=\sum_{k = 1}^m L_k a_k$ and $\gamma'=\sum_{k = 1}^m L_k a_k'$, and let \(\ell\) be the largest index at which $a_\ell\neq a_\ell'$.
Then
\[
|\gamma-\gamma'|
\ge
L_\ell - \sum_{j < \ell} \rdacc{Q}_j L_j
=
L_\ell - \sum_{j < \ell} {\mathfrak M}_j
\ge
(A-2){\mathfrak M}_{\ell - 1}
>
{A}/{4}.
\]
Thus, the intervals comprising \(E_m\) are pairwise disjoint.
So, $|E_m| = \frac{A}{4} (\#E'_m) = \frac{A}{4} \prod_{k = 1}^m {\rdacc{Q}}_k$.
Therefore, combining this with \eqref{em} gives $|\mathbf I_n \cap [-c\theta^n {\mathfrak M}_n, c\theta^n {\mathfrak M}_n]| \gtrsim \prod_{k = 1}^m {\rdacc{Q}}_k$, and hence
\begin{equation}
\label{eq:Bohr-set}
|\mathbf I_n^d \cap [-c\theta^n {\mathfrak M}_n, c\theta^n {\mathfrak M}_n]^d|
\gtrsim \prod_{k = 1}^m \rdacc{Q}_k^d, \quad \forall n\ge m.
\end{equation}

By \eqref{eq:gap-asymp} and Lemma~\ref{lem:absorption},
$ 
\prod_{k=1}^n\rdacc Q_k^d
\gtrsim_\varepsilon
\mathfrak M_n^{\,d-\alpha+\beta/2-\varepsilon/2}. 
$
Thus, it remains to show that, for any \(\varepsilon>0\), 
\begin{equation}
\label{prodQ}
\prod_{k = m + 1}^n \rdacc{Q}_k^d
\lesssim_{\varepsilon,c,\theta}
{\mathfrak M}_n^{\varepsilon/2}, \quad  \forall n>m. 
\end{equation}
Indeed, this implies $\prod_{k = 1}^m \rdacc{Q}_k^d  \gtrsim_{\varepsilon,c,\theta} {\mathfrak M}_n^{\,d-\alpha+\beta/2-\varepsilon}$, so \eqref{omegan} follows from \eqref{eq:Bohr-set}.

By maximality of our choice of \(m\) (see \eqref{4M}), $ 4{\mathfrak M}_{m+1} > c\theta^n {\mathfrak M}_n$.
Thus,
\[
\prod_{j = m + 2}^n M_j
=   {{\mathfrak M}_n}/{{\mathfrak M}_{m+1}}
<
{4}\theta^{-n}/c.
\]
Since $M_n \sim n$ and \(\rdacc{Q}_k^d \sim M_k^{\,d-\alpha+\beta/2}\), the above inequality yields
\[
\prod_{k = m + 1}^n \rdacc{Q}_k^d
\lesssim
M_{m+1}^{\,d-\alpha+\beta/2}
\Big(\prod_{j = m + 2}^n M_j\Big)^{d-\alpha+\beta/2}
\lesssim_{c,\theta}
n^{\,d-\alpha+\beta/2}\theta^{-n(d-\alpha+\beta/2)}.
\]
Thus, superexponential growth of ${\mathfrak M}_n$ gives \eqref{prodQ} for any $\varepsilon>0$.
\end{proof}

In order to show Proposition \ref{prop:first-necessary-verify}, we will use the following lemma, in which $L^2$ and $L^4$ norms of the Fourier transform of ${\aracc{\mu}}_n$ are estimated on a localized scale.

Let $\phi \in C_c^\infty(\mathbb{R}^d)$ such that
\[
0 \le \phi \le 1,
\quad
\phi^\vee \ge 0,
\quad
\phi > 0 \text{ on } [-1,1]^d,
\quad
\operatorname{supp}\phi \subset [-2,2]^d.
\]

\begin{lem}
\label{lem:second-necessary}
Let $n\ge 1$ and ${\aracc{\mu}}_n$ be given by \eqref{tildemu_n}.
Let $\phi_n(\xi) = \phi({\xi}/{c{\mathfrak M}_n})$ for some small $c > 0$.
Set
\[
W_n(\xi) = \phi_n(\xi)\,|\widehat{{\aracc{\mu}}_n}(\xi)|^2,
\qquad
V_n(\xi) =\phi_n(\xi)\,|\widehat{{\aracc{\mu}}_n}(\xi)|^4.
\]
Then, for every \(n\ge1\), the following hold provided that \(c>0\) is sufficiently small:
\begin{gather}
\label{Wn}
(W_n)^\vee \ge 0,
\quad (V_n)^\vee \ge 0,  
\quad  0 \le V_n \le  W_n \le 1,
\quad |\widehat{{\aracc{\mu}}}|^2 W_n \gtrsim V_n;
\\
\label{Nn1}
\int_{\mathbb{R}^d} W_n(\xi)\,d\xi
\lesssim_{\varepsilon} 
{\mathfrak M}_n^{\,d-\alpha+\beta+\varepsilon};
\\
\label{Nn}
\int_{\mathbb{R}^d} V_n(\xi)\,d\xi
\gtrsim_{\varepsilon} 
{\mathfrak M}_n^{\,d-\alpha+\beta - \varepsilon}.
\end{gather}
\end{lem}

\begin{proof}
We first verify \eqref{Wn}. Since $\phi^\vee \ge 0$, writing \(\widetilde{{\aracc{\mu}}_n}(E) = {\aracc{\mu}}_n(-E)\), we have
\[
(W_n)^\vee
=
(\phi_n)^\vee * ({\aracc{\mu}}_n * \widetilde{{\aracc{\mu}}_n})
\ge 0,
\qquad
(V_n)^\vee
=
(\phi_n)^\vee * ({\aracc{\mu}}_n * \widetilde{{\aracc{\mu}}_n}) * ({\aracc{\mu}}_n * \widetilde{{\aracc{\mu}}_n})
\ge 0.
\]
Since ${\aracc{\mu}}_n$ is a probability measure, we have \(0 \le |\widehat{{\aracc{\mu}}_n}| \le 1\), and hence $0 \le V_n \le W_n \le 1$.

After taking $c$ small enough, we may assume
\[
\operatorname{supp}\phi_n \subset \Big[-\frac{{\mathfrak M}_n}{100d}, \frac{{\mathfrak M}_n}{100d}\Big]^d.
\]
For \(\xi \in \operatorname{supp}\phi_n\) and \(x \in \operatorname{supp}\aracc{\sigma}_n \subset [0,1/{\mathfrak M}_n]^d\), $|x \cdot \xi| \le {1}/{100}$.
Hence \(\Re \widehat{\aracc{\sigma}_n}(\xi) \ge \cos(2\pi/100)\), so \(|\widehat{\aracc{\sigma}_n}(\xi)|^2 \ge c'\) on \(\operatorname{supp}\phi_n\) for all $n\ge 1$ and some constant \(c' > 0\). 
From \eqref{eq:res-lambda}, we have \(\widehat{{\aracc{\mu}}} = \widehat{{\aracc{\mu}}_n}\widehat{\aracc{\sigma}_n}\).
Thus, for \(\xi \in \mathbb{R}^d\),
\[
|\widehat{{\aracc{\mu}}}(\xi)|^2 W_n(\xi)
=
|\widehat{\aracc{\sigma}_n}(\xi)|^2 \phi_n(\xi) |\widehat{{\aracc{\mu}}_n}(\xi)|^4
\gtrsim
\phi_n(\xi) |\widehat{{\aracc{\mu}}_n}(\xi)|^4
=
V_n(\xi).
\]

It remains to show \eqref{Nn1} and \eqref{Nn}.
We first verify the inequality \eqref{Nn1}.
We begin by noting that \(\phi_n \ge c_0\) on \([-c{\mathfrak M}_n, c{\mathfrak M}_n]^d\) for some constant $c_0>0$, and \(\operatorname{supp}\phi_n \subset [-2c{\mathfrak M}_n,2c{\mathfrak M}_n]^d\).
Recalling \eqref{eq:truncated-induced-measure} and \eqref{tildemu_n}, we have
\[
{\aracc{\mu}}_n = \frac{1}{\#E_n} \sum_{x \in E_n} \delta_x,
\]
where $E_n=\supp {\aracc{\mu}}_n$ (see \eqref{supp-tildemu}).
Since \(E_n \subset {\mathfrak M}_n^{-1}\mathbb{Z}^d\), \(|\widehat{{\aracc{\mu}}_n}|^2\) is \({\mathfrak M}_n\mathbb{Z}^d\)-periodic.
Since $\operatorname{supp}\phi_n \subset [-2c{\mathfrak M}_n, 2c{\mathfrak M}_n]^d$,
\[
\int_{\mathbb{R}^d} W_n(\xi)\,d\xi
\le \int_{[-{\mathfrak M}_n, {\mathfrak M}_n]^d} |\widehat{{\aracc{\mu}}_n}(\xi)|^2\,d\xi
\lesssim \int_{[0,{\mathfrak M}_n]^d} |\widehat{{\aracc{\mu}}_n}(\xi)|^2\,d\xi.
\]
By Parseval's identity, we have
\[
\int_{[0,{\mathfrak M}_n]^d} |\widehat{{\aracc{\mu}}_n}(\xi)|^2\,d\xi
=
{\mathfrak M}_n^d \sum_{x\in E_n} {\aracc{\mu}}_n(\{x\})^2
=
\frac{{\mathfrak M}_n^d}{\#E_n}
=
\frac{{\mathfrak M}_n^d}{\prod_{k = 1}^n \aracc{M}_k^d}.
\]
By \eqref{eq:branches-asymp} and Lemma \ref{lem:absorption}, we have $\prod_{k = 1}^n \aracc{M}_k^d \gtrsim_{\varepsilon} {\mathfrak M}_n^{\alpha - \beta - \varepsilon}$. Combining these estimates yields \eqref{Nn1}.

We now verify the inequality \eqref{Nn}. Recalling \eqref{choice-pp}, we note $M_n = A \aracc{M}_n \aracc{Q}_n$.
Since $\aracc{\mathcal D}_n \subset \aracc{Q}_n \mathbb{Z}^d$, it follows from \eqref{supp-tildemu} that
\[
\supp {\aracc{\mu}}_n=E_n \subset  (A \aracc{M}_n {\mathfrak M}_{n - 1})^{-1}\mathbb{Z}^d.
\]
Thus, the function \(|\widehat{{\aracc{\mu}}_n}|^4\) is \(A \aracc{M}_n {\mathfrak M}_{n - 1}\mathbb{Z}^d\)-periodic, and hence
\[
\begin{aligned}
\int_{[0,c{\mathfrak M}_n]^d}|\widehat{{\aracc{\mu}}_n}(\xi)|^4\,d\xi
&\gtrsim \left(\frac{cM_n}{A\aracc{M}_n}\right)^d
\int_{[0,A\aracc{M}_n{\mathfrak M}_{n-1}]^d}
|\widehat{{\aracc{\mu}}_n}(\xi)|^4\,d\xi\\
&\ge c^d\int_{[0,{\mathfrak M}_n]^d}
|\widehat{{\aracc{\mu}}_n}(\xi)|^4\,d\xi .
\end{aligned}
\]
Since \(\phi_n \ge c_0\) on \([-c {\mathfrak M}_n, c {\mathfrak M}_n]^d\), $\int_{\mathbb{R}^d} V_n(\xi)\,d\xi \gtrsim \int_{[0, c{\mathfrak M}_n]^d} |\widehat{{\aracc{\mu}}_n}(\xi)|^4\,d\xi$.
Consequently, combining this with Parseval's identity as before, we have
\[
\int_{\mathbb{R}^d} V_n(\xi)\,d\xi
\gtrsim \int_{[0,{\mathfrak M}_n]^d} |\widehat{{\aracc{\mu}}_n}(\xi)|^4\,d\xi =
{\mathfrak M}_n^d \sum_u ({\aracc{\mu}}_n * \widetilde{{\aracc{\mu}}_n})(\{u\})^2.
\]

The measure \({\aracc{\mu}}_n * \widetilde{{\aracc{\mu}}_n}\) is a probability measure supported on \(E_n - E_n\).
As in the proof of Lemma~\ref{lem:small-difference-criterion} (see \eqref{enen}), we have $\#(E_n - E_n) \le \prod_{k = 1}^n (2\aracc{M}_k)^d \lesssim_{\varepsilon} {\mathfrak M}_n^{\alpha - \beta + \varepsilon}$ by \eqref{eq:branches-asymp} and Lemma \ref{lem:absorption}.
Now, the Cauchy-Schwarz inequality gives
\[
\sum_u ({\aracc{\mu}}_n * \widetilde{{\aracc{\mu}}_n})(\{u\})^2
\ge
\frac{1}{\#(E_n - E_n)}\gtrsim  {\mathfrak M}_n^{-\alpha +\beta - \varepsilon}.
\]
Therefore, the inequality \eqref{Nn} follows.
This completes the proof.
\end{proof}

\subsection{Random construction and basic properties}

We now turn to the random construction and build into it the arithmetic structure introduced in Section \ref{sec:as}.
For this purpose, we make use of Lemma~\ref{lem:two-partition-sampling}.
The idea is to force the desired arithmetic pattern along a thin deterministic subtree, while keeping the construction random elsewhere.
This strategy goes back to Hambrook--{\L}aba \cite{HL13}.

Recall the sets \(\rdacc{\mathcal D}_n\) given by \eqref{eq:grid-}.
We define a distinguished subtree \(\rdacc{\mathcal T}^\ar\subset \mathcal W^{[r]}\) by setting
\begin{equation}
\label{eq:arithmetic-subtree}
\rdacc{\mathcal T}^\ar_0=\{\varnothing\}, \qquad
\rdacc{\mathcal T}^\ar_n=\prod_{k=1}^n \rdacc{\mathcal D}_k, \qquad
\rdacc{\mathcal T}^\ar=\bigcup_{n=0}^\infty \rdacc{\mathcal T}^\ar_n.
\end{equation}
Thus, the branching of \(\rdacc{\mathcal T}^\ar\) at level \(n\) is exactly the arithmetic set \(\rdacc{\mathcal D}_n\).
Since \(\#\rdacc{\mathcal D}_n=\rdacc{M}_n^d\) and \(\rdacc{M}_n^d\le M_n^{\beta/2}\) by \eqref{eq:thin-subtree}, this subtree is much thinner than the ambient branching \(T_n\sim M_n^\beta\).

For each \(n\), we consider two partitions of \(\mathcal D_n^{[r]}\).
Let
\[
\mathcal B_n=\bigl\{\{u\}:u\in\mathcal D_n^{[r]}\bigr\}
\]
be the trivial partition into singletons, and let \(\mathcal R_n\) be the partition of \(\mathcal D_n^{[r]}\) into \(\rdacc{q}_n\)-blocks modulo \(\aracc{Q}_n\).
More precisely, \(\mathcal R_n\) is the partition of \(\mathcal D_n^{[r]}\) consisting of the nonempty sets
\[
\mathcal D_n^{[r]}\cap
\bigl(
b+\{0,1,\ldots,\rdacc{q}_n-1\}^d+\aracc{Q}_n\mathbb Z^d
\bigr),
\qquad
b\in \rdacc{q}_n\mathbb Z^d\cap [0,\aracc{Q}_n)^d.
\]

We now verify the hypotheses of Lemma~\ref{lem:two-partition-sampling} with
\[
(\mathcal D,\mathcal B,\mathcal R,T,p)
=
(\mathcal D_n^{[r]},\mathcal B_n,\mathcal R_n,T_n,p_n).
\]
The condition for \(\mathcal B_n\) is immediate: for every singleton block \(B=\{u\}\in\mathcal B_n\),
\[
\sum_{v\in B}p_n(v)=p_n(u)\le M_n^{-d}\le T_n^{-1}
\]
for all sufficiently large initial parameter \(R_0\), since \(T_n\sim M_n^\beta\) and \(\beta<d\).
Next, let \(R\in\mathcal R_n\).
Then, $\#R\le \rdacc{q}_n^d\bigl(1+rM_n/\aracc{Q}_n\bigr)^d$, and therefore we have
\[
\begin{aligned}
\sum_{u\in R}p_n(u)
&\le (\#R)M_n^{-d} \le 
\Big(
\frac{r+(A\aracc M_n)^{-1}}
{\rdacc M_nL_n}
\Big)^d \le
\Big(
\frac{r+1}
{\rdacc M_nL_n}
\Big)^d 
\le \frac1{T_n},
\end{aligned}
\]
where the last inequality follows from \eqref{eq:light-blocks}. 
Hence, Lemma~\ref{lem:two-partition-sampling} yields a probability measure on the family of subsets of \(\mathcal D_n^{[r]}\) of cardinality \(T_n\) which are \(\rdacc{q}_n\)-block sparse modulo \(\aracc{Q}_n\) and have one-point marginals \(T_np_n(u)\).

\subsubsection{Random set}
Let \(S_n\subset\mathcal D_n^{[r]}\) be a random set with this law.
Then,
\[
\#S_n=T_n,
\qquad
S_n\text{ is }\rdacc{q}_n\text{-block sparse modulo }\aracc{Q}_n,
\]
and
\[
\mathbb E\biggl[\frac{1}{T_n}\mathbbm 1_{S_n}(u)\biggr]=p_n(u),
\quad \forall u\in\mathcal D_n^{[r]}.
\]

Next, for each \(n\), fix a set \(\rdacc{H}_n \subset\mathcal D_n^{[r]}\) such that
\[
\rdacc{\mathcal D}_n\subset \rdacc{H}_n ,
\qquad
\#\rdacc{H}_n =T_n, \qquad \text{and \(\rdacc{H}_n \) is \(\rdacc{q}_n\)-block sparse modulo \(\aracc{Q}_n\).}
\]
This is possible because \(\rdacc{\mathcal D}_n\) itself is \(\rdacc{q}_n\)-block sparse modulo \(\aracc{Q}_n\).
Indeed, to see this, recall \eqref{eq:grid-} and \eqref{choice-pp}; since $\aracc{Q}_n=\rdacc{M}_n \rdacc{Q}_n$, we have \(\rdacc{\mathcal D}_n\subset [0,\aracc{Q}_n)^d\), and \(\rdacc{\mathcal D}_n\) meets exactly \(\rdacc{M}_n^d\) distinct \(\rdacc{q}_n\)-blocks modulo \(\aracc{Q}_n\).
On the other hand, the total number of \(\rdacc{q}_n\)-blocks modulo \(\aracc{Q}_n\) is
\[
(\aracc{Q}_n/\rdacc{q}_n)^d=(\rdacc{M}_nL_n)^d\ge T_n
\]
by \eqref{eq:light-blocks}.
Hence there remain at least \((T_n-\rdacc{M}_n^d)\) unused \(\rdacc{q}_n\)-blocks modulo \(\aracc{Q}_n\). Every block counted above meets \(\mathcal D_n^{[r]}\). Indeed, its
standard representative satisfies
\[
b\in\rdacc q_n\mathbb Z^d\cap[0,\aracc Q_n)^d
\subset\{0,1,\ldots,M_n-1\}^d
\subset\mathcal D_n^{[r]},
\]
where we used $
M_n=A\aracc M_n\aracc Q_n>\aracc Q_n. $
Choosing the standard representative from each of any
\((T_n-\rdacc M_n^d)\) unused blocks and adjoining these points to
\(\rdacc{\mathcal D}_n\), we obtain a set
\(\rdacc H_n\subset\mathcal D_n^{[r]}\) satisfying the required
properties. 

\subsubsection{Random uniform offspring system}
\label{Ruos}
We now define a random uniform offspring system \(\rdacc{S}\) on \(\mathcal W^{[r]}\) with profile \((T_n)_{n\ge1}\).
For each \(w\in\mathcal W_{n-1}^{[r]}\), let
\[
\rdacc{S}(w)=
\begin{cases}
    \quad \qquad \qquad   \rdacc{H}_n , & w\in\rdacc{\mathcal T}^\ar_{n-1},
    \\[2pt]
    \text{an independent copy of }S_n, & w\notin\rdacc{\mathcal T}^\ar_{n-1}.
\end{cases}
\]
In this way the construction is random away from \(\rdacc{\mathcal T}^\ar\), but along \(\rdacc{\mathcal T}^\ar\) it contains the arithmetic pattern \(\rdacc{\mathcal D}_n\) at every generation.
The family $\{\sigma(\rdacc{S}(w)):w\in\mathcal W^{[r]}\}$ is independent, since the sets attached to \(w\notin\rdacc{\mathcal T}^\ar\) are chosen independently and the sets attached to \(w\in\rdacc{\mathcal T}^\ar\) are deterministic.

\subsubsection{Random measure}
\label{Rm}
Let \({\rdacc{\mu}}\) be the induced measure of the random uniform offspring system \(\rdacc{S}\), and set
\[
\mu={\aracc{\mu}}*{\rdacc{\mu}}.
\]
By Proposition~\ref{prop:conv-offspring-system}, the measure \(\mu\) is the induced measure of the convolution offspring system \(\aracc{S}*\rdacc{S}\) defined in Section \ref{sec:convolution}.
For every admissible node \(w\in\mathcal T_{n-1}\), writing \(w=\aracc{w}+\rdacc{w}\) with \(\aracc{w}\in\aracc{\mathcal T}_{n-1}\) and \(\rdacc{w}\in\rdacc{\mathcal T}_{n-1}\), we have
\[
(\aracc{S}*\rdacc{S})(w)=\aracc{S}(\aracc{w})+\rdacc{S}(\rdacc{w}).
\]

Since \(\aracc{S}(\aracc{w})\subset \aracc{Q}_n\mathbb Z^d\) and \(\rdacc{S}(\rdacc{w})\) is \(\rdacc{q}_n\)-block sparse modulo \(\aracc{Q}_n\), Lemma~\ref{lem:residue-sep-add} shows that \((\aracc{S}*\rdacc{S})(w)\) is \(\rdacc{q}_n\)-block sparse.
The profile of \(\aracc{S}*\rdacc{S}\) is \((\aracc{M}_n^dT_n)_{n\ge1}\), and $\aracc{M}_n^dT_n\ge (r+1)^dM_n^\alpha$ (see \eqref{eq:criteria-parameters}).
Since \(\rdacc{q}_n\sim M_n^{1-\alpha/d}\), Proposition~\ref{prop:frostman-block} implies that \(\mu\) is \(\alpha\)-Frostman.

\subsubsection{Almost sure Fourier decay}
It remains to prove that \({\rdacc{\mu}}\) is of \(\beta/2\)-Fourier decay almost surely.
We follow the proof of Proposition~\ref{prop:fourier-criterion} with \(\alpha\) replaced by \(\beta\).

We begin by noting from our construction above that \(\rdacc{S}\) satisfies the following properties:
\begin{itemize}[itemsep=0.5em]
\item \(T_n\ge (r+1)^dM_n^\beta\) for every \(n\ge1\);
\item the family \(\{\sigma(\rdacc{S}(w)):w\in\mathcal W^{[r]}\}\) is independent;
\item for every \(n\ge1\) and \(w\in\mathcal W_{n-1}^{[r]}\setminus\rdacc{\mathcal T}^\ar_{n-1}\),
\begin{equation}    \label{El}
\mathbb E\Bigl[\frac{1}{T_n}\mathbbm 1_{\rdacc{S}(w)}(u)\Big]=p_n(u),
\quad \forall u\in\mathcal D_n^{[r]}.
\end{equation}
\end{itemize}

In view of the argument in Section \ref{conv-geo}, the only issue is that the correct-marginal identity fails on the deterministic subtree \(\rdacc{\mathcal T}^\ar\), and this contribution has to be estimated separately.

\begin{prop}
\label{prpp:f-decay}
Let \({\rdacc{\mu}}\) be the induced measure of the random offspring system \(\rdacc{S}\).
The random measure \({\rdacc{\mu}}\) is of \(\beta/2\)-Fourier decay almost surely.
\end{prop}

\begin{proof} 
Let $ {\rdacc{\mu}}_n$ be the $n$-th truncated induced measure of \({\rdacc{\mu}}\) with $ {\rdacc{\mu}}_0 := \delta_0$.
Define
\[
\mathfrak D_n(\xi)
:=
\widehat{{\rdacc{\mu}}_n}(\xi)         -        m_n(\xi/\mathfrak M_n)\widehat{{\rdacc{\mu}}_{n-1}}(\xi),
\]
where $m_n$ is defined by \eqref{eq:char} with $\mathcal D_n$ replaced by $\mathcal D_n^{[r]}$.

To show \({\rdacc{\mu}}\) is of \(\beta/2\)-Fourier decay almost surely, by the same summation argument as in the proof of Proposition~\ref{prop:fourier-criterion}, it suffices to prove the analogue of \eqref{eq:fourier-reduction} with \(\beta\) in place of \(\alpha\), namely
\begin{equation}
\label{eq:fourier-reduction-bar}
\lim_{A\to\infty}
\mathbb P\Big(
\mathfrak M_n^{\beta/2}\sup_{\xi\in\mathbb R^d}|{\mathfrak D_n}(\xi)|
\le A,
\quad \forall n\ge1
\Big)
=1.
\end{equation}

For \(v\in\rdacc{\mathcal T}_{n-1}\), let
\[
Y_v(\xi)
:=
\frac{1}{T_n}\sum_{u\in \rdacc{S}(v)} e^{-2\pi i\,\xi\cdot u/\mathfrak M_n}
-
m_n(\xi/\mathfrak M_n).
\]
Then, in the same manner as in the proof of Proposition~\ref{prop:fourier-criterion}, we have
\[
{\mathfrak D_n}(\xi)
=
\frac{1}{\prod_{k=1}^{n-1}T_k}
\sum_{v\in\rdacc{\mathcal T}_{n-1}}
e^{-2\pi i\,\xi\cdot X(v)}\,Y_v(\xi).
\]

We split
\[
{\mathfrak D_n}(\xi)=\mathfrak{D}_n^{\mathrm{good}}(\xi)+\mathfrak{D}_n^{\mathrm{bad}}(\xi),
\]
where
\begin{align*}
\mathfrak{D}_n^{\mathrm{good}}(\xi)
&:=
\frac{1}{\prod_{k=1}^{n-1}T_k}
\sum_{v\in\rdacc{\mathcal T}_{n-1}\setminus \rdacc{\mathcal T}^\ar_{n-1}}
e^{-2\pi i\,\xi\cdot X(v)}\,Y_v(\xi), \\
\mathfrak{D}_n^{\mathrm{bad}}(\xi)
&:=
\frac{1}{\prod_{k=1}^{n-1}T_k}
\sum_{v\in\rdacc{\mathcal T}^\ar_{n-1}}
e^{-2\pi i\,\xi\cdot X(v)}\,Y_v(\xi).
\end{align*}

For \(\mathfrak{D}_n^{\mathrm{good}}(\xi)\), we prove
\begin{equation}
\label{eq:Dn-good}
\lim_{A\to\infty}
\mathbb P\Big(
\mathfrak M_n^{\beta/2}\sup_{\xi\in\mathbb R^d}|\mathfrak{D}_n^{\mathrm{good}}(\xi)|
\le A,
\quad \forall n\ge1
\Big)
=1.
\end{equation}
To this end, we apply Lemma~\ref{lem:cond-hoeffding}.
Let \(\mathcal F_0\) be the trivial \(\sigma\)-algebra, and for \(k\ge1\), let
\[
\mathcal F_k
:=
\sigma\bigl(\rdacc{S}(w): w\in\mathcal W^{[r]},\ |w|<k\bigr).
\]
The index set
\(\rdacc{\mathcal T}_{n-1}\setminus\rdacc{\mathcal T}^\ar_{n-1}\)
is \(\mathcal F_{n-1}\)-measurable.
By the construction in Section \ref{Ruos}, the sets \(\rdacc S(v)\)
indexed by it are independent copies of \(S_n\) and are independent of
\(\mathcal F_{n-1}\). Hence the corresponding variables \(Y_v(\xi)\)
are conditionally independent given \(\mathcal F_{n-1}\). 
Moreover, the coefficients
\[
a_v(\xi)
:=
\Big(\prod_{k=1}^{n-1}T_k\Big)^{-1}
e^{-2\pi i\,\xi\cdot X(v)}
\]
are \(\mathcal F_{n-1}\)-measurable.
We next verify the conditional mean-zero property.
Recall that
\eqref{El} holds for \(w\in\mathcal W_{n-1}^{[r]}\setminus\rdacc{\mathcal T}^\ar_{n-1}\).
Since $ \mathbb E\bigl[Y_v(\xi)\mid \mathcal F_{n-1}\bigr]= \mathbb E\bigl[Y_v(\xi)\bigr]$, for $v\in \rdacc{\mathcal T}_{n-1}\setminus \rdacc{\mathcal T}^\ar_{n-1}$, we have
\begin{align*}
\mathbb E\bigl[Y_v(\xi)\mid \mathcal F_{n-1}\bigr]
&= 
\sum_{u\in\mathcal D_n^{[r]}}
\mathbb E\Bigl[\frac{1}{T_n}\mathbbm 1_{\rdacc{S}(v)}(u)\Big]
e^{-2\pi i\,\xi\cdot u/\mathfrak M_n}
-
m_n(\xi/\mathfrak M_n) 
\\
&= 
\sum_{u\in\mathcal D_n^{[r]}}
p_n(u)e^{-2\pi i\,\xi\cdot u/\mathfrak M_n}
-
m_n(\xi/\mathfrak M_n)
=0.
\end{align*}
Also, \(|Y_v(\xi)|\le 2\).
Since \(\#\rdacc{\mathcal T}_{n-1}=\prod_{k=1}^{n-1}T_k\), we have $\sum_{v\in\rdacc{\mathcal T}_{n-1}\setminus \rdacc{\mathcal T}^\ar_{n-1}} |a_v(\xi)|^2 \le {\#\rdacc{\mathcal T}_{n-1}}{(\prod_{k=1}^{n-1}T_k)^{-2}} = (\prod_{k=1}^{n-1}T_k)^{-1}$.
Hence Lemma~\ref{lem:cond-hoeffding} yields
\[
\mathbb P\bigl(|\mathfrak{D}_n^{\mathrm{good}}(\xi)|\ge t \,\big|\, \mathcal F_{n-1}\bigr)
\le
4\exp\biggl(
-\frac{1}{16}\Bigl(\prod_{k=1}^{n-1}T_k\Bigr)t^2
\biggr).
\]

After taking expectations, we choose \(t = B {\mathfrak M}_n^{-\beta/2}\).
Then, using \(T_k\ge (r+1)^dM_k^\beta\), we obtain
\[
\mathbb P\bigl(\mathfrak M_n^{\beta/2}|\mathfrak{D}_n^{\mathrm{good}}(\xi)|\ge B\bigr)
\le
4\exp\bigl(-cB^2 (r+1)^{d(n-1)} M_n^{-\beta}\bigr)
\]
for some constant \(c>0\).  With this established, the rest of the argument is almost identical to that in the proof of Proposition~\ref{prop:fourier-criterion}. Indeed, note that each function
\[
\xi\mapsto e^{-2\pi i\,\xi\cdot X(v)}Y_v(\xi)
\]
is \(\mathfrak M_n\)-periodic in each coordinate and \(O_{r,d}(1)\)-Lipschitz, and hence so is \(\mathfrak{D}_n^{\mathrm{good}}\).
Let \(\Lambda_n\) be a \(c_{r,d}\mathfrak M_n^{-\beta/2}\)-net of \([0,\mathfrak M_n]^d\), where \(c_{r,d}>0\) is sufficiently small.
Then $\#\Lambda_n \lesssim \mathfrak M_n^{d(1+\beta/2)}$, and by the same net argument as in Proposition~\ref{prop:fourier-criterion},
\[
\mathbb P\Big(
\mathfrak M_n^{\beta/2}\sup_{\xi\in\mathbb R^d}|\mathfrak{D}_n^{\mathrm{good}}(\xi)|
> B+1
\Big)
\lesssim
\mathfrak M_n^{d(1+\beta/2)}
\exp\bigl(-cB^2 (r+1)^{d(n-1)} M_n^{-\beta}\bigr).
\]
By the growth condition \eqref{eq:growth-condition}, the right-hand side is summable in \(n\) for each fixed \(B\), and its sum tends to \(0\) as \(B\to\infty\).
Therefore, we obtain \eqref{eq:Dn-good}.

For \(\mathfrak{D}_n^{\mathrm{bad}}(\xi)\), we obtain a deterministic bound.
Since \(|Y_v(\xi)|\le2\), we have $|\mathfrak{D}_n^{\mathrm{bad}}(\xi)| \le {2}(\prod_{k=1}^{n-1}T_k)^{-1}\, \#\rdacc{\mathcal T}^\ar_{n-1}$.
Also, since $\#\rdacc{\mathcal T}^\ar_{n-1} = \prod_{k=1}^{n-1} \#\rdacc{\mathcal D}_k= \prod_{k=1}^{n-1}\rdacc{M}_k^d$, we obtain
\[
|\mathfrak{D}_n^{\mathrm{bad}}(\xi)|
\le
2\prod_{k=1}^{n-1}\frac{\rdacc{M}_k^d}{T_k}
\le
2(r+1)^{-d(n-1)}\mathfrak M_{n-1}^{-\beta/2}
\lesssim
\mathfrak M_n^{-\beta/2},
\]
where we use \eqref{eq:thin-subtree}, \(T_k\ge (r+1)^dM_k^\beta\), and Lemma~\ref{lem:absorption}.

Combining \eqref{eq:Dn-good} with the above bound on \(\mathfrak{D}_n^{\mathrm{bad}}\), we obtain \eqref{eq:fourier-reduction-bar}.
\end{proof}

\subsubsection{\texorpdfstring{Choosing a realization of ${\rdacc{\mu}}$}{Choosing a realization of mu rd}}
Using Proposition \ref{prpp:f-decay}, we now fix a realization of the random measure ${\rdacc{\mu}}$ so that \({\rdacc{\mu}}\) is of \(\beta/2\)-Fourier decay.
Consequently, we have the following.

\begin{prop}
\label{dim-mu}
Let $\mu={\aracc{\mu}}\ast{\rdacc{\mu}}$ be given by the construction above, with ${\rdacc{\mu}}$ a realization of the random measure ${\rdacc{\mu}}$ of \(\beta/2\)-Fourier decay.
Then, \(\mu\in\mathcal P_{\alpha,\beta}(\mathbb R^d)\) and \(\dim_{\mathcal H}(\operatorname{supp}\mu)=\alpha\).
\end{prop}
\begin{proof}
Note that the uniform offspring system \(\aracc{S}*\rdacc{S}\) inducing \(\mu\) has profile \((\aracc{M}_n^dT_n)_{n\ge1}\).
By \eqref{eq:branches-asymp}, we have \(\aracc{M}_n^dT_n\sim M_n^\alpha\).
Thus, Proposition~\ref{prop:seq-lower} together with the \(\alpha\)-Frostman property of \(\mu\) (observed in Section \ref{Rm}) implies that \(\mu\) is near \(\alpha\)-AD regular.
Therefore, by \eqref{eq:measure-dim-2} in Lemma~\ref{lem:element} it follows that \(\dim_{\mathcal H}(\operatorname{supp}\mu)=\alpha\).
Note that $|\widehat {\mu}|\le |\widehat{{\rdacc{\mu}}}|$.
Since ${\rdacc{\mu}}$ is of $\beta/2$-Fourier decay, so is $\mu$.
Therefore, \(\mu\in\mathcal P_{\alpha,\beta}(\mathbb R^d)\).
\end{proof}

At this stage, the construction of \(\mu\) and the verification of its basic properties are complete. It remains only to establish the necessary conditions, which will be done in the next subsection. 

\subsection{Verification of the necessary conditions}

As before, let \(\rdacc{\mathcal T}\) denote the tree of the offspring system \(\rdacc{S}\) on \(\mathcal W^{[r]}\) that is constructed above.
Let \(\rdacc{\tau}\) denote the uniform product measure on \(\rdacc{\mathcal T}_\infty\) as in \eqref{eq:uniform-product-measure}.
Recall that \({\rdacc{\mu}}=\pi_\sharp\rdacc{\tau}\).

\begin{prop}
\label{prop:first-necessary-verify}
Let $\mu={\aracc{\mu}}\ast{\rdacc{\mu}}$ be given by the construction.
Then, for every realization of ${\rdacc{\mu}}$, the estimate \eqref{eq:respq} holds only if \eqref{eq:reconpq} holds.
\end{prop}

\begin{proof}
Fix an arbitrary realization of the random offspring system \(\rdacc{S}\), and let \(\rdacc{\mathcal T}\), \(\rdacc{\tau}\), and \({\rdacc{\mu}}\) denote the corresponding tree, product measure, and induced measure, respectively.
Let us set
\[
\mathcal E_n=\{\mathbf w\in\rdacc{\mathcal T}_\infty:\mathbf w|_n\in\rdacc{\mathcal T}^\ar_n\}.
\]
Then, since $\#\rdacc{\mathcal T}^\ar_{n} = \prod_{k=1}^{n}\rdacc{M}_k^d$, we have
\begin{equation}
\label{btE}
\rdacc{\tau}(\mathcal E_n)=  \sum_{w\in\rdacc{\mathcal T}^\ar_{n}}   \rdacc{\tau}([w])  = \prod_{k=1}^n(\rdacc{M}_k^d/T_k).
\end{equation}
We also note from \eqref{eq:res-eta-n} and \eqref{eq:arithmetic-subtree} that
\begin{equation}
\label{pin}
(\pi_n)_\sharp(\rdacc{\tau}|_{\mathcal E_n})
=
\frac{1}{\prod_{k=1}^nT_k}\sum_{w\in\rdacc{\mathcal T}^\ar_n}\delta_{X(w)}=\rdacc{\eta}_n.
\end{equation}
Since \(\pi_\sharp(\rdacc{\tau}|_{\mathcal E_n})\le {\rdacc{\mu}}\), the measure $ {\aracc{\mu}}*\pi_\sharp(\rdacc{\tau}|_{\mathcal E_n})$ is absolutely continuous with respect to $\mu={\aracc{\mu}}\ast {\rdacc{\mu}}$, and hence there exists a function \(f_n\) such that
\begin{equation}
\label{tmus}
f_n\,d\mu={\aracc{\mu}}*\pi_\sharp(\rdacc{\tau}|_{\mathcal E_n}).
\end{equation}
Since $0\le f_n\le 1$, it follows that $\|f_n\|_{L^p(\mu)}\le \rdacc{\tau}(\mathcal E_n)^{1/p}$.
By \eqref{btE}, \eqref{eq:branches-asymp}, and Lemma~\ref{lem:absorption}, we have, for any $\varepsilon>0$,
\begin{equation}
\label{fn}
\|f_n\|_{L^p(\mu)}\lesssim_\varepsilon {\mathfrak M}_n^{-\beta/(2p)+\varepsilon}.
\end{equation}
For each \(w\in\rdacc{\mathcal T}^\ar_n\), we may write
\begin{equation}
\label{pi-sharp}
\pi_\sharp(\rdacc{\tau}|_{[w]})=\frac{1}{\prod_{k=1}^nT_k}\,\delta_{X(w)}*\rho_w
\end{equation}
for some probability measure \(\rho_w\) with \(\operatorname{supp}\rho_w\subset [0,r/{\mathfrak M}_n]^d\).
Indeed, for \(\mathbf v\in[w]\), write \(\pi(\mathbf v)=X(w)+R_w(\mathbf v)\), where \(R_w(\mathbf v)\in [0,r/{\mathfrak M}_n]^d\) by \eqref{eq:coding-basic-bounds}.
Defining
\[
\rho_w=(R_w)_\sharp\!\left(\frac{\rdacc{\tau}|_{[w]}}{\rdacc{\tau}([w])}\right),
\]
we have $\pi_\sharp(\rdacc{\tau}|_{[w]})=\rdacc{\tau}([w])\,\delta_{X(w)}*\rho_w,$ which yields \eqref{pi-sharp} since $\rdacc{\tau}([w])=(\prod_{k=1}^nT_k)^{-1}$.

Combining \eqref{tmus}, \eqref{pin}, and \eqref{pi-sharp}, we have
\[
\widehat{f_n\,d\mu}(\xi)
=\widehat{{\aracc{\mu}}_n}(\xi) \widehat{\aracc{\sigma}_n}(\xi)
\sum_{w\in{\rdacc{\mathcal T}}^\ar_n}
\frac{e^{-2\pi i\,\xi\cdot X(w)}}{\prod_{k=1}^nT_k}
\widehat{\rho_w}(\xi).
\]
Thus, we may write
\[
\widehat{f_n\,d\mu}(\xi)-\widehat{{\aracc{\mu}}_n*\rdacc{\eta}_n}(\xi)=  \widehat{{\aracc{\mu}}_n}(\xi)\Big(  \sum_{w\in{\rdacc{\mathcal T}}^\ar_n}
\frac{e^{-2\pi i\,\xi\cdot X(w)}}{\prod_{k=1}^nT_k}
( \widehat{\aracc{\sigma}_n}(\xi)\widehat{\rho_w}(\xi)- 1)
\Big).
\]
Note that \(\operatorname{supp}\aracc{\sigma}_n\subset[0,1/{\mathfrak M}_n]^d\).
Using the first equality in \eqref{btE}, we have
\[
\bigl|\widehat{f_n\,d\mu}(\xi)-\widehat{{\aracc{\mu}}_n*\rdacc{\eta}_n}(\xi)\bigr|
\le 
\rdacc{\tau}(\mathcal E_n)\Bigl(|\widehat{\aracc{\sigma}_n}(\xi)-1|+\sup\nolimits_{w\in\rdacc{\mathcal T}^\ar_n}|\widehat{\rho_w}(\xi)-1|\Bigr)
\lesssim 
\rdacc{\tau}(\mathcal E_n){|\xi|}/{{\mathfrak M}_n}.
\]
For the second inequality, we use the fact that $\aracc{\sigma}_n$ and $\rho_w$ are probability measures with supports contained in a fixed compact set.

Now, we fix \(\theta=2^{-d-2}\).
Then, using the above inequality, we have
\[
\bigl|\widehat{f_n\,d\mu}(\xi)-\widehat{{\aracc{\mu}}_n*\rdacc{\eta}_n}(\xi)\bigr|
\lesssim c\,2^{-dn}\rdacc{\tau}(\mathcal E_n)
\]
for \(\xi\in[-c\theta^n{\mathfrak M}_n,c\theta^n{\mathfrak M}_n]^d\).
Choosing \(c>0\) sufficiently small, by \eqref{mmn} we obtain
\[
|\widehat{f_n\,d\mu}(\xi)|\ge 2^{-dn-1}\rdacc{\tau}(\mathcal E_n)
\]
for $\xi\in\mathbf I_n^d\cap[-c\theta^n{\mathfrak M}_n,c\theta^n{\mathfrak M}_n]^d$.
By \eqref{omegan}, \eqref{btE}, and \eqref{eq:branches-asymp} combined with Lemma~\ref{lem:absorption} as before, it follows that
\[
\|\widehat{f_n\,d\mu}\|_{L^q(\mathbb R^d)}
\gtrsim_\varepsilon
\rdacc{\tau}(\mathcal E_n){\mathfrak M}_n^{(d-\alpha+\beta/2)/q-\varepsilon}
\gtrsim_\varepsilon
{\mathfrak M}_n^{-\beta/2+(d-\alpha+\beta/2)/q-\varepsilon}
\]
for any $\varepsilon>0$.
Consequently, this and \eqref{fn} yield
\[
\frac{\|\widehat{f_n\,d\mu}\|_{L^q(\mathbb R^d)}}{\|f_n\|_{L^p(\mu)}}
\gtrsim_\varepsilon {\mathfrak M}_n^{(d-\alpha+\beta/2)/q-\beta/(2p')-\varepsilon}.
\]
If \(q<\frac{2d-2\alpha+\beta}{\beta}p'\), choose \(\varepsilon>0\) so small that the exponent is positive.
Then the right-hand side tends to \(+\infty\), contradicting the restriction estimate \eqref{eq:respq}.
Therefore, the estimate \eqref{eq:respq} implies \eqref{eq:reconpq}.
\end{proof}

We now finish the proof of Proposition \ref{prop:res-geo} by proving the following proposition.

\begin{prop}
\label{prop:second-necessary-verify}
Let $\mu={\aracc{\mu}}\ast {\rdacc{\mu}}$ be given by the construction above.
Then, for every realization of ${\rdacc{\mu}}$, the estimate \eqref{eq:respq} holds only if $q\ge 2(d-\alpha+\beta)/{\beta}$.
\end{prop}

\begin{proof}
Fixing an arbitrary realization of the random \({\rdacc{\mu}}\), we consider $\mu={\aracc{\mu}}\ast {\rdacc{\mu}}$.

If \(q=\infty\), there is nothing to prove, so we assume \(q<\infty\).
Taking \(f\equiv1\), the restriction estimate \eqref{eq:respq} implies
\begin{equation}    \label{widehatmu}
\|\widehat\mu\|_{L^q(\mathbb R^d)}\lesssim1.
\end{equation}
Since \(\widehat\mu\in L^\infty(\mathbb R^d)\) as well, the case \(q\le2\) would imply \(\widehat\mu\in L^2(\mathbb R^d)\).
By Plancherel's theorem, \(\mu\) would then have an \(L^2\) density, so \(\operatorname{supp}\mu\) would have positive Lebesgue measure, contradicting \(\dim_{\mathcal H}(\operatorname{supp}\mu)=\alpha<d\) (see Section \ref{Rm}).
Thus, we may assume \(q>2\).

We prove \eqref{widehatmu} holds only if $q\ge 2(d-\alpha+\beta)/{\beta}$.
Let \(W_n,V_n\) be the functions given in Lemma~\ref{lem:second-necessary}. From \eqref{Wn},  note that 
\(V_n\ge0\) and \((V_n)^\vee\) is real-valued. 
Since \(\operatorname{supp}V_n\subset[-2c{\mathfrak M}_n,2c{\mathfrak M}_n]^d\), taking $c$ sufficiently small,  we have 
\begin{equation}\label{Vn}
(V_n)^\vee(x)
\gtrsim (V_n)^\vee(0),
\quad \forall x\in B(0,r\sqrt d\,{\mathfrak M}_n^{-1}).
\end{equation}
Recall that the measure ${\rdacc{\mu}}$ is induced by the system \(\rdacc{S}\) with profile $(T_n)_{n\ge 1}$.
Since \(T_n\sim M_n^\beta\) (see \eqref{eq:branches-asymp}), applying Proposition~\ref{prop:seq-lower} to \(\rdacc{S}\) gives
\[
{\rdacc{\mu}}(B(x,r\sqrt d\,{\mathfrak M}_n^{-1}))\gtrsim_\varepsilon {\mathfrak M}_n^{-\beta-\varepsilon}
\]
for $x\in\operatorname{supp}{\rdacc{\mu}}$.
Thus, it follows that
\[
({\rdacc{\mu}}*\widetilde{{\rdacc{\mu}}})(B(0,r\sqrt d\,{\mathfrak M}_n^{-1}))
=\int {\rdacc{\mu}}(B(x,r\sqrt d\,{\mathfrak M}_n^{-1}))\,d{\rdacc{\mu}}(x)
\gtrsim_\varepsilon {\mathfrak M}_n^{-\beta-\varepsilon},
\]
where \(\widetilde{{\rdacc{\mu}}}(E)={\rdacc{\mu}}(-E)\).
Since $\int_{\mathbb R^d}|\widehat{{\rdacc{\mu}}}(\xi)|^2V_n(\xi)\,d\xi =\int_{\mathbb R^d}(V_n)^\vee(x)\,d({\rdacc{\mu}}*\widetilde{{\rdacc{\mu}}})(x)$, using \eqref{Vn}, we obtain
\[
\int_{\mathbb R^d}|\widehat{{\rdacc{\mu}}}(\xi)|^2V_n(\xi)\,d\xi
\gtrsim \int_{{B(0,r\sqrt d\,{\mathfrak M}_n^{-1})}}(V_n)^\vee(x)\,d({\rdacc{\mu}}*\widetilde{{\rdacc{\mu}}})(x)
\gtrsim_\varepsilon {\mathfrak M}_n^{-\beta-\varepsilon}(V_n)^\vee(0).
\]
Since $\mu={\aracc{\mu}}\ast{\rdacc{\mu}}$, using the last inequality in \eqref{Wn} and \eqref{Nn}, we obtain
\begin{equation}
\label{eq:second-lower}
\int_{\mathbb R^d}|\widehat\mu(\xi)|^2W_n(\xi)\,d\xi
\gtrsim\int  |\widehat{{\rdacc{\mu}}}|^2V_n(\xi)\,d\xi
\gtrsim_\varepsilon {\mathfrak M}_n^{d-\alpha-\varepsilon}.
\end{equation}

On the other hand, since \(0\le W_n\le1\) and \(q>2\), applying Hölder's inequality, the assumption \eqref{widehatmu}, and \eqref{Nn1} successively, we obtain
\[
\int_{\mathbb R^d}|\widehat\mu(\xi)|^2W_n(\xi)\,d\xi
\le \|\widehat\mu\|_{L^q(\mathbb R^d)}^2
\| W_n\|_{L^\frac{q}{q-2}(\mathbb R^d)}  \lesssim_\varepsilon {\mathfrak M}_n^{(d-\alpha+\beta+\varepsilon)\frac{q-2}{q} }.
\]
Combining this and \eqref{eq:second-lower}, we obtain
\[
{\mathfrak M}_n^{d-\alpha-\varepsilon}\lesssim_\varepsilon {\mathfrak M}_n^{(d-\alpha+\beta+\varepsilon)(1-2/q)}
\]
for every $\varepsilon>0$.
If \(q<\frac{2d-2\alpha+2\beta}{\beta}\), then for sufficiently small \(\varepsilon>0\) the exponent on the left is strictly larger than the exponent on the right.
This leads to a contradiction for large \(n\).
Therefore, we have $q\ge 2(d-\alpha+\beta)/{\beta}$.
\end{proof}

This completes the proof of Proposition \ref{prop:res-geo}, and hence of Theorem \ref{thm:res-geo-i}.
 \section{Further Applications}
\label{sec8}

In this section, we provide the proofs of Corollaries \ref{f-spec} and \ref{cor:m-fold-singular}.

\subsection{Fourier spectrum}

As explained in the discussion after \cite[Proposition~4.2]{CFdO24}, the sharpness of \eqref{eq:fourier-spectrum} follows if one can show that, for any given $t \in [0, d]$ and an unbounded sequence $(z_n)$, there exists a compactly supported finite Borel measure $\mu$ which has Salem dimension $t$ and $\limsup_{n\to\infty} |\widehat{\mu}(z_n)|^2 |z_n|^s = \infty$ for all $s > t$.

As a consequence of Theorem~\ref{thm:near-AD-Salem}, we prove the following, which is somewhat stronger than what we need to prove the sharpness of \eqref{eq:fourier-spectrum}.

\begin{prop}
\label{prop:fourier-lower-bound}
Let \(t \in [0,d]\), and let $\psi : [1, \infty) \to (0, 1]$ be a decreasing function such that
\[
\int_1^\infty \psi(R) \frac{dR}{R} < \infty.
\]
Then there exists \(\mu \in \mathcal{P}(\mathbb{R}^d)\) such that \(\dim_{\mathcal{H}}(\operatorname{supp}\mu) = t\) and
\[
(1 + |\xi|)^{-t/2} \psi(1 + |\xi|) \lesssim \widehat{\mu}(\xi) \lesssim (1 + |\xi|)^{-t/2}
\qquad
\forall \xi \in \mathbb{R}^d.
\]
\end{prop}

Postponing the proof of Proposition \ref{prop:fourier-lower-bound} for the moment, we first show Corollary \ref{f-spec}.

\begin{proof}[Proof of Corollary \ref{f-spec}]
Let $\mu_1$ be the measure given by \cite[Lemma~4.3]{CFdO24} so that $\dim_{\mathcal{F}} \mu_1 = 0$ and $\dim_{\mathcal{S}} \mu_1 = d$.
Here $\dim_{\mathcal{S}}$ denotes the Sobolev dimension.
Then $\dim_{\mathcal{F}}^\theta \mu_1 \le d\theta$ by \cite[Proposition~4.2]{CFdO24} and $\dim_{\mathcal{F}}^\theta \mu_1 \ge d\theta$ by the concavity of the map $\theta \mapsto \dim_{\mathcal{F}}^\theta \mu_1$.
Hence $\dim_{\mathcal{F}}^\theta \mu_1 = d\theta$.

Since $\dim_{\mathcal{F}}\mu_1=0$, there exists a sequence $(z_n)_{n=1}^\infty$ with $|z_n|\to \infty$ such that $|\widehat{\mu_1}(z_n)|\gtrsim_\varepsilon |z_n|^{-\varepsilon}$ for every $\varepsilon>0$. 
Applying Proposition~\ref{prop:fourier-lower-bound} with $\psi(R)=(1+\log R)^{-2}$, which satisfies the required integrability condition and $\psi(R)\gtrsim_\varepsilon R^{-\varepsilon}$, we can take $\mu_2\in\mathcal P(\mathbb R^d)$ with $\dim_{\mathcal H}(\operatorname{supp}\mu_2)=t$ and
\[
|\xi|^{-t/2-\varepsilon}\lesssim_\varepsilon|\widehat{\mu_2}(\xi)|
\lesssim|\xi|^{-t/2}
\]
for every $\varepsilon>0$ and $|\xi|\ge1$.  Let $\mu=\mu_1*\mu_2$.
Then
\[
|\widehat{\mu}(\xi)|\lesssim|\xi|^{-t/2},
\qquad
|\widehat{\mu}(z_n)|\gtrsim_\varepsilon|z_n|^{-t/2-2\varepsilon}
\]
for $|\xi|\ge1$ and all sufficiently large $n$.
Thus $\dim_{\mathcal F}\mu=t$.
By \cite[Proposition~4.2]{CFdO24}, $\dim_{\mathcal F}^\theta\mu\le t+d\theta$. Using \cite[Theorem~6.1]{Fraser}, we see that
\[
\dim_{\mathcal{F}}^\theta \mu \ge \dim_{\mathcal{F}}^\theta \mu_1 + \dim_{\mathcal{F}} \mu_2 = d\theta + t.
\]
Therefore, we obtain $\dim_{\mathcal{F}}^\theta \mu = t + d\theta$.
\end{proof}

\begin{proof}[Proof of Proposition \ref{prop:fourier-lower-bound}]
By Theorem~\ref{thm:near-AD-Salem}, there exists a near \(t/2\)-AD regular measure \(\sigma\) of \(t/4\)-Fourier decay.
Set
\[
\nu = \sigma * \widetilde{\sigma},
\]
where $\widetilde{\sigma}(A) = \sigma(-A)$.
Then, \(\nu \in \mathcal{P}(\mathbb{R}^d)\), $\widehat{\nu}(\xi) = |\widehat{\sigma}(\xi)|^2 \ge 0$, and
\begin{equation}
\label{f-decay}
|\widehat{\nu}(\xi)| \lesssim (1 + |\xi|)^{-t/2}.
\end{equation}
Since $\sigma$ is near \(t/2\)-AD regular, by \eqref{eq:measure-dim-2} we also have $\overline{\dim}_{\mathcal{M}}(\operatorname{supp}\sigma) =t/2$.
Moreover, since $\operatorname{supp}\nu = \operatorname{supp}\sigma-\operatorname{supp}\sigma$, $ \dim_{\mathcal{H}}(\operatorname{supp}\nu) \le \overline{\dim}_{\mathcal{M}}(\operatorname{supp}\nu) \le 2 \overline{\dim}_{\mathcal{M}}(\operatorname{supp}\sigma) $.
Thus, it follows that $\dim_{\mathcal{H}}(\operatorname{supp}\nu) \le t$.
On the other hand, the $t/2$-Fourier decay of $\nu$ gives $t\le \dim_{\mathcal{F}} \nu \le \dim_{\mathcal{H}}(\operatorname{supp}\nu)$.
Hence, we conclude
\[
\dim_{\mathcal{H}}(\operatorname{supp}\nu) = t.
\]

We obtain the desired measure $\mu$ by modifying $\nu$.
Since \(\widehat{\nu}\) is continuous and \(\widehat{\nu}(0) = 1\), there exists \(c > 0\) such that $\widehat{\nu}(\xi) \ge 1/2$ whenever $|\xi| \le c$.
For each \(n \ge 1\), set
\[
r_n = \frac{c}{2^{n + 1}},
\qquad
w_n = 2^{-nt/2} \psi(2^n),
\qquad
W = \sum_{n=1}^{\infty} w_n.
\]
Let \(\nu_n\) be the pushforward of \(\nu\) under the dilation \(x \mapsto r_n x\), so $\widehat{\nu_n}(\xi) = \widehat{\nu}(r_n \xi)$.
We consider
\[
\mu = \frac{1}{W}\sum_{n=1}^{\infty} w_n \nu_n.
\]
Then \(\mu \in \mathcal{P}(\mathbb{R}^d)\).
Since \(\nu_n\) is the pushforward of \(\nu\) under the dilation \(x\mapsto r_nx\), $\dim_{\mathcal H}(\operatorname{supp}\nu_n) = \dim_{\mathcal H}(\operatorname{supp}\nu) = t$. As in the proof of Theorem~\ref{thm:res-geo-i} in Section~\ref{sec6}, we have
\[
\dim_{\mathcal H}(\operatorname{supp}\mu)=t.
\]
Indeed,  \(\operatorname{supp}\nu_n\subset B(0,Cr_n)\) for some \(C>0\) and \(r_n\to0\) imply that any accumulation point of \(\bigcup_n\operatorname{supp}\nu_n\) outside that union must be \(0\). Hence, as all \(w_n>0\),
\[
\operatorname{supp}\mu
=
\overline{\bigcup_{n\ge1}\operatorname{supp}\nu_n}
=
\{0\}\cup\bigcup_{n\ge1}\operatorname{supp}\nu_n.
\]
Therefore, by countable stability of Hausdorff dimension,  we obtain the desired conclusion. 

Since $\psi$ is decreasing and $ \int_1^\infty \psi(R) \frac{dR}{R} < \infty$, we have $\sum_{n=1}^\infty \psi(2^n) < \infty$.
Hence, using \eqref{f-decay},
\[
|\widehat{\mu}(\xi)|
\lesssim
\sum_{n=1}^{\infty} 2^{-nt/2} \psi(2^n) (2^{-n} |\xi|)^{-t/2}
=
|\xi|^{-t/2} \sum_{n=1}^{\infty} \psi(2^n)
\lesssim
|\xi|^{-t/2}.
\]

Fix \(\xi \in \mathbb{R}^d\) with \(|\xi|\ge2\), and choose \(n\in\mathbb N\) such that \(2^n \le |\xi| \le 2^{n + 1}\).
Then $|r_n \xi| \le c$, and so \(\widehat{\nu_n}(\xi) \ge 1/2\).
Since \(\widehat{\nu_m} \ge 0\) for every \(m\), we have
\[
\widehat{\mu}(\xi)
=
\frac{1}{W}\sum_{m=1}^{\infty} w_m \widehat{\nu_m}(\xi)
\ge
\frac{w_n}{W}\widehat{\nu_n}(\xi)
\gtrsim
2^{-nt/2} \psi(2^n).
\]
Thus $|\widehat{\mu}(\xi)| \gtrsim |\xi|^{-t/2} \psi(|\xi|)$ whenever $2^n \le |\xi| \le 2^{n + 1}$. 
For \(|\xi|\le2\), we have \(|r_n\xi|\le c\) for every \(n\), so \(\widehat{\nu_n}(\xi)\ge1/2\) and hence \(\widehat{\mu}(\xi)\ge1/2\). Together with \(|\widehat{\mu}(\xi)|\le1\), this proves the stated bounds also for \(|\xi|\le2\).
\end{proof}

\subsection{Iterated convolutions}
\label{sec:iter-conv}

We first provide the proof of Proposition \ref{prop:m-fold-L2}.

\begin{proof}[Proof of Proposition \ref{prop:m-fold-L2}]
For the nongeometric case $\alpha < \beta$, the condition becomes $m \beta > d$.
Thus the conclusion is immediate from Plancherel's theorem.

On the other hand, suppose the geometric case $\alpha \ge \beta$.
Then we have $\alpha + (m - 1) \beta > d$.
Choose \(s\) such that $\max\{0,d-(m-1)\beta\} < s < \alpha$.
Since \(\mu\) is \(\alpha\)-Frostman, it has finite \(s\)-energy, and hence
\[
\int_{\mathbb{R}^d}
|\widehat{\mu}(\xi)|^2 (1+|\xi|)^{s-d}\, d\xi
< \infty.
\]
On the other hand, since \(\mu \in \mathcal{P}_{\alpha,\beta}(\mathbb{R}^d)\), we have
\[
|\widehat{\mu}(\xi)|^{2m}
=
|\widehat{\mu}(\xi)|^2 |\widehat{\mu}(\xi)|^{2m-2}
\lesssim
|\widehat{\mu}(\xi)|^2 (1+|\xi|)^{-(m-1)\beta}.
\]
We also have $(1+|\xi|)^{-(m-1)\beta} \lesssim (1+|\xi|)^{s-d}$, since \(s > d-(m-1)\beta\).
Therefore, it follows that $\int_{\mathbb{R}^d} |\widehat{\mu}(\xi)|^{2m}\, d\xi < \infty$.
By Plancherel's theorem, $\mu^{*m} \in L^2(\mathbb{R}^d)$.
\end{proof}

We now show that the threshold $\max\{\alpha, \beta\} + (m-1)\beta > d$ is sharp, up to the endpoint, by using Proposition~\ref{prop:conv-geo-factorization-i}.

\begin{proof}[Proof of Corollary \ref{cor:m-fold-singular}]
Since $\mathcal P_{\beta,\beta}\subset \mathcal P_{\alpha,\beta}$ for $\beta\ge \alpha$, the nongeometric case $\alpha < \beta$ follows from the geometric case $\alpha \ge \beta$.
Thus, we may assume $\alpha \ge \beta$.

Let $\mu\in \mathcal P_{\alpha, \beta}$ be the measure given in Proposition~\ref{prop:conv-geo-factorization-i} such that \(\mu ={\aracc{\mu}} * {\rdacc{\mu}}\) for some near $(\alpha - \beta)$-AD regular measure ${{\aracc{\mu}}}$ and near $\beta$-AD regular measure ${{\rdacc{\mu}}}$.
We show $\mu^{\ast m}$ is a singular measure. It suffices to show that the support of $\mu^{\ast m}$ has Lebesgue measure zero.

Let us set
\[
E = \operatorname{supp}{\aracc{\mu}},
\qquad
F = \operatorname{supp}{\rdacc{\mu}}.
\]
For a set $E\subset \mathbb R^d$, we denote
\[
E^{(m)}= \{x_1+\dots+ x_m: x_j\in E, \quad j=1, \dots, m\}.
\]
Since $\mu^{*m} = ({\aracc{\mu}} * {\rdacc{\mu}})^{*m} = ({\aracc{\mu}})^{*m} * ({\rdacc{\mu}})^{*m}$, we have
\[
\supp \mu^{\ast m} \subset    E^{(m)}+ F^{(m)}.
\]
As a result, it suffices to show that the set $E^{(m)}+ F^{(m)}$ has Lebesgue measure zero.

Let \(\mathcal N_r(\cdot)\) denote the covering number at scale \(r>0\).
By the standard comparison argument, it follows that
\begin{equation}
\label{Nr}
\mathcal N_r(A) \sim_d r^{-d}|A_r|
\end{equation}
for any bounded set
\(A \subset \mathbb{R}^d\).
Since \({\aracc{\mu}}\) is near \((\alpha-\beta)\)-AD regular, by Lemma~\ref{lem:element} we have $|E_{\delta}| \gtrsim \delta^{\,d-(\alpha - \beta)}$ and $|E_{\delta}| \lesssim_{\varepsilon} \delta^{\,d-(\alpha - \beta)-\varepsilon}$ for every \(\varepsilon > 0\).
From \eqref{eq:small-difference-set-i}, we also have $|(E-E)_{2\delta}| \lesssim_{\varepsilon} \delta^{\,d-(\alpha - \beta)-\varepsilon}$.
Thus, by \eqref{Nr}   we obtain
\[
\mathcal N_{\delta}(E-E)
\lesssim_{\varepsilon}
\delta^{-2\varepsilon} \mathcal N_{\delta}(E).
\]

We now use a covering-number version of the Plünnecke--Ruzsa inequality:
if \(\mathcal N_\delta(A+B)\le K\,\mathcal N_\delta(A)\), then
\[ \mathcal N_\delta(B^{(m)}-B^{(n)})\lesssim_{m,n,d} K^{m+n}\mathcal N_\delta(A)\]  (see, for example,  \cite[Lemma~22]{He}).
Fix \(\varepsilon>0\) and set \(\varepsilon_0=\varepsilon/(2m+1)\). Applying the preceding estimate with \(\varepsilon_0\), and then the Plünnecke--Ruzsa inequality with \(A=-E\) and \(B=E\), gives
\[
\mathcal N_{\delta}(E^{(m)})
\lesssim_{\varepsilon,m}
\delta^{-2m\varepsilon_0}\mathcal N_{\delta}(E).
\]
Therefore, by \eqref{Nr} and the preceding upper bound for \(|E_\delta|\),
\begin{equation}
\label{Em}
|(E^{(m)})_{\delta}|
\lesssim_{\varepsilon,m}
\delta^{\,d-(\alpha-\beta)-(2m+1)\varepsilon_0}
=
\delta^{\,d-(\alpha-\beta)-\varepsilon}.
\end{equation}

Next, since \({\rdacc{\mu}}\) is near \(\beta\)-AD regular, Lemma~\ref{lem:element} implies that \(F\) can be covered by at most 
$
C_{\varepsilon,m} \delta^{-\beta-\varepsilon/m}
$
balls of radius \(\delta/m\). Taking \(m\)-fold sums of their centers, we obtain
$\displaystyle 
F^{(m)}\subset\bigcup_{j=1}^{J}B(y_j,\delta)$ with 
$
J\lesssim_{\varepsilon,m}\delta^{-m\beta-\varepsilon}.
$
Thus, 
\[
E^{(m)} + F^{(m)}
\subset \bigcup_{j=1}^{J} \bigl(y_j + ( E^{(m)} )_{\delta}\bigr).
\]
Consequently, using \eqref{Em}, we obtain
\[
|E^{(m)} + F^{(m)}|
\le
\sum_{j=1}^{J} |( E^{(m)})_{\delta}|
\lesssim_{\varepsilon,m}
\delta^{\,d-\alpha-(m-1)\beta-2\varepsilon}.
\]
Since \(\alpha+(m-1)\beta<d\) by hypothesis, choose \(\varepsilon>0\) such that \(\alpha+(m-1)\beta+2\varepsilon<d\). 
Then the right-hand side tends to \(0\) as \(\delta \to 0\), and hence
$|E^{(m)} + F^{(m)}| = 0$ as desired.
\end{proof}
 \section*{Acknowledgement} 
This work was supported by the National Research Foundation of Korea (RS-2024-00342160; Sanghyuk Lee and Sungchul Lee).

 \bibliographystyle{plain}

\begin{thebibliography}{99}

\bibitem{BakSeeger}
J.-G. Bak and A. Seeger,
\newblock Extensions of the {Stein--Tomas} theorem,
\newblock {\em Math. Res. Lett.} 18 (2011), no.~4, 767--781.

\bibitem{Bluhm}
C. Bluhm,
\newblock Random recursive construction of {Salem} sets,
\newblock {\em Ark. Mat.} 34 (1996), no.~1, 51--63.

\bibitem{BLM}
S. Boucheron, G. Lugosi, and P. Massart,
\newblock {\em Concentration Inequalities: A Nonasymptotic Theory of Independence},
\newblock Oxford Univ. Press, Oxford, 2013.

\bibitem{CFdO24}
M. Carnovale, J. M. Fraser, and A. E. de Orellana,
\newblock Obtaining the {Fourier} spectrum via {Fourier} coefficients,
\newblock {\em Proc. Amer. Math. Soc.} (2026),
\newblock Published electronically March 10, DOI: 10.1090/proc/17610.

\bibitem{Chen14}
X. Chen,
\newblock A {Fourier} restriction theorem based on convolution powers,
\newblock {\em Proc. Amer. Math. Soc.} 142 (2014), no.~11, 3897--3901.

\bibitem{Chen}
X. Chen,
\newblock Sets of {Salem} type and sharpness of the {$L^2$}-{Fourier} restriction theorem,
\newblock {\em Trans. Amer. Math. Soc.} 368 (2016), no.~3, 1959--1977.

\bibitem{CS}
X. Chen and A. Seeger,
\newblock Convolution Powers of {Salem} Measures With Applications,
\newblock {\em Canad. J. Math.} 69 (2017), no.~2, 284--320.



\bibitem{Christ85}
M. Christ,
\newblock A convolution inequality concerning {Cantor-Lebesgue} measures,
\newblock {\em Rev. Mat. Iberoam.} 1 (1985), no.~4, 79--83.

\bibitem{Fraser}
J. M. Fraser,
\newblock The {Fourier} spectrum and sumset type problems,
\newblock {\em Math. Ann.} 390 (2024), no.~3, 3891--3930.

\bibitem{FHR25-1}
R. Fraser, K. Hambrook, and D. Ryou,
\newblock Fourier restriction and well-approximable numbers,
\newblock {\em Math. Ann.} 391 (2025), no.~3, 4233--4269.

\bibitem{FHR25-2}
R. Fraser, K. Hambrook, and D. Ryou, 
\newblock Sharpness of the {Mockenhaupt--Mitsis--Bak--Seeger} {Fourier} restriction theorem in all dimensions,
\newblock arXiv:2505.19526, 2025.

\bibitem{Hambrook}
K. Hambrook,
\newblock {\em Restriction theorems and {Salem} sets},
\newblock Ph.D. thesis, University of British Columbia, 2015.

\bibitem{HL13}
K. Hambrook and  I. {\L}aba,
\newblock On the Sharpness of {Mockenhaupt}'s Restriction Theorem,
\newblock {\em Geom. Funct. Anal.} 23 (2013), no.~4, 1262--1277.



\bibitem{HL16}
K. Hambrook and  I. {\L}aba,
\newblock Sharpness of the {Mockenhaupt--Mitsis--Bak--Seeger} restriction theorem in higher dimensions,
\newblock {\em Bull. Lond. Math. Soc.} 48 (2016), no.~5, 757--770.

\bibitem{He}
W. He,
\newblock Orthogonal projections of discretized sets,
\newblock {\em J. Fractal Geom.} 7 (2020), no.~3, 271--317.

\bibitem{HW19}
J. Hickman and J. Wright,
\newblock An abstract $L^2$ Fourier restriction theorem,
\newblock {\em Math. Res. Lett.} 26 (2019), no.~1, 75--100.

\bibitem{KFK09}
B. A. Kpata, I. Fofana, and K. Koua,
\newblock Necessary condition for measures which are {$(L^q,L^p)$} multipliers,
\newblock {\em Ann. Math. Blaise Pascal} 16 (2009), no.~2, 339--353.

\bibitem{KL17}
A. K{\"a}enm{\"a}ki and J. Lehrb{\"a}ck, 
\newblock Measures with predetermined regularity and inhomogeneous self-similar sets,
\newblock {\em Ark. Mat.} 55 (2017), no.~1, 165--184.

\bibitem{KLV13}
A. K{\"a}enm{\"a}ki, J. Lehrb{\"a}ck, and M. Vuorinen,
\newblock Dimensions, Whitney covers, and tubular neighborhoods,
\newblock {\em Indiana Univ. Math. J.} 62 (2013), no.~6, 1861--1889.

\bibitem{LP}
I. {\L}aba and M. Pramanik,
\newblock Arithmetic progressions in sets of fractional dimension,
\newblock {\em Geom. Funct. Anal.} 19 (2009), no.~2, 429--456.

\bibitem{LL}
L. Li and B. Liu,
\newblock Dimension of diophantine approximation and applications,
\newblock arXiv:2409.12826, 2024.

\bibitem{Littman}
W. Littman,
\newblock {$L^p$--$L^q$}-estimates for singular integral operators arising from hyperbolic equations,
\newblock in {\em Partial differential equations} (Proc. Sympos. Pure Math., Vol.~XXIII, Univ. California, Berkeley, Calif., 1971),
\newblock Amer. Math. Soc., Providence, RI, 1973, pp.~479--481.

\bibitem{Mattila}
P. Mattila,
\newblock {\em Fourier Analysis and Hausdorff Dimension},
\newblock Cambridge Stud. Adv. Math., vol.~150, Cambridge Univ. Press, Cambridge, 2015.

\bibitem{MNT}
G. Migliorati, F. Nobile, and R. Tempone,
\newblock Convergence estimates in probability and in expectation for discrete least squares with noisy evaluations at random points,
\newblock {\em J. Multivariate Anal.} 142 (2015), 167--182.

\bibitem{Mitsis}
T. Mitsis,
\newblock A {Stein-Tomas} restriction theorem for general measures,
\newblock {\em Publ. Math. Debrecen} 60 (2002), no.~1--2, 89--99.

\bibitem{Mockenhaupt}
G. Mockenhaupt,
\newblock Salem sets and restriction properties of {Fourier} transforms,
\newblock {\em Geom. Funct. Anal.} 10 (2000), no.~6, 1579--1587.

\bibitem{Oberlin82}
D. Oberlin,
\newblock A convolution property of the {Cantor-Lebesgue} measure,
\newblock {\em Colloq. Math.} 47 (1982), no.~1, 113--117.

\bibitem{Oberlin}
D. Oberlin,
\newblock Convolution estimates for some measures on curves,
\newblock {\em Proc. Amer. Math. Soc.} 99 (1987), no.~1, 56--60.

\bibitem{Oberlin03}
D. Oberlin, 
\newblock Affine dimension: measuring the vestiges of curvature,
\newblock {\em Michigan Math. J.} 51 (2003), no.~1, 13--26.

\bibitem{Schrijver}
A. Schrijver,
\newblock {\em A Course in Combinatorial Optimization},
\newblock Lecture notes, CWI and University of Amsterdam, Amsterdam, 2017, pp.~134--144.

\bibitem{SS17}
P. Shmerkin and V. Suomala,
\newblock A class of random {Cantor} measures, with applications,
\newblock in {\em Recent Developments in Fractals and Related Fields},
\newblock Trends Math., Birkh\"auser/Springer, Cham, 2017, pp.~233--260.

\bibitem{SS}
P. Shmerkin and V. Suomala, 
\newblock Spatially independent martingales, intersections, and applications,
\newblock {\em Mem. Amer. Math. Soc.} 251 (2018), no.~1195.

\bibitem{Strichartz}
R. Strichartz,
\newblock Convolutions with kernels having singularities on a sphere,
\newblock {\em Trans. Amer. Math. Soc.} 148 (1970), 461--471.

\bibitem{Wolff}
T. H. Wolff,
\newblock {\em Lectures on Harmonic Analysis},
\newblock University Lecture Series, vol.~29, American Mathematical Society, Providence, RI, 2003.

\end{thebibliography}

 \end{document}